\documentclass{amsart}
\usepackage[utf8]{inputenc}
\usepackage{t1enc}
\usepackage{amssymb}
\usepackage[polish, french, english]{babel}

\usepackage{graphicx}
\usepackage{sidecap}
\usepackage{relsize}
\usepackage{multirow}

\usepackage{makecell}
\usepackage[titletoc]{appendix}

\usepackage{caption}
\usepackage{subcaption}
\usepackage{amsfonts,amsmath,amsthm}
\usepackage{float}
\usepackage{cleveref}
\crefname{de}{Definition}{Definitions}
\crefname{prop}{Proposition}{Propositions}
\crefname{lem}{Lemma}{Lemmas}
\crefname{rem}{Remark}{Remarks}
\crefname{ex}{Example}{Examples}
\crefname{tw}{Theorem}{Theorems}
\crefname{cor}{Corollary}{Corollaries}
\crefname{con}{Conjecture}{Conjectures}
\crefname{equation}{}{}
\crefname{figure}{Figure}{Figures}
\crefname{section}{Section}{Sections}
\crefname{ob}{Observation}{Observations}

\usepackage{lmodern}
\usepackage{xcolor}
\usepackage{titlesec}
\usepackage{enumerate}
\usepackage{tikz}
\usepackage{pgfplots}
\pgfplotsset{compat=1.11}
\usepgfplotslibrary{fillbetween}
\usetikzlibrary{intersections}
\pgfdeclarelayer{bg}
\pgfsetlayers{bg,main}

\usetikzlibrary{patterns}
\usetikzlibrary{decorations.text}
\usetikzlibrary{decorations.markings}
\usepackage[object=vectorian]{pgfornament}
\usetikzlibrary{calc}
\usepackage{relsize}
\usetikzlibrary{arrows.meta}
\usetikzlibrary{bending}

\tikzset{fontscale/.style = {font=\relsize{#1}}}

\title{THE TOP-DEGREE PART IN THE MATCHINGS-JACK CONJECTURE}
\author{Adam Burchardt }
\thanks{Research supported by Narodowe Centrum Nauki, grant number
2014/15/B/ST1/00064}

\titleformat{\section}[block]{\Large\bfseries\filcenter}{\thesection .}{0.5em}{}
\titleformat{\subsection}[runin]{\bfseries}{\thesubsection .}{0.5em}{}
\DeclareMathOperator{\Ch}{Ch}
\DeclareMathOperator{\wt}{wt}

\DeclareMathOperator{\E}{\mathcal{E}}

\DeclareMathOperator{\topp}{top}
\newcommand{\stat}{\text{stat}_\eta }

\DeclareMathOperator{\Q}{\mathbb{Q}}

\newcommand{\Laurent}{\mathbb{Q}\left[ A,A^{-1}\right]}

\newcommand{\Laurentdd}{\mathbb{Q}\left[ A,A^{-1}\right]_{|\pi | -\ell (\pi )}}

\DeclareMathOperator{\V}{\mathcal{V}}

\usepackage{mathtools}
\DeclarePairedDelimiter{\Kumulant}{\left(\right.}{\left.\right)}

\let\oldk\k
\let\oldc\c
\renewcommand{\c}{\kappa_\nu\Kumulant}
\renewcommand{\k}{\kappa\Kumulant}

\newcommand\A[1]{\left[ A^{#1} \right] }
\newcommand\Sym[1]{\mathfrak{S}\left( {#1} \right) }

\newcommand{\g}[1]{g_{\pi,\sigma}^{#1}}

\newtheorem{tw}{Theorem}[section]
\newtheorem{prop}[tw]{Proposition}
\newtheorem{cor}[tw]{Corollary}
\newtheorem{lem}[tw]{Lemma}
\newtheorem{con}[tw]{Conjecture}
\newtheorem{ob}[tw]{Observation}

\theoremstyle{remark}
\newtheorem{ex}[tw]{Example}
\newtheorem{rem}[tw]{Remark}
\newtheorem{de}[tw]{Definition}

\newcommand{\Gl}{\mathcal{G}^{\lambda}_{\pi,\sigma}}
\newcommand{\Glp}{\mathcal{G}^{\lambda ; \mu}_{\pi,\sigma}}
\newcommand{\Gll}{\mathcal{G}^{\lambda ; \lambda}_{\pi,\sigma}}
\newcommand{\Mll}{M^{\lambda ; \lambda}_{\pi,\sigma}}
\newcommand{\GMlp}{M \left( \mathcal{G}^{\lambda ; \mu}_{\pi,\sigma} \right) }
\newcommand{\ie}{\emph{i.e.}\ }
\newcommand{\Ml}{M^\lambda_{\pi,\sigma}}
\newcommand{\M}{M^{\bullet}_{\pi,\sigma}}
\newcommand{\Mp}{M^{\bullet;\mu}_{\pi,\sigma}}
\newcommand{\Mlp}{M^{\lambda;\mu}_{\pi,\sigma}}
\newcommand{\MOlp}{{\widetilde{M}}^{\lambda; \mu}_{\pi,\sigma}}
\newcommand{\GOlp}{\mathcal{\widetilde{G}}^{\lambda ; \mu}_{\pi,\sigma}}
\newcommand{\MOp}{{\widetilde{M}}^{\bullet; \mu}_{\pi,\sigma}}
\newcommand{\MOpj}{{\widetilde{M}}^{\bullet; \mu}_{\pi \cup 1^{|\mu | - |\pi|},\sigma \cup 1^{|\mu | - |\sigma|}}}
\newcommand{\MOpjj}{{\widetilde{M}}^{\bullet; \mu}_{\pi ,\sigma \cup 1}}
\newcommand{\MOpl}{{\widetilde{M}}^{\bullet; \lambda}_{\pi,\sigma}}

\makeatletter
\newcommand{\subalign}[1]{
  \vcenter{
    \Let@ \restore@math@cr \default@tag
    \baselineskip\fontdimen10 \scriptfont\tw@
    \advance\baselineskip\fontdimen12 \scriptfont\tw@
    \lineskip\thr@@\fontdimen8 \scriptfont\thr@@
    \lineskiplimit\lineskip
    \ialign{\hfil$\m@th\scriptstyle##$&$\m@th\scriptstyle{}##$\crcr
      #1\crcr
    }
  }
}
\makeatother

\begin{document}
\begin{abstract}
In 1996 Goulden and Jackson introduced a family of coefficients $( c_{\pi, \sigma}^{\lambda} ) $ indexed by triples of partitions which arise in the power sum expansion
of some Cauchy sum for Jack symmetric functions $( J^{(\alpha )}_\pi )$. 
The coefficients $ c_{\pi, \sigma}^{\lambda}  $ can be viewed as an interpolation between the structure constants of the class algebra and the double coset algebra. Goulden and Jackson suggested that the coefficients $ c_{\pi, \sigma}^{\lambda}  $ are polynomials in the variable $\beta := \alpha-1$ with non-negative integer coefficients and that there is a combinatorics of matching hidden behind them. 
This \emph{Matchings-Jack Conjecture} remains open. 
Do\l{}\oldk{e}ga and F\'eray showed the polynomiality of connection coefficients $c^\lambda_{\pi,\sigma}$ and gave the upper bound on the degrees. 
We give a necessary and sufficient condition for the polynomial $ c_{\pi, \sigma}^{\lambda}$ to achieve this bound.  
We show that the leading coefficient of $ c_{\pi, \sigma}^{\lambda}$ is a positive integer and we present it in the context of Matchings-Jack Conjecture of Goulden and Jackson. 
\end{abstract}

\maketitle

\section{Introduction}
\label{section1}

\subsection{Jack polynomials.}
\label{sec:jack}
Jack polynomials $\left(  J^{(\alpha )}_\pi\right) $ are a family of symmetric functions that depend on a parameter $\alpha >0$ and is indexed by an integer partition $\pi$. They were introduced by Henry Jack in his seminal paper \cite{Jack1970/1971}. 
For certain values of $\alpha$, Jack polynomials coincide with various well-known symmetric polynomials. 
For instance, up to multiplicative constants, Jack polynomials coincide with Schur polynomials for $\alpha =1$; 
with the zonal polynomials, for $\alpha = 2$; 
with the symplectic zonal polynomials, for $\alpha = 1/2$; 
with the elementary symmetric functions, for $\alpha = 0$; 
and in some sense with the monomial symmetric functions, for $\alpha = \infty$. 
Since it has been shown that several results concerning Schur and zonal polynomials can be generalized in a rather natural way to Jack polynomials \cite[Section (VI.10)]{MR3443860}, 
Jack polynomials can be viewed as a natural interpolation between several interesting families of symmetric functions.

Connections of Jack polynomials with various fields of mathematics and physics were established:
it turned out that they play a crucial role in understanding Ewens random permutations model \cite{MR1199122}, 
generalized $\beta$-ensembles and some statistical mechanics models 
\cite{MR1432811},
Selberg-type integrals \cite{MR1226865},
certain random partition models 
\cite{MR1756734},
and some problems of the algebraic geometry \cite{Nak96,Oko03}, 
among many others. 
Better understanding of Jack polynomials is also very desirable in the context of generalized 
$\beta$-ensembles and their discrete counterpart model \cite{MR1432811}. 
Jack polynomials are a special case of the \emph{Macdonald polynomials} \cite{Sta89,MR3443860}.

\subsection{Connection coefficients for Jack symmetric functions.}
\label{Intro2}
Goulden and Jackson \cite{MR1325917} defined two families of coefficients $\left( c_{\pi, \sigma}^{\lambda} \right) $ and $\left( h_{\pi, \sigma}^{\lambda} \right)$ depending implicitly on the deformation parameter~$\alpha$ and indexed by triples of integer partitions $\pi ,\sigma,\lambda \vdash n$ of the same integer $n$. 
These coefficients are given by expansions of the left-hand sides in terms of 
the power-sum symmetric functions:
\begin{multline}
\label{star} 
\sum_{\theta\in\mathcal{P}} \frac{1}{\langle J_\theta,J_\theta \rangle_\alpha} 
J^{(\alpha)}_\theta(\mathbf{x})
J^{(\alpha)}_\theta(\mathbf{y})
J^{(\alpha)}_\theta(\mathbf{z})
t^{|\theta|}=\\
   \sum_{n\geq 1} t^n \sum_{\lambda,\pi,\sigma\vdash n}
\frac{c_{\pi, \sigma}^{\lambda}}{\alpha^{\ell(\lambda)}}
z_\lambda^{-1}
p_\pi(\mathbf{x})
p_\sigma(\mathbf{y})
p_\lambda(\mathbf{z}),
\end{multline}
and
\begin{multline}
\label{star2} 
\alpha t\dfrac{\partial}{\partial t} \log \left( \sum_{\theta\in\mathcal{P}} \frac{1}{\langle J_\theta,J_\theta \rangle_\alpha} 
J^{(\alpha)}_\theta(\mathbf{x})
J^{(\alpha)}_\theta(\mathbf{y})
J^{(\alpha)}_\theta(\mathbf{z})
t^{|\theta|}\right) = \\
   \sum_{n\geq 1} t^n \sum_{\lambda,\pi,\sigma\vdash n}
h_{\pi, \sigma}^{\lambda}
p_\pi(\mathbf{x})
p_\sigma(\mathbf{y})
p_\lambda(\mathbf{z}),
\end{multline}
see \cite[Equations (1),(5) and Equations (2),(4)]{MR1325917}.

Do\l{}\oldk{e}ga and F\'eray showed that the connection coefficients $(c^\lambda_{\pi,\sigma})$ are polynomials in the variable $\beta: =\alpha-1$ with rational coefficients and proved the following upper bound on the degrees of these polynomials \cite[Proposition B.2.]{Dolega2014}:
\begin{equation}
\label{Max}
\deg_\beta c^\lambda_{\pi,\sigma} \leq d \left( \pi,\sigma ;\lambda \right) ,
\end{equation} 
\noindent
where 
\[ d \left( \pi,\sigma ;\lambda \right) := \Big( |\pi | - \ell (\pi ) \Big) +\Big( |\sigma | - \ell (\sigma ) \Big) -
\Big( |\lambda| - \ell (\lambda ) \Big) .
\]
\noindent
One may wonder of the use of the new variable $\beta$, but this shift seems to be the adequate one in order to look at the connection coefficients from the combinatorial point of view.

\subsection{Matchings.}
\label{prob}

We present the well established terminology of matching given in \cite{MR1325917}. 
For a given integer $n$ we consider the following set
\begin{equation*}
\mathcal{N}_n =\left\lbrace 1, \hat{1}, \ldots , n, \hat{n} \right\rbrace .
\end{equation*}
\noindent
We denote by $\mathcal{F}_n$ the set of all matchings (partitions on two-elements sets) on $\mathcal{N}_n$. 
For matchings $\delta_1,\delta_2,\ldots\in  \mathcal{F}_n$ we denote by $G (\delta_1,\delta_2, \ldots) $ the multi-graph with the vertex set $\mathcal{N}_n$ whose edges are formed by the pairs in $\delta_1,\delta_2 ,\ldots$. 
For given matchings $\delta_1,\delta_2$ the corresponding graph $G (\delta_1,\delta_2) $ consists of disjoint even cycles, since each vertex has degree $2$ and around each cycle the edges alternate between $\delta_1$ and $\delta_2$. 
Denote by $\Lambda (\delta_1,\delta_2)$ the partition of $n$ which specifies halves the lengths of the cycles in $G (\delta_1,\delta_2) $. 
More generally, denote by $\Lambda (\delta_1,\ldots, \delta_s)$ the partition of $n$ which specifies halves of the number of vertices in each connected component of $G (\delta_1,\delta_2,\ldots) $ (it is an easy observation that such numbers form a partition of $n$).

We call the sets $\{ 1, \ldots , n \}$ and $\{\hat{1}, \ldots , \hat{n}\}$ \emph{classes} of $\mathcal{N}_n $. 
A pair in a matching is called a \emph{between-class} pair if it contains elements of different classes. 
A matching $\delta$ in which every pair is a between-class pair is called a \emph{bipartite} matching (in this case $G (\delta )$ is a bipartite graph on the vertex-sets given by the two classes of $\mathcal{N}_n $).

We introduce two specific bipartite matchings in the set $\mathcal{F}_n$. First, let
\begin{equation*}
\epsilon :=\left\lbrace \{1,\hat{1} \},\ldots , \{ n,\hat{n}\}\right\rbrace ;
\end{equation*}
\noindent
second, for a given partition $\mu\vdash n$, let 
\begin{multline*}
\delta_\lambda := \left\lbrace 
\{1,\hat{2} \}, \{2,\hat{3} \}, \ldots ,\{\lambda_1 -1, \hat{\lambda_1} \}, \{\lambda_1 ,\hat{1} \}, \right. \\
\left. \{\lambda_1+1,\hat{\lambda_1+2} \}, \ldots ,\{\lambda_1+\lambda_1 -1, \hat{\lambda_1+\lambda_1} \}, \{\lambda_1+\lambda_1 ,\hat{\lambda_1+1} \},
\ldots 
\right\rbrace ,
\end{multline*}
\noindent
see \cref{matching}. Observe that both matchings: $\epsilon$ and $\delta_\lambda$ are bipartite and $\Lambda ( \epsilon, \delta_\lambda ) =\lambda$.

\begin{figure}
\centering
\begin{tikzpicture}[scale=0.3]
\draw[gray,dashed] (3,10)  -- (7,10) node [midway, below, sloped] (TextNode) {};
\draw[gray,dashed] (3,0)  -- (7,0) node [midway, below, sloped] (TextNode) {};
\draw[gray,dashed] (0,3)  -- (0,7) node [midway, below, sloped] (TextNode) {};
\draw[gray,dashed] (10,3)  -- (10,7) node [midway, below, sloped] (TextNode) {};

\draw[black] (3,10)  -- (0,7) node [midway, below, sloped] (TextNode) {};
\draw[black] (10,3)  -- (7,0) node [midway, below, sloped] (TextNode) {};
\draw[black] (0,3)  -- (3,0) node [midway, below, sloped] (TextNode) {};
\draw[black] (10,7)  -- (7,10) node [midway, below, sloped] (TextNode) {};

\filldraw[black] (3,10) circle (4pt) node[anchor=south] {$\hat{1}$};
\filldraw[black] (10,7) circle (4pt) node[anchor=west] {$\hat{2}$};
\filldraw[black] (7,0) circle (4pt) node[anchor=north] {$\hat{3}$};
\filldraw[black] (0,3) circle (4pt) node[anchor=east] {$\hat{4}$};

\fill[white] (7,10) circle (4pt) node[anchor=east] {};
\draw[black] (7,10) circle (4pt) node[anchor=south] {$1$};
\fill[white] (10,3) circle (4pt) node[anchor=east] {};
\draw[black] (10,3) circle (4pt) node[anchor=west] {$2$};

\fill[white] (3,3) circle (4pt) node[anchor=east] {};
\draw[black] (3,0) circle (4pt) node[anchor=north] {$3$};
\fill[white] (0,7) circle (4pt) node[anchor=east] {};
\draw[black] (0,7) circle (4pt) node[anchor=east] {$4$};

\draw[gray,dashed] (17,7)  -- (21,7) node [midway, below, sloped] (TextNode) {};
\draw[gray,dashed] (17,3)  -- (21,3) node [midway, below, sloped] (TextNode) {};

\draw[black] (17,3)  -- (17,7) node [midway, below, sloped] (TextNode) {};
\draw[black] (21,3)  -- (21,7) node [midway, below, sloped] (TextNode) {};

\filldraw[black] (17,7) circle (4pt) node[anchor=south] {$\hat{5}$};
\filldraw[black] (21,3) circle (4pt) node[anchor=north] {$\hat{6}$};

\fill[white] (21,7) circle (4pt) node[anchor=east] {};
\draw[black] (21,7) circle (4pt) node[anchor=south] {$5$};
\fill[white] (17,3) circle (4pt) node[anchor=east] {};
\draw[black] (17,3) circle (4pt) node[anchor=north] {$6$};

\draw[gray,dashed] (27,7)  -- (31,7) node [midway, below, sloped] (TextNode) {};
\draw[gray,dashed] (27,3)  -- (31,3) node [midway, below, sloped] (TextNode) {};

\draw[black] (27,3)  -- (27,7) node [midway, below, sloped] (TextNode) {};
\draw[black] (31,3)  -- (31,7) node [midway, below, sloped] (TextNode) {};

\filldraw[black] (27,7) circle (4pt) node[anchor=south] {$\hat{7}$};
\filldraw[black] (31,3) circle (4pt) node[anchor=north] {$\hat{8}$};

\fill[white] (31,7) circle (4pt) node[anchor=east] {};
\draw[black] (31,7) circle (4pt) node[anchor=south] {$7$};
\fill[white] (27,3) circle (4pt) node[anchor=east] {};
\draw[black] (27,3) circle (4pt) node[anchor=north] {$8$};

\draw[black] (15,-2) circle (0pt) node[anchor=north] {$\lambda =\left( 4,2,2 \right) $};

\draw[black] (15,-4.7) circle (0pt) node[anchor=north] {$\delta_\lambda := \Big\{ 
\underbrace{\{1,\hat{2} \}, \{2,\hat{3} \},\{3,\hat{4} \}, \{4,\hat{1} \}}_{\lambda_1 =4},
\underbrace{\{5,\hat{6} \}, \{6,\hat{5} \}}_{\lambda_2 =2},
\underbrace{\{7,\hat{8} \}, \{8,\hat{7} \}}_{\lambda_3 =2} \Big\} $};
\end{tikzpicture}
\caption{An example of matchings $\epsilon$ (dotted line) and $\delta_\lambda$ (continuous line) for $\lambda = (4,2,2) $. Observe that both matchings: $\epsilon$ and $\delta_\lambda$ are bipartite and $\Lambda ( \epsilon, \delta_\lambda ) =\lambda$.} 
\label{matching}
\end{figure}
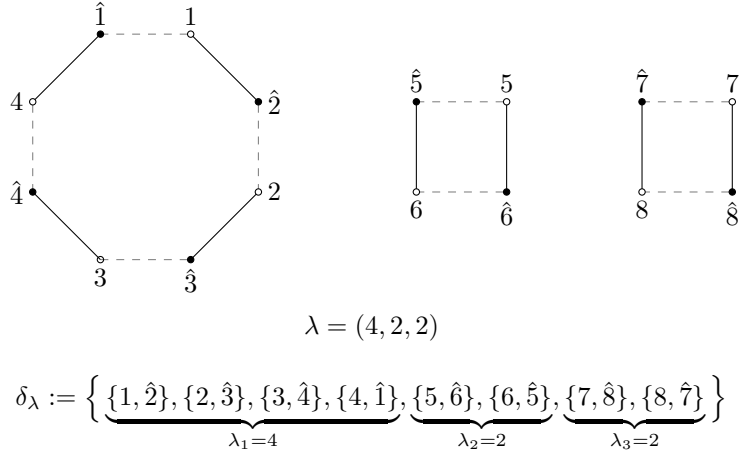

\subsection{Matchings-Jack Conjecture.}
\label{444}

\begin{de}
\label{g}
For given three partitions $\pi,\sigma , \lambda \vdash n$, we denote by $\Gl$ the set of all matchings $\delta \in \mathcal{F}_n$, for which $\Lambda (\delta,\epsilon ) =\pi$ and $
\Lambda (\delta,\delta_\lambda ) =\sigma$.
\end{de}

Goulden and Jackson observed that the specializations of $c_{\pi, \sigma}^{\lambda} (\beta )$ for $\beta \in \{0,1\}$ may be expressed in terms of matchings, namely
\begin{align*} 
c_{\pi, \sigma}^{\lambda} ( 0 ) &= \Big\vert \big\{ \delta \in \mathcal{G}_{\pi, \sigma}^{\lambda} : 
 \delta \text{ is bipartite} \big\} \Big\vert ,\\
c_{\pi, \sigma}^{\lambda} ( 1 )&=  \Big\vert  \big\{ \delta \in \mathcal{G}_{\pi, \sigma}^{\lambda} \big\} \Big\vert. 
\end{align*} 
\noindent
In fact, those specialisations coincide with the connection coefficients of two commutative subalgebras of the group algebra of the symmetric group: the \emph{class algebra} and the \emph{double coset algebra} ($\beta =0$ and $\beta =1$ respectively) \cite{Sta91}. 

Based on this observation Goulden and Jackson conjectured that the family $\left( c_{\pi, \sigma}^{\lambda} \right) $  of polynomials may have a combinatorial interpretation. 
The conjecture is known as the \emph{Matchings-Jack Conjecture}.

\begin{con}[Matchings-Jack Conjecture]
For any partitions $\pi,\sigma,\lambda\vdash n$ the quantity $c_{\pi, \sigma}^{\lambda}$ can be expressed as
\begin{equation*}
c_{\pi, \sigma}^{\lambda} (\beta ) =\sum_{\delta \in \mathcal{G}_{\pi, \sigma}^{\lambda}} \beta^{\wt_\lambda (\delta )},
\end{equation*}
\noindent
where $\wt_\lambda : \mathcal{G}_{\pi, \sigma}^{\lambda} \longrightarrow \mathbb{N}_0$ is some hypothetical combinatorial statistic, which vanishes if and only if $\delta $ is bipartite.
\end{con}

Clearly, it seems that the statistic $\wt_\lambda$ should be a marker of non-bipartiteness for matchings. 
Matchings-Jack Conjecture remains still open in the general case, however some special cases have been settled. 
Goulden and Jackson constructed some statistics $\wt_\lambda$ for $\lambda = [1^n]$ and $\lambda= [2, 1^{n-1}]$ and proved the conjecture in those cases \cite{MR1325917}. 
Later on, the Matchings-Jack Conjecture has been proved by Kanunnikov and Vassiliveva \cite{Vass16} in the case $\pi =\sigma =(n)$ of the partitions with exactly one part. 
Recently, in a joint paper with Promyslov \cite{Vass18}, they proved the conjecture in the case when one of the three partitions is equal to $(n)$. 
They made use of the \emph{measure of non-orientability} $\theta$ defined by La Croix in his PhD thesis \cite{Lacroix}. The measure of non-orientability $\theta$ is a statistic defined on a class of \emph{rooted maps}. 
In some special cases it may be translated into the field of matchings, however generally significant difficulties appear. 
We also shall use the same statistic.

\subsection{The main result.}
In the current paper we give a necessary and sufficient condition for the polynomial $c^\lambda_{\pi,\sigma}$ to achieve the maximal degree given by \cref{Max}. 
Moreover, we show that the leading coefficient of $c^\lambda_{\pi,\sigma}$ of this maximal degree is a non-negative integer and we present it in the context of Matchings-Jack Conjecture.

\begin{de}
\label{order}
Consider two integer partitions $\lambda$ and $\mu$ of the same integer $n$, let $k =\ell (\lambda)$ and $m=\ell (\mu )$ be the lengths of the partitions. 
We say that $\lambda$ is a \emph{subpartition} of $\mu$ (denoted $\lambda \preceq \mu$) if there exists a set-partition $\nu$ of $[k]$, such that
\begin{equation*}
\mu_i = \sum_{j \in \nu_i} \lambda_j
\end{equation*}
\noindent
for any $i\in [m]$, see \cref{fig00}. We denote $\lambda \prec\mu $ if $\lambda \preceq \mu$ and $\lambda \neq \mu$.
\end{de}

\begin{figure}
\centering
\begin{tikzpicture}[scale=0.6]
\draw[gray] (8,5) rectangle (9,6);

\filldraw[fill=blue!40!white, draw=black] (8,3) rectangle (9,4) ; 
\filldraw[fill=blue!40!white, draw=black] (9,3) rectangle (10,4);
\filldraw[fill=blue!40!white, draw=black] (10,3) rectangle (11,4);
\filldraw[fill=blue!40!white, draw=black] (11,3) rectangle (12,4) ;
\filldraw[fill=red!40!white, draw=black] (8,4) rectangle (9,5);
\filldraw[fill=red!40!white, draw=black] (9,4) rectangle (10,5) ;
\filldraw[fill=red!40!white, draw=black] (10,4) rectangle (11,5) ;
\filldraw[fill=blue!40!white, draw=black] (8,5) rectangle (9,6);
\filldraw[fill=green!40!white, draw=black] (8,6) rectangle (9,7);

\begin{scope}[shift={(9,0)}]
\filldraw[fill=blue!40!white, draw=black] (8,3) rectangle (9,4) ; 
\filldraw[fill=blue!40!white, draw=black] (9,3) rectangle (10,4);
\filldraw[fill=blue!40!white, draw=black] (10,3) rectangle (11,4);
\filldraw[fill=blue!40!white, draw=black] (11,3) rectangle (12,4) ;
\filldraw[fill=blue!40!white, draw=black] (12,3) rectangle (13,4) ;
\filldraw[fill=red!40!white, draw=black] (8,4) rectangle (9,5);
\filldraw[fill=red!40!white, draw=black] (9,4) rectangle (10,5) ;
\filldraw[fill=red!40!white, draw=black] (10,4) rectangle (11,5) ;
\filldraw[fill=green!40!white, draw=black] (8,5) rectangle (9,6);
\end{scope}

\draw[->,blue] (12.5,3.5) -- (16.5,3.5);
\draw[->,blue,rounded corners=8pt] (12.5,5.5)-- (13.5,5.5) -- (14.3,3.5) -- (16.5,3.5);
\draw[red] (12.5,4.5) -- (13.7,4.5);
\draw[->,red] (14.1,4.5) -- (16.5,4.5);
\draw[->,green,rounded corners=8pt] (12.5,6.5)-- (13.7,6.5) -- (14.6,5.5) -- (16.5,5.5);
\end{tikzpicture}
\caption{A pair of partitions $\lambda = (4,3,1,1)$ and $\mu = (5,3,1)$ presented as Young diagrams. Partition $\lambda$ is sub-partition of $\mu$; indeed, each part of $\mu$ is given as a sum of different parts of $\lambda$.} 
\label{fig00}
\end{figure}
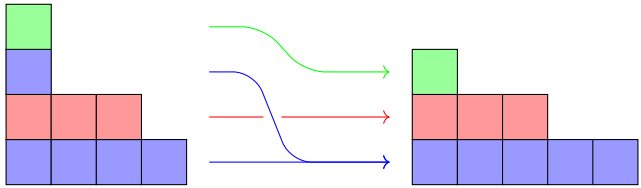

\begin{de}
For given partitions $\pi,\sigma , \lambda, \mu \vdash n$, we denote by $ \Glp$ the set of all matchings $\delta \in \Gl$ which are $\mu$-\emph{connected}, \ie $\Lambda(\delta,\epsilon,\delta_\lambda )=\mu$.
\end{de}

The class $\Gl$ splits naturally into the classes $\Glp$, namely
\begin{equation*}
\Gl =\bigsqcup_{\mu : \lambda \preceq \mu} \Glp .
\end{equation*}

Contrary to previous works on the Matchings-Jack Conjecture we do not attempt to define the statistic $\wt_\lambda$ on $\Gl$ for a particular class of partitions $\lambda$, $\pi$ or $\sigma$. 
We define the statistic "$\stat$" on the class $\Gll$.

\begin{tw}[The main result]
\label{C-main result}
For any triple of partitions $\pi,\sigma ,\lambda \vdash n$ the corresponding polynomial $c^\lambda_{\pi,\sigma}(\beta)$ achieves the upper bound on the degree given in \cref{Max} if and only if $\pi$ and $\sigma$ are sub-partitions of $\mu$. 
For such partitions, the leading coefficient of $c^\lambda_{\pi,\sigma}(\beta)$ may be expressed in two different manners:
\begin{align*} 
\left[ \beta^{d (\pi, \sigma ;\lambda )}\right] c^\lambda_{\pi ,\sigma}= 
\Big\vert \delta \in \Gll : \delta &\emph{ is unhandled }  \Big\vert = \\
&\sum_{\nu : \nu \preceq \lambda} \dfrac{z_\lambda}{z_\nu} \Big\vert \delta \in \mathcal{G}_{\pi, \sigma}^{\nu ;\lambda}  : \delta\emph{ is bipartite }\Big\vert ,
\end{align*} 
\noindent
for notion of unhandled matchings see \cref{unhandledM}.  
Moreover, there exists a statistic $\stat : \Gll \longrightarrow  \mathbb{N}_0$, which satisfies
\begin{equation*}
\left[ \beta^{d (\pi, \sigma ;\lambda )}\right] c^\lambda_{\pi ,\sigma}= 
\left[ \beta^{d (\pi, \sigma ;\lambda )}\right] \sum_{\delta \in \Gll} \beta^{\stat (\delta)}
\end{equation*}
\noindent
and for $\delta \in \Gll$ the statistic $\stat (\delta )$ vanishes if and only if $\delta$ is bipartite.
\end{tw}

\subsection{Organisation of the paper.}
In \cref{3333} we introduce the terminology of maps and we investigate relations between maps and matchings. 
In \cref{4444} we present a measure of non-orientability in the context of $b$-Conjecture. 
We present a problem of transferring it into the satisfactory statistic which measures non-bipartitness of a matching. 
We discuss the recent result of Do\l{}\oldk{e}ga \cite{Dol17} about the top-degree part in $b$-Conjecture. 
In fact, our result and the result of Do\l{}\oldk{e}ga are equivalent, see \cref{appendix}, however the methods used in both papers are different. 
\cref{5555} presents the notions of Jack characters and their structure constants. 
We show the relation between the structure constants for Jack characters and the connection coefficients for Jack symmetric functions. 
We give a formula for top-degree part of structure constants for Jack characters and translate this result into the field of connection coefficients. We prove this formula in \cref{6666}.

\section{Preliminaries}
\label{p}

\subsection{Partitions.}
\label{p1}
A \emph{partition} $\lambda$ of $n$ (denoted by $\lambda\vdash n$) is a non-increasing list 
$\left( \lambda_1, \ldots ,\lambda_l\right) $ of positive integers of sum equal to $n$. 
Number $n$ is called the \emph{size} of $\lambda$ and is denoted by $| \lambda |$, the number $l$ is the \emph{length} of the partition, denoted by $\ell \left( \lambda \right)$. Finally,
\begin{equation*}
m_i(\lambda):=\big|\{k: \lambda_k=i\}\big| ,
\end{equation*}
\noindent
is the \emph{multiplicity} of $i\geq 1$ in the partition $\lambda$.

There are many orders on the set of partitions. Beside the one shown in \cref{order} we introduce the \emph{dominance order}. We say that $\lambda \leq \mu$ if and only if 
\begin{equation*}
\sum_{i\leq j} \lambda_i \leq \sum_{i\leq j} \mu_i
\end{equation*}
\noindent
holds for any positive integer $j$.

For given two partitions $\lambda $ and $\mu $ we construct their \emph{concatenation} (denoted $\lambda \cup \mu$) by merging all parts from $\lambda$ and $\mu$ and ordering them in a decreasing fashion. 

\subsection{Jack polynomials.}
\label{p2}
Let us consider the vector space $\Lambda_{\mathbb{Q}(\alpha )}$ of the \emph{symmetric functions} \cite{SymFunct} over the field of rational functions $\mathbb{Q}(\alpha )$ and its basis $\left( p_\lambda \right)_\lambda $ of \emph{power-sum symmetric functions}, \ie the symmetric functions given by
\begin{equation*}
p_\lambda(\mathbf{x}) =\prod_i p_{\lambda_i}(\mathbf{x}), \quad\quad p_{k}(\mathbf{x}) =x_1^k +x_2^k+\cdots .
\end{equation*}

The following scalar product on $\Lambda_{\mathbb{Q}(\alpha )}$ is defined on the power-sum 
basis by the formula
\begin{equation*}
\langle p_\lambda,p_\mu \rangle_\alpha :=  \alpha^{\ell (\lambda )} z_\lambda \delta_{\lambda ,\mu} ,
\end{equation*}
\noindent
where
\begin{equation*}
z_\lambda=\prod_i i^{m_i(\lambda)}\ m_i(\lambda)! 
\end{equation*}
\noindent
and further extended by bilinearity. This is a classical deformation of the \emph{Hall inner product}, which corresponds to $\alpha$ =1 \cite{Jack1970/1971}.

Jack polynomials are the only family of symmetric functions $\left(  J^{(\alpha )}_\pi\right) $ which satisfies the following three criteria:
\begin{enumerate}
\item 
$  J^{(\alpha )}_\lambda =\sum_{\mu \leq \lambda} a_\mu^\lambda m_\mu$, where $a_\mu^\lambda \in \mathbb{Q}[\alpha ]$,
\item 
$\left[ m_{\text{1}^{|\lambda |}} \right] J^{(\alpha )}_\lambda :=a_{\text{1}^{|\lambda |}}^\lambda =|\lambda | ! $,
\item 
$\langle J^{(\alpha )}_\lambda,J^{(\alpha )}_\mu \rangle_\alpha =0$ for $\lambda \neq \mu$,
\end{enumerate}
\noindent
where $ m_\lambda$ denotes the monomial symmetric function associated with $\lambda$.

\subsection{The deformation parameters.}
\label{sec:fast-forward}

In order to avoid dealing with the square root of the variable $\alpha$
we introduce an indeterminate $A$ such that $ A^2 := \alpha$. 
Jack characters are usually defined in terms of the deformation parameter $\alpha$.
After the substitution $ \alpha:= A^2 $, each Jack character becomes a function of $A$. 
In order to keep the notation light, we will make this dependence implicit and we will simply write $\Ch_{\pi}(\lambda)$.

The algebra of Laurent polynomials in the indeterminate $A$ will be denoted by $\Laurent$. 
For an integer $d$ we will say that a Laurent polynomial
\begin{equation*}
f=\sum_{k\in\mathbb{Z}} f_k A^k\in \Laurent 
\end{equation*}
\noindent
is of \emph{degree at most $d$} if $f_k=0$ holds for each integer $k>d$.

A special role will be played by the quantity
\begin{equation*}
\label{eq:gamma}
 \gamma := -A +\frac{1}{A} \in \Laurent 
\end{equation*}
\noindent
and its opposite
\begin{equation*}
\label{eq:delta}
 \delta := A -\frac{1}{A} \in \Laurent .
\end{equation*}

\section{Matchings and maps}
\label{3333}

\subsection{Maps.}
\label{sec:labeled-maps}

In the literature a \emph{map} \cite{MR2036721} is classically defined as a connected graph $G$ (possibly, with multiple edges) drawn on a surface $\Sigma$, \ie a compact connected $2$-dimensional manifold without boundary. 
We assume that a collection of \emph{faces} (\ie $\Sigma\setminus \E$) is homeomorphic to a collection of open discs. 
A choice of an edge-side and one of its endpoints is called a \emph{root} of the map, see \cref{New2a}. 
A map together with a choice of a root is called a \emph{rooted map}.

A vertex two-coloured map is called \emph{bipartite} if each edge connects vertices of different colors; for simplicity we set that there are \emph{white} and \emph{black} vertices, we denote by $\mathcal{W}$ ($\mathcal{B}$) the set of white (black) vertices. 
By convention, from a \emph{rooted bipartite} map we require that the rooted vertex is black. 
\cref{New2a} presents an example of a rooted bipartite map $M$. 
For a given bipartite map $M$ with $n$ edges we establish two integer partitions of $n$:
\begin{equation*}
\Lambda_\mathcal{W} \left( M\right) \quad \text{and} \quad \Lambda_\mathcal{B}\left( M\right) ,
\end{equation*}
\noindent
given by the degrees of white/black vertices, 
For such a map we assign also the third partition 
\begin{equation*}
\Lambda_\mathcal{F} \left( M\right) 
\end{equation*}
\noindent
of $n$, which describes the face structure of $M$; it is specified by reading halves of the numbers of edges fencing each face 
(since the map $M$ is bipartite, for each face there is an even number of edges adjacent to the face). 
The partition $\Lambda_\mathcal{F} \left( M\right)$ is called the \emph{face-type} of the map $M$.

\begin{de}
\label{Maps}
For three given partitions $\pi,\sigma ,\lambda\vdash n$ we denote by $\M$ the set of all bipartite, rooted maps $M$ with $n$ edges, for which $\Lambda_\mathcal{W} ( M ) = \pi$ and $  \Lambda_\mathcal{B} ( M ) =\sigma $. 
Moreover we denote by $\Ml$ the set of all such a maps $M$ which additionally have the face-type $\lambda$, \ie $\Lambda_\mathcal{F} ( M ) = \lambda$, see \cref{New2a}.
\end{de}

\begin{figure}
\centering
\begin{tikzpicture}[scale=0.4,
white/.style={circle,thick,draw=black,fill=white,inner sep=2pt},
white2/.style={circle,thick,draw=white,fill=white,inner sep=2pt},
black/.style={circle,draw=black,fill=black,inner sep=2pt},
connection1/.style={draw=brown,thick,auto,brown,fontscale=2},
connection2/.style={draw=blue,thick,auto,blue,fontscale=2},
connection/.style={draw=gray,thick,gray,auto}
]
\scriptsize

\fill[violet!15] (0,0) rectangle (10,10);
\fill[green!15] (8,7) -- (6,9) -- (4,7) --(4,3) ;

\draw[very thick,decoration={
    markings,
    mark=at position 1  with {\arrow{>}}},
    postaction={decorate},double=black]  
(1,7) -- (1.6,6.2);

\begin{scope}
\clip (0,0) rectangle (10,10);

\fill[orange!15] (9,-1) -- (4,3) -- (7,4) -- (11,5.7) --(10,0);
\fill[orange!15] (-2,5.7) -- (4,3) -- (1,7) -- (4.2,11) -- (0,10);

\draw (1,7)  node (b1)     [black] {};
\draw (4,3)  node (w1)     [white] {};
\coordinate (w12)       at (6,5);
\coordinate (w11)       at (4,5);
\coordinate (w1prim)      at (11,5.7);
\coordinate (w1bis)       at (4.2,11);
\coordinate (b1bis)       at (9,-1);

\draw (7,4)  node (bb1)    [black] {};
\coordinate (bb1prim) at (-2,5.7);

\draw (8,7)  node (AA)     [black] {};
\draw (6,9)  node (BA)     [white] {};
\draw (4,7)  node (AB)     [black] {};

\draw[connection]         (w1)       to  (AA);
\draw[connection]         (AA)       to  (BA);
\draw[connection]         (BA)       to  (AB);
\draw[connection]         (AB)       to  (w1);

\draw[connection]         (b1)       --    (w1);
\draw[connection]         (b1)       to  (w1bis) ;
\draw[connection]         (bb1)      to  (w1prim);

\draw[connection]         (w1)       to  (bb1);
\draw[connection]         (w1)       to  (bb1prim);
\draw[connection]         (w1)       to  (b1bis);

\end{scope}

\draw[thick,decoration={
    markings,
    mark=at position 0.333  with {\arrow{<}}},
    postaction={decorate}]  
(0,0) -- (10,0);

\draw[thick,decoration={
    markings,
    mark=at position 0.666  with {\arrow{>}}},
    postaction={decorate}]   
(0,10) -- (10,10);

\draw[thick,decoration={
    markings,
    mark=at position 0.666  with {\arrow{>>}}},
    postaction={decorate}]  
(10,0) -- (10,10);

\draw[thick,decoration={
    markings,
    mark=at position 0.333  with {\arrow{<<}}},
    postaction={decorate}]   
(0,0) -- (0,10);

\begin{scope}
\clip (10,0) rectangle (23,10);
\draw[black,fontscale=1] 
 (13,4) circle (0pt) node[anchor=west] {$\Lambda_\mathcal{F} \left( M\right)= \left( 4,2,2\right)$};
\draw[black,fontscale=1] 
 (13,2.5) circle (0pt) node[anchor=west] {$\Lambda_\mathcal{W} \left( M\right)= \left( 6,2 \right)$};
\draw[black,fontscale=1] 
 (13,1) circle (0pt) node[anchor=west] {$\Lambda_\mathcal{B} \left( M\right)= \left( 2,2,2,2\right)$};
\end{scope}

\end{tikzpicture}
\caption{Example of a \emph{rooted bipartite} map $M$ on a projective plane. 
The left side of the square should be glued to the right side, as well as bottom to top, as indicated by the arrows. 
We present also the face, white and black vertex distributions.}
\label{New2a}
\end{figure}
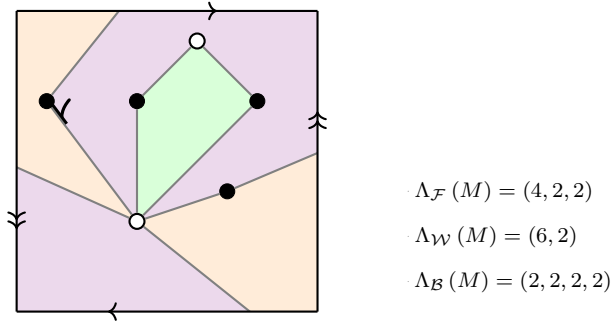

\begin{figure}
\centering
\begin{tikzpicture}[scale=0.35,
white/.style={circle,thick,draw=black,fill=white,inner sep=2pt},
white2/.style={circle,thick,draw=white,fill=white,inner sep=2pt},
black/.style={circle,draw=black,fill=black,inner sep=2pt},
connection1/.style={draw=brown,thick,auto,brown,fontscale=2},
connection2/.style={draw=blue,thick,auto,blue,fontscale=2},
connection/.style={draw=gray,thick,gray,auto}
]
\scriptsize

\fill[violet!15] (0,0) rectangle (10,10);
\draw[very thick,decoration={
    markings,
    mark=at position 1  with {\arrow{>}}},
    postaction={decorate},double=black]  
(8,5) -- (6,4);
\fill[violet!15] (8,5) -- (6,4) --(6,0) -- (10,0);

\begin{scope}
\clip (0,0) rectangle (10,10);

\draw (2,8)  node (b2)    [black] {};
\draw (8,5)  node (b1)     [black] {};
\draw (4,3)  node (w1)     [white] {};
\coordinate (L)       at (-2,5);
\coordinate (R)       at (11,4);
\coordinate (B)       at (3,0);
\coordinate (T)       at (3,10);

\draw[connection]         (w1)       to  (b1);
\draw[connection]         (b1)       to  (R);
\draw[connection]         (w1)       to  (L);
\draw[connection]         (w1)       to  (b2);
\draw[connection]         (b2)       to  (T);
\draw[connection]         (w1)       to  (B);

\end{scope}

\draw[very thick,decoration={
    markings,
    mark=at position 0.666  with {\arrow{>}}},
    postaction={decorate}]  
(0,0) -- (10,0);

\draw[very thick,decoration={
    markings,
    mark=at position 0.666  with {\arrow{>}}},
    postaction={decorate}]   
(0,10) -- (10,10);

\draw[very thick,decoration={
    markings,
    mark=at position 0.8 with {\arrow{>>}}},
    postaction={decorate}]  
(10,0) -- (10,10);

\draw[very thick,decoration={
    markings,
    mark=at position 0.8  with {\arrow{>>}}},
    postaction={decorate}]   
(0,0) -- (0,10);

\begin{scope}[shift={(12,0)}]
\clip (0,0) rectangle (10,10);

\coordinate (L)       at (-2,5);
\coordinate (R)       at (11,7);
\coordinate (B)       at (6.5,0);
\coordinate (T)       at (3.5,10);

\fill[green!15] (0,0) -- (L) -- (4,3) -- (B) -- (0,0);
\fill[green!15] (10,10) -- (R) -- (8,5)-- (4,3) --(2,8) -- (T) -- (10,10);

\draw[very thick,decoration={
    markings,
    mark=at position 1  with {\arrow{>}}},
    postaction={decorate},double=black]  
(2,8) -- (3,5.5);

\fill[orange!15] (10,0) -- (R) -- (8,5) -- (4,3) -- (B) -- (10,0);
\fill[orange!15] (0,10) -- (L) -- (4,3) -- (2,8) -- (T) -- (0,10);

\draw (2,8)  node (b2)    [black] {};
\draw (8,5)  node (b1)     [black] {};
\draw (4,3)  node (w1)     [white] {};

\draw[connection]         (w1)       to  (b1);
\draw[connection]         (b1)       to  (R);
\draw[connection]         (w1)       to  (L);
\draw[connection]         (w1)       to  (b2);
\draw[connection]         (b2)       to  (T);
\draw[connection]         (w1)       to  (B);

\end{scope}

\begin{scope}[shift={(12,0)}]
\draw[very thick,decoration={
    markings,
    mark=at position 0.333 with {\arrow{<}}},
    postaction={decorate}]  
(0,0) -- (10,0);

\draw[very thick,decoration={
    markings,
    mark=at position 0.666  with {\arrow{>}}},
    postaction={decorate}]   
(0,10) -- (10,10);

\draw[very thick,decoration={
    markings,
    mark=at position 0.2 with {\arrow{<<}}},
    postaction={decorate}]  
(10,0) -- (10,10);

\draw[very thick,decoration={
    markings,
    mark=at position 0.8  with {\arrow{>>}}},
    postaction={decorate}]   
(0,0) -- (0,10);
\end{scope}

\draw[black,fontscale=1] (23,6.25) circle (0pt) node[anchor=west] {$\mu= \left( 4,4\right)$};
\draw[black,fontscale=1]
 (23,4.5) circle (0pt) node[anchor=west] {$\Lambda_\mathcal{F} \left( M\right)= \left( 4,2,2\right)$};
\draw[black,fontscale=1]
 (23,2.75) circle (0pt) node[anchor=west] {$\Lambda_\mathcal{W} \left( M\right)= \left( 4,4 \right)$};
\draw[black,fontscale=1]
 (23,1) circle (0pt) node[anchor=west] {$\Lambda_\mathcal{B} \left( M\right)= \left( 2,2,2,2\right)$};

\begin{scope}
\clip (-1,-2.5) rectangle (25,0);
\draw[black,fontscale=2] (6,-1.5) circle (0pt) node[anchor=west] {$\mu_1 =4$};
\draw[black,fontscale=2] (18,-1.5) circle (0pt) node[anchor=west] {$\mu_2 =4$};

\tikzset{my circle/.pic={\node [draw, thick, circle, minimum width=17pt,fontscale=2] {\tikzpictext};},}
  
\pic [pic text=$2$] at (13,-1.5) {my circle};  
\pic [pic text=$1$] at (1,-1.5) {my circle};  
\end{scope} 

\end{tikzpicture}
\caption{Example of a \emph{rooted, bipartite} $\mu$-\emph{list} of maps for a partition $\mu =(4,4)$. 
The first map is drawn on a torus, the second one on a projective plane. 
We present also the face, white vertices and black vertices distributions. 
By erasing the roots and the numbering of the connected components we obtain a \emph{bipartite} $\mu$-\emph{collection} of maps.}
\label{New3a}
\end{figure}
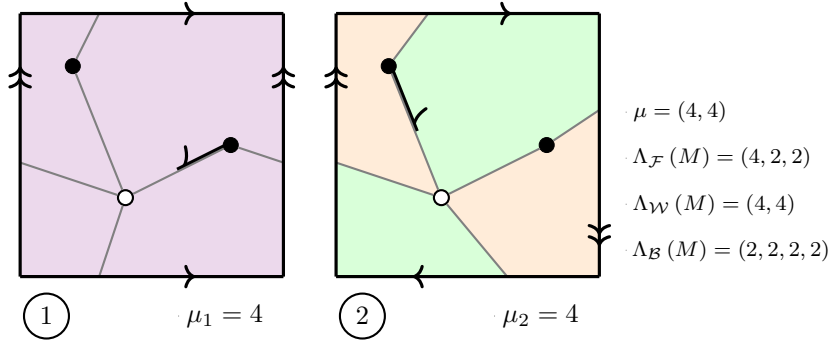

Due to the nature of our result we extend this definition slightly, namely we waive the assumption of connectedness in the definition of a map. 
There are two natural ways to generalize the notion of connected maps to non-connected ones: either we consider \emph{lists} of connected maps or we consider \emph{collections} of them. 

\subsection{Lists and collections of maps.}

\begin{de}
\label{maps1}
Let $\mu =(\mu_1,\ldots ,\mu_k )$ be a partition of an integer $n$. 
A list of maps $ (M_1,\ldots ,M_k )$ is called a $\mu$-\emph{list of maps} if the map $M_i$ has $\mu_i$ edges for each $i\in [k]$. 
We say that such a list is \emph{rooted}, respectively \emph{bipartite} if each map $M_i$ is so. 
For a bipartite $\mu$-list of maps we associate three partitions describing the black vertex, the white vertex and the face structures
\begin{equation*}
\Lambda_\mathcal{W} ( M ) := \bigcup_{i=1}^{k} \Lambda_\mathcal{W} ( M_i ) ,\quad
\Lambda_\mathcal{B} ( M ) := \bigcup_{i=1}^{k} \Lambda_\mathcal{B} ( M_i )   ,  \quad
\Lambda_\mathcal{F} ( M ) := \bigcup_{i=1}^{k} \Lambda_\mathcal{F} ( M_i ) ,
\end{equation*}
\noindent
where $\bigcup$ denotes the concatenation of partitions.
\end{de}

\begin{de}[Extension of \cref{Maps}]
\label{extension1}
For given partitions $\pi,\sigma, \mu \vdash n$, we denote by $\Mp$ the set of all bipartite rooted $\mu$-lists of maps $M$ which satisfy
\begin{equation*}
 \Lambda_\mathcal{W} (M ) = \pi  \quad\quad\text{and}\quad\quad
\Lambda_\mathcal{B} (M ) = \sigma ,
\end{equation*}
\noindent
see \cref{New3a}. Moreover, for a given partition $\lambda\vdash n$ we denote by $ \Mlp$ the set of all $\mu$-lists of maps $M \in \Mp$ which have face-type $\lambda$, \ie $\Lambda_\mathcal{F} (M)=\lambda$. 
\end{de}

\begin{de}
\label{maps3}
Let $\mu =(\mu_1,\ldots ,\mu_k )$ be a partition of an integer $n$. 
A set of maps $\{M_1,\ldots ,M_k\}$ is called a $\mu$-\emph{collection of maps} if the map $M_i$ has $\mu_i$ edges for each $i\in [k]$. 
We say that such a collection is \emph{rooted} or \emph{bipartite} if each map $M_i$ is so. For such a collection of maps we associate three partitions describing black, white and face structures as in \cref{maps1}.
\end{de}

Roughly speaking, a $\mu$-collection of maps could be created from a $\mu$-list of maps by erasing the numbering of the connected components, \ie the order on the connected components (see \cref{New3a}).

\subsection{Matchings and maps.}
\label{42}

Matching and maps are closely related notions. 
Roughly speaking, a bipartite matching can be treated as a (possibly non-connected) bipartite map with rooted and numbered faces. 
We shall discuss relations between matchings and rooted list of maps with the same face, black vertices and white vertices distribution.

\begin{de}
Consider two partitions $\lambda,\mu \vdash n$. We say that a bipartite $\mu$-collection $M$ of maps with the face distribution given by $\lambda$ has \emph{rooted and numbered faces} if all faces of $M$ are rooted (\ie on each face there is one marked edge-side) and the face labelled by the number $i$ is surrounded by $2 \lambda_i$ edges, for each $i$, see \cref{New3b}. The set of such collections of maps with the face, black vertices and white vertices distributions given by the partitions $\lambda,\pi,\sigma \vdash n$ is denoted by $\GMlp$.
\end{de}

\begin{rem}
\label{rooting}
Observe that rooting a face is nothing else but choosing one of the face corners adjacent to some black vertex and \emph{orienting} the face. 
Through a map (or a list/a collection of maps) with rooted faces we can understand a map with oriented faces and chosen black corners for each of the faces, see \cref{New3b}. 
Similarly, rooting a map is choosing one corner of a black vertex and orienting the face adjacent to this corner. 
\end{rem}

\begin{figure}
\centering
\begin{tikzpicture}[scale=0.4,
white/.style={circle,thick,draw=black,fill=white,inner sep=2pt},
white2/.style={circle,thick,draw=white,fill=white,inner sep=2pt},
black/.style={circle,draw=black,fill=black,inner sep=2pt},
connection1/.style={draw=brown,thick,auto,brown,fontscale=2},
connection2/.style={draw=blue,thick,auto,blue,fontscale=2},
connection/.style={draw=gray,thick,gray,auto}
]
\scriptsize

\begin{scope}
\clip (0,0) rectangle (10,10);

\fill[violet!15] (0,0) rectangle (10,10);

\draw[very thick,decoration={
    markings,
    mark=at position 0.9  with {\arrow{>}}},
    postaction={decorate},double=black]   
(2,8) -- (2.7,9.4);

\coordinate (L)       at (-2,5);
\coordinate (R)       at (11,4);
\coordinate (B)       at (3,0);
\coordinate (T)       at (3,10);

\fill[violet!15] (0,10) -- (L) -- (4,3) -- (2,8) -- (T) -- (0,10);

\draw (2,8)  node (b2)    [black] {};
\draw (8,5)  node (b1)     [black] {};
\draw (4,3)  node (w1)     [white] {};

\draw[connection]         (w1)       to  (b1);

\draw[connection]         (b1)       to  (R);

\draw[connection]         (w1)       to  (L);

\draw[connection]         (w1)       to  (b2);

\draw[connection]         (b2)       to  (T);

\draw[connection]         (w1)       to  (B);

\end{scope}

\draw[very thick,decoration={
    markings,
    mark=at position 0.666  with {\arrow{>}}},
    postaction={decorate}]  
(-0.1,-0.1) -- (10.1,-0.1);

\draw[very thick,decoration={
    markings,
    mark=at position 0.666  with {\arrow{>}}},
    postaction={decorate}]   
(-0.1,10.1) -- (10.1,10.1);

\draw[very thick,decoration={
    markings,
    mark=at position 0.8 with {\arrow{>>}}},
    postaction={decorate}]  
(10.1,-0.1) -- (10.1,10.1);

\draw[very thick,decoration={
    markings,
    mark=at position 0.8  with {\arrow{>>}}},
    postaction={decorate}]   
(-0.1,-0.1) -- (-0.1,10.1);

\draw[-{>[flex=0.75]},>=stealth,gray,thick] (7,7) arc (290:00:0.7cm);
\draw [->,>=stealth,red,thick] (3.2,7.1) -- +(b2);

\begin{scope}[shift={(15,0)}]
\clip (0,0) rectangle (10,10);

\coordinate (L)       at (-2,5);
\coordinate (R)       at (11,7);
\coordinate (B)       at (6.5,0);
\coordinate (T)       at (3.5,10);

\fill[green!15] (0,0) -- (L) -- (4,3) -- (B) -- (0,0);
\fill[green!15] (10,10) -- (R) -- (8,5) -- (10,10);
\fill[orange!15] (10,0) -- (8,5) -- (4,3) -- (B) -- (10,0);

\draw[very thick,decoration={
    markings,
    mark=at position 0.8  with {\arrow{>}}},
    postaction={decorate},double=black]   
(8,5) -- (6,4);

\draw[very thick,decoration={
    markings,
    mark=at position 0.99  with {\arrow{>}}},
    postaction={decorate},double=black]   
(8,5) -- (9.5,6);
\fill[green!15] (10,10) -- (8,5)-- (4,3) --(2,8) -- (T) -- (10,10);
\fill[orange!15] (10,0) -- (R) -- (8,5)-- (10,0);
\fill[orange!15] (0,10) -- (L) -- (4,3) -- (2,8) -- (T) -- (0,10);

\draw (2,8)  node (b2)    [black] {};
\draw (8,5)  node (b1)     [black] {};
\draw (4,3)  node (w1)     [white] {};

\draw[connection]         (w1)       to  (b1);

\draw[connection]         (b1)       to  (R);

\draw[connection]         (w1)       to  (L);

\draw[connection]         (w1)       to  (b2);

\draw[connection]         (b2)       to  (T);

\draw[connection]         (w1)       to  (B);

\end{scope}

\begin{scope}[shift={(15,0)}]
\draw[very thick,decoration={
    markings,
    mark=at position 0.333 with {\arrow{<}}},
    postaction={decorate}]  
(-0.1,-0.1) -- (10.1,-0.1);

\draw[very thick,decoration={
    markings,
    mark=at position 0.666  with {\arrow{>}}},
    postaction={decorate}]   
(-0.1,10.1) -- (10.1,10.1);

\draw[very thick,decoration={
    markings,
    mark=at position 0.2 with {\arrow{<<}}},
    postaction={decorate}]  
(10.1,-0.1) -- (10.1,10.1);

\draw[very thick,decoration={
    markings,
    mark=at position 0.8  with {\arrow{>>}}},
    postaction={decorate}]   
(-0.1,-0.1) -- (-0.1,10.1);

\draw [->,>=stealth,red,thick] (7.5,6) -- +(b1);
\draw [->,>=stealth,red,thick] (8.2,4) -- +(b1);
\draw[-{>[flex=0.75]},>=stealth,gray,thick] (6,6.7) arc (20:310:0.7cm);
\draw[-{>[flex=0.75]},>=stealth,gray,thick] (8,2.7) arc (00:290:0.7cm);

\end{scope}

\end{tikzpicture}
\caption{Example of a bipartite $(4,4)$-collection of maps with rooted faces. By \emph{rooting faces} we understand choosing one edge-side of each face (drawn as a black half-arrow going from a black vertex) or, equivalently, orienting each face (the rounded arrows) and choosing one black vertex for each face (the red arrows).}
\label{New3b}
\end{figure}
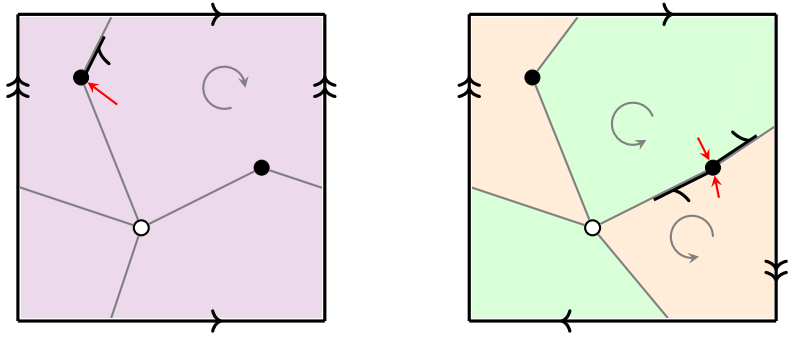

We consider four partitions: $\pi,\sigma, \lambda,\mu \vdash n$. 
To a given matching $\delta \in \Glp$ we associate a bipartite $\mu$-collection $M_\delta \in \GMlp$ given by the following procedure. 
\begin{enumerate}

\item 
The matchings $\epsilon$ and $\delta_\lambda$ determine the polygons with the \emph{vertices} labelled by $\mathcal{N}_n$, see \cref{matching}. 
We take theirs \emph{duals}, \ie the polygons with the \emph{edges} labelled by $\mathcal{N}_n$, see \cref{New1b}. 
The consecutive polygons have $2 \lambda_1 ,2 \lambda_2 ,\ldots$ edges. Observe that the parts of $\epsilon$ (respectively $\delta_\lambda$) can be identified with the black (respectively white) vertices as it is shown on \cref{New1b};

\item 
The matching $\delta$ determines the unique way of gluing the edges of the polygons in such a way that black (white) vertices are glued with black (white) ones. 
\cref{New2b} presents such a gluing for the matching 
\begin{equation*}
\delta= \big\{ \{\hat{1},\hat{6} \}, \{ 1, 5\}, \{\hat{2},\hat{8} \}, \{ 2, 7\} , 
\{\hat{3},\hat{7} \}, \{ 3, 8\}, \{\hat{4},6 \}, \{ 4, \hat{5}\} \big\} .
\end{equation*}
Observe that the distribution of black (respectively white) vertices is given by $\Lambda (\delta, \epsilon ) $ (respectively by $\Lambda (\delta, \delta_\lambda )$). Moreover, $\mu = \Lambda (\delta,\epsilon,\delta_\lambda )$.

\item
Each face is canonically numbered by an integer $s$ related to the polygon $\lambda_s$, \ie the edge-sides of this face are labelled by the elements
\begin{equation*}
 \sum_{i=1}^{s-1} \lambda_i +1 ,\ldots , \sum_{i=1}^{s-1} \lambda_i +\lambda_s , 
 \widehat{\sum_{i=1}^{s-1} \lambda_i +1}, \ldots \widehat{\sum_{i=1}^{s-1} \lambda_i +\lambda_s } .
\end{equation*}
Such a face is canonically rooted by selecting the edge-side labelled by the number $\sum_{i=1}^{s-1} \lambda_i +1$, see \cref{New2b}.

\item
We remove the labelling by the elements from $\mathcal{N}_n$. 
\end{enumerate}

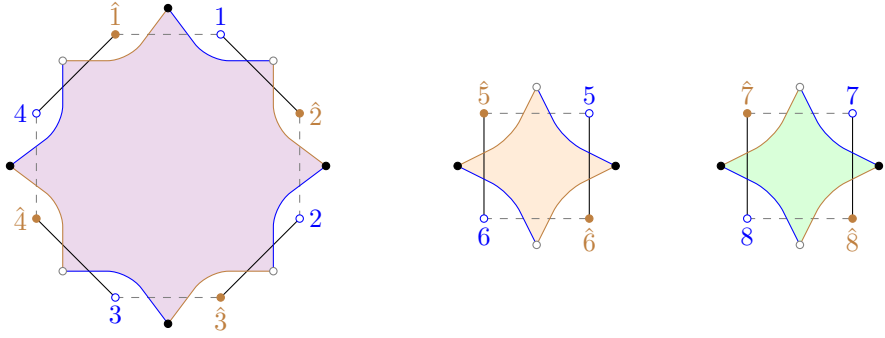
\begin{figure}
\centering
\begin{tikzpicture}[scale=0.35]
\draw[name path = 2I,violet!15] (5,11) -- (5,5) --(9,9) -- (5,5) -- (11,5) -- (5,5) -- (9,1) -- (5,5) -- (5,-1) -- (5,5) --(1,1) -- (5,5) -- (-1,5) -- (5,5) -- (1,9) -- (5,5) -- (5,11);

\draw[gray,dashed] (3,10)  -- (7,10) node [midway, below, sloped] (TextNode) {};
\draw[gray,dashed] (3,0)  -- (7,0) node [midway, below, sloped] (TextNode) {};
\draw[gray,dashed] (0,3)  -- (0,7) node [midway, below, sloped] (TextNode) {};
\draw[gray,dashed] (10,3)  -- (10,7) node [midway, below, sloped] (TextNode) {};

\draw[black] (3,10)  -- (0,7) node [midway, below, sloped] (TextNode) {};
\draw[black] (10,3)  -- (7,0) node [midway, below, sloped] (TextNode) {};
\draw[black] (0,3)  -- (3,0) node [midway, below, sloped] (TextNode) {};
\draw[black] (10,7)  -- (7,10) node [midway, below, sloped] (TextNode) {};

\filldraw[brown] (3,10) circle (4pt) node[anchor=south] {$\hat{1}$};
\filldraw[brown] (10,7) circle (4pt) node[anchor=west] {$\hat{2}$};
\filldraw[brown] (7,0) circle (4pt) node[anchor=north] {$\hat{3}$};
\filldraw[brown] (0,3) circle (4pt) node[anchor=east] {$\hat{4}$};

\fill[white] (7,10) circle (4pt) node[anchor=east] {};
\draw[blue] (7,10) circle (4pt) node[anchor=south] {$1$};
\fill[white] (10,3) circle (4pt) node[anchor=east] {};
\draw[blue] (10,3) circle (4pt) node[anchor=west] {$2$};

\fill[white] (3,0) circle (4pt) node[anchor=east] {};
\draw[blue] (3,0) circle (4pt) node[anchor=north] {$3$};
\fill[white] (0,7) circle (4pt) node[anchor=east] {};
\draw[blue] (0,7) circle (4pt) node[anchor=east] {$4$};

\draw[name path = 2A,blue,rounded corners=8pt]  (9,9)-- (6.5,9) -- (5,11);
\draw[name path = 2B,brown,rounded corners=8pt] (11,5) -- (9,6.5) -- (9,9);
\draw[name path = 2C,blue,rounded corners=8pt] (9,1)-- (9,3.5) -- (11,5);
\draw[name path = 2D,brown,rounded corners=8pt] (5,-1)-- (6.5,1) -- (9,1);
\draw[name path = 2E,brown,rounded corners=8pt] (1,9)-- (3.5,9) -- (5,11);
\draw[name path = 2F,blue,rounded corners=8pt] (1,9) -- (1,6.5) -- (-1,5);
\draw[name path = 2G,brown,rounded corners=8pt] (-1,5)-- (1,3.5) --(1,1) ;
\draw[name path = 2H,blue,rounded corners=8pt] (1,1)-- (3.5,1) -- (5,-1);

\filldraw[black] (-1,5) circle (4pt);
\filldraw[black] (11,5) circle (4pt);
\filldraw[black] (5,11) circle (4pt);
\filldraw[black] (5,-1) circle (4pt);

\fill[white] (1,9) circle (4pt);
\draw[gray] (1,9) circle (4pt);
\fill[white] (1,1) circle (4pt);
\draw[gray] (1,1) circle (4pt);
\fill[white] (9,9) circle (4pt);
\draw[gray] (9,9) circle (4pt);
\fill[white] (9,1) circle (4pt);
\draw[gray] (9,1) circle (4pt);

\begin{pgfonlayer}{bg}
\fill [violet!15,intersection segments={of=2I and 2A,sequence={L2--R2}}];
\fill [violet!15,intersection segments={of=2I and 2B,sequence={L2--R2}}];          
\fill [violet!15,intersection segments={of=2I and 2C,sequence={L2--R2}}];
\fill [violet!15,intersection segments={of=2D and 2I,sequence={L2--R2}}];  
\fill [violet!15,intersection segments={of=2E and 2I,sequence={L2--R2}}];
\fill [violet!15,intersection segments={of=2F and 2I,sequence={L2--R2}}];          
\fill [violet!15,intersection segments={of=2G and 2I,sequence={L2--R2}}];
\fill [violet!15,intersection segments={of=2H and 2I,sequence={L2--R2}}];  
\end{pgfonlayer}

\draw[name path = 2E,orange!15] (22,5) -- (19,5) --(19,8) -- (19,5) -- (19,2) -- (19,5) -- (16,5) -- (19,5) -- (22,5);

\draw[gray,dashed] (17,7)  -- (21,7) node [midway, below, sloped] (TextNode) {};
\draw[gray,dashed] (17,3)  -- (21,3) node [midway, below, sloped] (TextNode) {};

\draw[black] (17,3)  -- (17,7) node [midway, below, sloped] (TextNode) {};
\draw[black] (21,3)  -- (21,7) node [midway, below, sloped] (TextNode) {};

\filldraw[brown] (17,7) circle (4pt) node[anchor=south] {$\hat{5}$};
\filldraw[brown] (21,3) circle (4pt) node[anchor=north] {$\hat{6}$};

\fill[white] (21,7) circle (4pt) node[anchor=east] {};
\draw[blue] (21,7) circle (4pt) node[anchor=south] {$5$};
\fill[white] (17,3) circle (4pt) node[anchor=east] {};
\draw[blue] (17,3) circle (4pt) node[anchor=north] {$6$};

\draw[name path = 2A,blue,rounded corners=8pt] (19,8)-- (20,6) -- (22,5);
\draw[name path = 2B,blue,rounded corners=8pt] (16,5)-- (18,4) -- (19,2);

\draw[name path = 2C,brown,rounded corners=8pt] (19,2)-- (20,4) -- (22,5);
\draw[name path = 2D,brown,rounded corners=8pt] (16,5) -- (18,6) -- (19,8);

\filldraw[black] (22,5) circle (4pt);
\filldraw[black] (16,5) circle (4pt);

\fill[white] (19,8) circle (4pt);
\draw[gray] (19,8) circle (4pt);
\fill[white] (19,2) circle (4pt);
\draw[gray] (19,2) circle (4pt);

\begin{pgfonlayer}{bg}
\fill [orange!15,intersection segments={of=2A and 2E,sequence={L2--R2}}];
\fill [orange!15,intersection segments={of=2B and 2E,sequence={L2--R2}}];          
\fill [orange!15,intersection segments={of=2C and 2E,sequence={L2--R2}}];
\fill [orange!15,intersection segments={of=2E and 2D,sequence={L2--R2}}];  
\end{pgfonlayer}

\begin{scope}[shift={(10,0)}]
\draw[name path = 2E,green!15] (22,5) -- (19,5) --(19,8) -- (19,5) -- (19,2) -- (19,5) -- (16,5) -- (19,5) -- (22,5);
\end{scope}

\draw[gray,dashed] (27,7)  -- (31,7) node [midway, below, sloped] (TextNode) {};
\draw[gray,dashed] (27,3)  -- (31,3) node [midway, below, sloped] (TextNode) {};

\draw[black] (27,3)  -- (27,7) node [midway, below, sloped] (TextNode) {};
\draw[black] (31,3)  -- (31,7) node [midway, below, sloped] (TextNode) {};

\filldraw[brown] (27,7) circle (4pt) node[anchor=south] {$\hat{7}$};
\filldraw[brown] (31,3) circle (4pt) node[anchor=north] {$\hat{8}$};

\fill[white] (31,7) circle (4pt) node[anchor=east] {};
\draw[blue] (31,7) circle (4pt) node[anchor=south] {$7$};
\fill[white] (27,3) circle (4pt) node[anchor=east] {};
\draw[blue] (27,3) circle (4pt) node[anchor=north] {$8$};

\begin{scope}[shift={(10,0)}]
\draw[name path = 2A,blue,rounded corners=8pt] (19,8)-- (20,6) -- (22,5);
\draw[name path = 2B,blue,rounded corners=8pt] (16,5)-- (18,4) -- (19,2);

\draw[name path = 2C,brown,rounded corners=8pt] (19,2)-- (20,4) -- (22,5);
\draw[name path = 2D,brown,rounded corners=8pt] (16,5) -- (18,6) -- (19,8);

\filldraw[black] (22,5) circle (4pt);
\filldraw[black] (16,5) circle (4pt);

\fill[white] (19,8) circle (4pt);
\draw[gray] (19,8) circle (4pt);
\fill[white] (19,2) circle (4pt);
\draw[gray] (19,2) circle (4pt);

\begin{pgfonlayer}{bg}
\fill [green!15,intersection segments={of=2A and 2E,sequence={L2--R2}}];
\fill [green!15,intersection segments={of=2B and 2E,sequence={L2--R2}}];          
\fill [green!15,intersection segments={of=2C and 2E,sequence={L2--R2}}];
\fill [green!15,intersection segments={of=2E and 2D,sequence={L2--R2}}];  
\end{pgfonlayer}
\end{scope}

\end{tikzpicture}
\caption{Duals of the polygons created by the matchings $\epsilon$ and $\delta_\lambda$ presented on \cref{matching}. Black (respectively white) vertices of such polygons are labelled by the elements of $\epsilon$ (respectively $\delta_\lambda$), the edges by the elements from $\mathcal{N}_8$.} 
\label{New1b}
\end{figure}

\newcommand{\faceA}{red!50}
\newcommand{\faceB}{blue}
\newcommand{\faceAfill}{red!10}
\newcommand{\faceBfill}{blue!20}
\newcommand{\faceCfill}{black!50!green!10}

\begin{figure}
\centering
\begin{tikzpicture}[scale=0.6,
white/.style={circle,thick,draw=black,fill=white,inner sep=2pt},
white2/.style={circle,thick,draw=white,fill=white,inner sep=2pt},
black/.style={circle,draw=black,fill=black,inner sep=2pt},
connection1/.style={draw=brown,thick,auto,brown,fontscale=2},
connection2/.style={draw=blue,thick,auto,blue,fontscale=2},
connection/.style={draw=gray,thick,gray,auto}
]
\scriptsize

\fill[violet!15] (0,0) rectangle (10,10);
\fill[green!15] (8,7) -- (6,9) -- (4,7) --(4,3) ;

\draw[very thick,red,decoration={
    markings,
    mark=at position 0.99  with {\arrow{>}}},
    postaction={decorate},double=black] 
(8,7)  -- (7,8);

\fill[violet!15] (8,7) --(7,8) --(8,9) --(9,8);

\begin{scope}
\clip (0,0) rectangle (10,10);

\fill[orange!15] (9,-1) -- (4,3) -- (7,4) -- (11,5.7) --(10,0);
\fill[orange!15] (-2,5.7) -- (4,3) -- (1,7) -- (4.2,11) -- (0,10);

\draw[very thick,red,decoration={
    markings,
    mark=at position 0.99  with {\arrow{>}}},
    postaction={decorate},double=black]   
(7,4) -- (6,3.66);

\draw (1,7)  node (b1)     [black] {};
\draw (4,3)  node (w1)     [white] {};
\coordinate (w12)       at (6,5);
\coordinate (w11)       at (4,5);
\coordinate (w1prim)      at (11,5.7);
\coordinate (w1bis)       at (4.2,11);
\coordinate (b1bis)       at (9,-1);

\draw (7,4)  node (bb1)    [black] {};
\coordinate (bb1prim) at (-2,5.7);

\draw (8,7)  node (AA)     [black] {};
\draw (6,9)  node (BA)     [white] {};
\draw (4,7)  node (AB)     [black] {};

\draw[connection1]        (w12)      -- node[below] {$\hat{2}$} (AA);
\draw[connection2]        (7,8)       -- node[below] {$7$}  (BA);
\draw[connection1]        (BA)       -- node[below] {$\hat{7}$}  (AB);
\draw[connection2]        (AB)       -- node[right] {$8$}  (w11);
\draw[connection1]        (w12)      -- node[above] {$\hat{8}$}  (AA);
\draw[connection2]        (AA)       -- node[above] {$2$}   (BA);
\draw[connection1]        (BA)       -- node[above] {$\hat{3}$}   (AB);
\draw[connection2]        (AB)       -- node[left] {$3$}  (w11);

\draw[connection]         (w1)       to  (AA);
\draw[connection]         (AA)       to  (BA);
\draw[connection]         (BA)       to  (AB);
\draw[connection]         (AB)       to  (w1);

\draw[connection2]        (b1)       -- node[below] {$6$}      (w1);
\draw[connection1]        (b1)       -- node[above] {$\hat{4}$}   (w1);
\draw[connection]         (b1)       --    (w1);
\draw[connection2]        (b1)       -- node[below] {$4$}      (w1bis) ;
\draw[connection1]        (b1)       -- node[above] {$\hat{5}$}    (w1bis) ;
\draw[connection]         (b1)       to  (w1bis) ;
\draw[connection1]        (bb1)      -- node[below] {$\hat{6}$}   (w1prim);
\draw[connection1]        (bb1)      -- node[above] {$\hat{1}$}  (w1prim);
\draw[connection]         (bb1)      to  (w1prim);

\draw[connection2]        (w1)       -- node[below] {$5$}   (bb1);
\draw[connection2]        (w1)       -- node[above] {$1$}  (bb1);
\draw[connection]         (w1)       to  (bb1);
\draw[connection1]        (w1)       -- node[below] {$\hat{1}$}    (bb1prim);
\draw[connection1]        (w1)       -- node[above] {$\hat{6}$}  (bb1prim);
\draw[connection]         (w1)       to  (bb1prim);
\draw[connection2]        (w1)       -- node[below] {$4$}   (b1bis);
\draw[connection1]        (w1)       -- node[above] {$\hat{5}$}  (b1bis);
\draw[connection]         (w1)       to  (b1bis);

\end{scope}

\draw[very thick,decoration={
    markings,
    mark=at position 0.333  with {\arrow{<}}},
    postaction={decorate},double=black]  
(-0.1,-0.1) -- (10.1,-0.1);

\draw[very thick,decoration={
    markings,
    mark=at position 0.666  with {\arrow{>}}},
    postaction={decorate},double=black]   
(-0.1,10.1) -- (10.1,10.1);

\draw[very thick,decoration={
    markings,
    mark=at position 0.666  with {\arrow{>>}}},
    postaction={decorate},double=black]  
(10.1,-0.1) -- (10.1,10.1);

\draw[very thick,decoration={
    markings,
    mark=at position 0.333  with {\arrow{<<}}},
    postaction={decorate},double=black]   
(-0.1,-0.1) -- (-0.1,10.1);

\tikzset{my circle/.pic={\node [draw, thick, circle, minimum width=17pt,fontscale=2] {\tikzpictext};},}
\pic [pic text=$1$] at (2,1.5) {my circle};  
\pic [pic text=$2$] at (8,2) {my circle}; 
\pic [pic text=$3$] at (5.4,6.5) {my circle};

\begin{scope}
\clip (-5,-3) rectangle (15,0);
\draw[black,fontscale=2] (5,-2) circle (0pt) node  {$\delta= \big\{ \{ \color{brown}{\hat{1}}\color{black},\color{brown}\hat{6} \color{black}\}, \{ \color{blue}1\color{black}, \color{blue}5\color{black}\}, \{\color{brown}\hat{2}\color{black},\color{brown}\hat{8}\color{black} \}, \{ \color{blue}2\color{black},\color{blue}7\color{black}\} , \{\color{brown}\hat{3}\color{black},\color{brown}\hat{7} \color{black}\}, \{\color{blue} 3\color{black}, \color{blue}8\color{black}\}, \{\color{brown}\hat{4}\color{black},\color{blue}6 \color{black}\}, \{ \color{blue}4\color{black}, \color{brown}\hat{5}\color{black}\} \big\} $};
\end{scope}
\end{tikzpicture}
\caption{Matching $\delta$ on the set $\mathcal{N}_8$ describes the way of gluing the sides of the polygons from \cref{New1b}. 
Labels from $\mathcal{N}_8$ determine the way of numbering and rooting faces of such a map (in general it could be a collection of maps), the roots (presented as half-arrows) correspond to the labels $1,5,7$. Numbers of faces are presented in black circles.} 
\label{New2b}
\end{figure}

\begin{cor}
\label{coro}
The procedure described above gives a bijection $\delta \mapsto M_\delta$ between the set of matchings $ \Glp$ and the set of collections of maps $ \GMlp$. 
\end{cor}

We compare the terminologies of matchings and maps in the table below. 

\begin{center}
  \label{Table}
 \setlength{\tabcolsep}{0.5em} 
{\renewcommand{\arraystretch}{1.4}
  \begin{tabular}{| l | c | c |}
    \hline
                  & Matching $\delta$ & $\mu$-list of maps $M$\\ \hline
    face-type & $ \lambda $ & $\Lambda_\mathcal{F} (M)$ \\ 
    distribution of black vertices & $\Lambda(\delta ,\epsilon) $ & $\Lambda_\mathcal{B} (M)$ \\ 
    distribution of white vertices & $\Lambda(\delta ,\delta_\lambda) $ & $\Lambda_\mathcal{W} (M)$ \\
    connected components & $\Lambda(\delta_\lambda ,\epsilon, \delta) $ & $\mu$ \\
    \makecell[l]{$\mu$-collections of maps with given faces, \\black and white vertices distribution \\ and with rooted and numbered faces}& \multirow{ 1}{*}{$\Glp$} &  \multirow{ 1}{*}{$\GMlp$} \\   
    \hline
  \end{tabular}}
\end{center} 

\subsection{Matchings and lists of rooted maps.}
We showed that matchings are equivalent to collections of maps with rooted and numbered \emph{faces}. 
However, collections of maps with rooted and numbered \emph{connected components} (\ie lists of rooted maps) are much more natural objects. 
We give a relation between those two ways of numbering and rooting collections of maps. More precisely, we present a relation between the set $\GMlp$ and the set $\Mlp$. 

What is common for those two classes is the fact that by rooting and numbering faces or connected components, the group of automorphisms becomes trivial.

\begin{de}
\label{rootGM}
For a given $\mu$-collection of maps with rooted and numbered faces $M \in \GMlp$ we define the set $\mathcal{R} ( M)$ of all numberings of the connected components and rooting each of them in such a way that with respect to them $M$ becomes a $\mu$-\emph{list} of maps from $ \Mlp$. 
We call  $\mathcal{R} ( M)$ the \emph{set of components-labellings} of $M$. 
For a given $r \in \mathcal{R} ( M)$ we denote $(M,r ) \in   \Mlp$. 

Similarly, for a given $\mu$-list of maps $M \in \Mlp$ we define the set $\mathcal{L} ( M)$ of all numberings of the faces and rooting each of them in such a way that $M $ becomes an element from $\GMlp$. 
We call  $\mathcal{L} ( M)$ the \emph{set of faces-labellings} of $M$. 
For a given $l \in \mathcal{L} ( M)$ we denote $(M,l ) \in   \GMlp$. 
\end{de}

\begin{ob}
\label{RL}
Let us fix partitions $\pi,\sigma, \mu , \lambda \vdash n$. For each $M_1 \in \GMlp$ and $M_2 \in \Mlp$ we have
\begin{equation*}
\Big\vert \mathcal{R} ( M_1)\Big\vert = 2^{\ell (\mu )} z_\mu  
\quad\text{and}\quad
\Big\vert \mathcal{L} ( M_1) \Big\vert= 2^{\ell (\lambda )} z_\lambda.
\end{equation*}
\end{ob}

\begin{proof}
Let us take $M \in \GMlp$. 
There is $\prod_i {m_i(\mu)} !$ ways of numbering the connected components and $ \prod_i \big( 2 i \big)^{ m_i(\mu)} $ ways of rooting each of them. 
We may carry out a similar deduction for $M \in \Mlp$.
\end{proof}

\begin{ob}
\label{strange}
For given partitions $\pi,\sigma, \mu , \lambda \vdash n$ we have
\begin{equation*}
\Big\vert \Glp \Big\vert =
\Big\vert \GMlp \Big\vert= 
\dfrac{z_\lambda 2^{\ell (\lambda)}}{z_\mu 2^{\ell (\mu)}}
\Big\vert  \Mlp \Big\vert .
\end{equation*} 
\end{ob}

\begin{proof}
The first equation follows from \cref{coro}. We investigate the second one. 
Each collection of maps from $\GMlp$ has rooted and numbered faces, each collection of maps from $\Mlp$ has rooted and numbered components. 
From each of them we can get a collection of maps which have rooted and numbered both: faces and components. The number of ways of doing it is given in \cref{RL}. We use the double counting method and conclude the second equation. 
\end{proof}

\subsection{Orientable maps and bipartite matchings.}

By an \emph{orientable map} we understand a map which is drawn on an orientable surface. 
An \emph{orientation} of a map is given by orienting each face 
in such a way, that the two edge-sides forming the same edge are oriented in the opposite way. 
We say that such an orientation of faces is \emph{coherent}. 
Orienting any face is equivalent to orienting a map. 
Observe that a rooted map possesses the canonical orientation given by the root, see \cref{rooting}. 
By a \emph{rooted orientable} map we understand an orientable map together with the \emph{orientation} given by the root, see \cref{Orientedmap}.

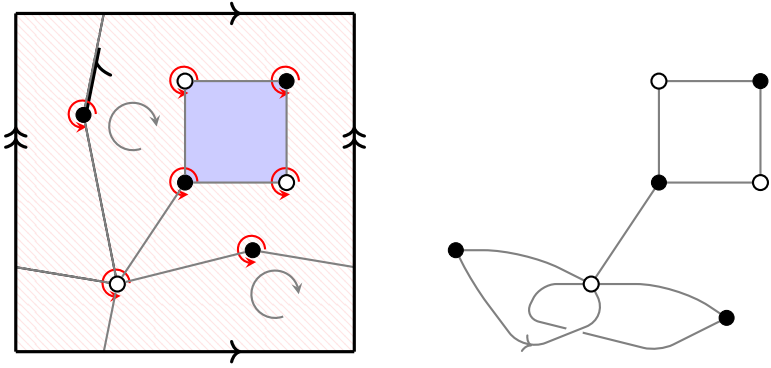
\begin{figure}
\centering
\begin{tikzpicture}[scale=0.45,
white/.style={circle,thick,draw=black,fill=white,inner sep=2pt},
white2/.style={circle,thick,draw=white,fill=white,inner sep=2pt},
black/.style={circle,draw=black,fill=black,inner sep=2pt},
connection/.style={draw=gray,thick,gray,auto}
]
\scriptsize

\draw[gray, thick]  (19,5) --  (22,5) ;
\draw[gray, thick]  (19,5) --  (19,8) ;
\draw[gray, thick]  (19,8) --  (22,8) ;
\draw[gray, thick]  (22,5) --  (22,8) ;
\draw[gray, thick] (17,2) -- (19,5);

\draw[gray, thick,rounded corners=15pt] (17,2)  -- (15,3) -- (13,3) ;
\draw[gray, thick,rounded corners=15pt] (17,2)  -- (19.5,2) -- (21,1) ;
\draw[gray, thick,rounded corners=6pt] (17,2)  -- (15.5,2) -- (15,1)  -- (19,0) -- (21,1) ;
\draw (16.5,0.6)  node (w1)     [white2] {};
\draw[gray, thick,decoration={markings,mark=at position 0.5  with {\arrow{<}}},postaction={decorate},rounded corners=9pt] (17,2)  --  (17.5,1) --  (15,0) --  (13.5,2) -- (13,3) ;

\draw (17,2)  node (w1)     [white] {};

\draw (13,3)  node (b1)     [black] {};
\draw (21,1)  node (bb1)    [black] {};

\draw (19,5)  node (AA)     [black] {};
\draw (22,5)  node (BA)     [white] {};
\draw (19,8)  node (AB)     [white] {};
\draw (22,8)  node (BB)     [black] {};

\draw[white] (23,1) circle (0pt) node[anchor=south] {};

\begin{scope}
\clip (0,0) rectangle (10,10);

\fill[pattern color=\faceAfill,pattern=north west lines] (0,0) rectangle (10,10);
\draw[very thick,decoration={
    markings,
    mark=at position 0.9  with {\arrow{>}}},
    postaction={decorate},double=black]   
(2,7) -- (2.4,9);
\fill[white] (-3,3) --(3,2) -- (2,7) --(3,12) -- (0,10);
\fill[pattern color=\faceAfill,pattern=north west lines] (-3,3) --(3,2) -- (2,7) --(3,12) -- (0,10);
\fill[\faceBfill] (5,5) rectangle (8,8);

\draw[-{>[flex=0.75]},>=stealth,gray,thick] (7.9,1.04) arc (290:00:0.7cm);
\draw[-{>[flex=0.75]},>=stealth,gray,thick] (3.7,6) arc (290:00:0.7cm);
\draw[{<[flex=0.75]}-,>=stealth,red,thick] (2.1,6.65) arc (290:00:0.4cm);
\draw[{<[flex=0.75]}-,>=stealth,red,thick] (3.1,1.65) arc (290:00:0.4cm);
\draw[{<[flex=0.75]}-,>=stealth,red,thick] (5.1,4.65) arc (290:00:0.4cm);
\draw[{<[flex=0.75]}-,>=stealth,red,thick] (8.1,4.65) arc (290:00:0.4cm);
\draw[{<[flex=0.75]}-,>=stealth,red,thick] (8.1,7.65) arc (290:00:0.4cm);
\draw[{<[flex=0.75]}-,>=stealth,red,thick] (5.1,7.65) arc (290:00:0.4cm);
\draw[{<[flex=0.75]}-,>=stealth,red,thick] (7.1,2.65) arc (290:00:0.4cm);

\draw (2,7)  node (b1)     [black] {};
\draw (3,2)  node (w1)     [white] {};
\coordinate (w1prim)      at (13,2);
\coordinate (w1bis)       at (3,12);
\coordinate (b1bis)       at (2,-3);

\draw (7,3)  node (bb1)    [black] {};
\coordinate (bb1prim) at (-3,3);

\draw (5,5)  node (AA)     [black] {};
\draw (8,5)  node (BA)     [white] {};
\draw (5,8)  node (AB)     [white] {};
\draw (8,8)  node (BB)     [black] {};

\draw[connection]         (w1)      to  (AA);
\draw[connection]         (AA)      to  (AB);
\draw[connection]         (AB)      to  (BB);
\draw[connection]         (BB)      to  (BA);
\draw[connection]         (BA)      to  (AA);

\draw[connection]         (b1)       to  (w1);
\draw[connection]         (b1)       to  (w1bis);
\draw[connection]         (bb1)      to  (w1prim);

\draw[connection]  (w1)      to  (bb1);
\draw[connection]         (w1)      to  (bb1prim);
\draw[connection]         (w1)      to  (b1bis);

\end{scope}

\draw[very thick,decoration={
    markings,
    mark=at position 0.666  with {\arrow{>}}},
    postaction={decorate}]  
(0,0) -- (10,0);

\draw[very thick,decoration={
    markings,
    mark=at position 0.666  with {\arrow{>}}},
    postaction={decorate}]  
(0,10) -- (10,10);

\draw[very thick,decoration={
    markings,
    mark=at position 0.666  with {\arrow{>>}}},
    postaction={decorate}]  
(10,0) -- (10,10);

\draw[very thick,decoration={
    markings,
    mark=at position 0.666  with {\arrow{>>}}},
    postaction={decorate}]  
(0,0) -- (0,10);
\end{tikzpicture}
\caption{Example of a \emph{rooted oriented} map $M$ drawn as a graph on a torus (on the left). 
There is the canonical orientation (grey arrows) given by the root. 
We are going to present oriented maps in such a way that their orientation is consistent with the \emph{clockwise} orientation of the page (grey arrows) or, equivalently, the \emph{counter-clockwise} orientation around each vertex (red arrows). The distinction between chosen orientations of the page and the vertices may seem awkward. However, it is more convenient for the purpose of \cref{M}. 
With this convention we can present the root of a map (similarly roots of lists of maps) by an \emph{arrow} going out from a black vertex. 
Since $M$ is oriented, it can be recovered from a graphical representation on the plane as a graph with a fixed cyclic order of outgoing edges around each vertex together with a choice of the root (on the right).}
\label{Orientedmap}
\end{figure}

\begin{de}
\label{orient}
We use the following notation:
\begin{align*} 
\MOlp &:= \left\lbrace M\in \Mlp : M \text{ is orientable} \right\rbrace ,\\
\MOp &:= \left\lbrace M\in \Mp : M \text{ is orientable} \right\rbrace ,\\
\GOlp &:= \left\lbrace \delta\in \Glp : \delta \text{ is bipartite} \right\rbrace .
\end{align*}
\end{de}

The notion of bipartiteness of a matching is closely related to the notion of orientability. 

\begin{ob}
\label{strange2}
For given partitions $\pi,\sigma, \mu , \lambda \vdash n$, we have
\begin{equation*}
\Big\vert \GOlp \Big\vert =
\dfrac{z_\lambda}{z_\mu}
\Big\vert  \MOlp \Big\vert .
\end{equation*} 
\end{ob}

\begin{proof}
We identify a matching $\delta \in \Glp$ with a collection of maps $M_\delta \in \GMlp$ with rooted and numbered faces by the procedure described in \cref{42}. 
Observe that a \emph{bipartite} matching corresponds to a collection of \emph{oriented} maps. 
Indeed, the orientations of faces given by the edge-sides: $1, \lambda_1+1 ,\ldots$ are coherent. 
\cref{strange} gives a relation between collections of maps with rooted and numbered faces and collections of maps with rooted and numbered components (lists of maps). An analysis similar to the one given in \cref{strange} convinces us that the quantity $2^{\ell (\mu)} \prod_i  i^{m_i(\lambda)} $ specifies the number of manners of rooting the faces in a \emph{coherent} way and $\prod_i  m_i(\lambda) !$ specifies the number of manners of numbering the faces. 
On the other hand, the quantity $ z_\mu 2^{\ell (\mu)}$ is relevant for numbering and rooting the connected components. 
We use the double counting method and conclude the statement. 
\end{proof}

\section{Measures of non-orientability and non-bipartiteness}
\label{4444}

\subsection{The $b$-Conjecture.}

Equations \cref{star} and \cref{star2} define two families of coefficients $\left( c_{\pi, \sigma}^{\lambda}\right) $ and $\left( h_{\pi, \sigma}^{\lambda} \right)$. 
Goulden and Jackson \cite{MR1325917} discussed some specialisations of the family $\left( c_{\pi, \sigma}^{\lambda}\right) $ and hypothetical combinatorial interpretations of the polynomials $c_{\pi, \sigma}^{\lambda}$ in terms of matchings known as the \emph{Matchings-Jack Conjecture}, see \cref{444}. 
In the same paper they observed that specializations of $h_{\pi, \sigma}^{\lambda} (\beta )$ for $\beta =0,1$ may be expressed in terms of rooted maps, namely
\begin{align*} 
h_{\pi, \sigma}^{\lambda} ( 0 ) &= \Big\vert \big\{ M \in \Ml : 
 M \text{ is orientable} \big\} \Big\vert ,\\
h_{\pi, \sigma}^{\lambda} ( 1 )&=  \Big\vert  \big\{ M \in \Ml \big\} \Big\vert. 
\end{align*} 
\noindent
Based on this observation Goulden and Jackson conjectured that the family $\left( h_{\pi, \sigma}^{\lambda} \right) $  of polynomials may have a combinatorial interpretation. 
The conjecture is known as the \emph{$b$-Conjecture}.

\begin{con}[$b$-Conjecture]
For any partitions $\pi,\sigma,\lambda\vdash n$ the quantity $h_{\pi, \sigma}^{\lambda}$ can be expressed as
\begin{equation*}
h_{\pi, \sigma}^{\lambda} (\beta ) =\sum_{M \in \Ml} \beta^{\eta (M )},
\end{equation*}
\noindent
where $\eta : \Ml \longrightarrow \mathbb{N}_0$ is some hypothetical combinatorial statistic such that $\eta (M)=0$ if and only if $M $ is orientable.
\end{con}

\subsection{Root-deletion procedure and a measure of non-orientability.}
\label{M}

The statistic $\eta$ from $b$-Conjecture should be a marker of non-orientability of maps. 
We shall present the definition of the measure of non-orientability introduced by La Croix \cite[Definition 4.1]{Lacroix}, which seems to be a good candidate for the hypothetical statistic conjectured by Goulden and Jackson. 
We adapt the statistic given by La Croix to the case of lists of maps.

\begin{de}[Root-deletion procedure]
\label{rootd}
Denote by $e$ the root edge of the map $M$. 
By deleting $e$ from $M$ we create either a new map, or two new maps. 
We give the canonical procedure of rooting it or them. 
By rooting a map we will understand choosing an oriented corner, see \cref{deletion}. 
Denote by $c$ the root corner of $M$. 

Suppose that $M \setminus e$ is connected. 
Observe that $c$ is contained in the unique oriented corner of $M \setminus e$, we define such an oriented corner as the root of $M \setminus e$. 

Suppose that $M \setminus e$ has two connected components. One of them can be rooted as above. Observe that the first corner in the root face of $M$ following $c$ is contained in the unique oriented corner of the second component of $M \setminus e$, see \cref{deletion}. We define such an oriented corner as the root of this component. 
\end{de}

\begin{rem}
The Root-deletion procedure is defined for all maps, not necessary bipartite. In particular, we do not require that the rooted vertex is black. 
\end{rem}

\begin{figure}\centering
\begin{tikzpicture}[scale=0.45,
white/.style={circle,thick,draw=black,fill=white,inner sep=2pt},
white2/.style={circle,thick,draw=white,fill=white,inner sep=2pt},
black/.style={circle,draw=black,fill=black,inner sep=2pt},
connection/.style={draw=gray,thick,gray,auto}
]
\scriptsize

\begin{scope}
\clip (0,0) rectangle (10,10);

\fill[pattern color=\faceAfill,pattern=north west lines] (0,0) rectangle (10,10);
\draw[very thick,black,decoration={
    markings,
    mark=at position 0.9  with {\arrow{>}}},
    postaction={decorate},double=red]   
(3,2) -- (4.5,4.33);
\fill[white] (2,7) --(3,2) -- (5,5);
\fill[pattern color=\faceAfill,pattern=north west lines] (2,7) --(3,2) -- (5,5);

\draw[very thick,blue,decoration={
    markings,
    mark=at position 0.9  with {\arrow{>}}},
    postaction={decorate},double=red]   
(3,2) -- (2.5,4.5);
\fill[white] (-3,3) --(3,2) -- (2,7) ;
\fill[pattern color=\faceAfill,pattern=north west lines] (-3,3) --(3,2) -- (2,7) ;
\draw[very thick,blue,decoration={
    markings,
    mark=at position 0.9  with {\arrow{>}}},
    postaction={decorate},double=red]   
(5,5) -- (7.3,5);
\fill[\faceBfill] (5,5) rectangle (8,8);

\filldraw[red,fontscale=2] (5.4,2.8) circle (0pt) node[anchor=east] {$c$};
\filldraw[red,fontscale=2] (5.9,4) circle (0pt) node[anchor=east] {$c'$};

\draw[-{>[flex=0.75]},>=stealth,red,thick] (4.66,2.37) arc (00:75:1cm);
\draw[-{>[flex=0.75]},>=stealth,red,thick] (4.5,4.3) arc (240:359:0.7cm);

\draw (2,7)  node (b1)     [black] {};
\draw (3,2)  node (w1)     [black] {};
\coordinate (w1prim)      at (13,2);
\coordinate (w1bis)       at (3,12);
\coordinate (b1bis)       at (2,-3);

\draw (7,3)  node (bb1)    [black] {};
\coordinate (bb1prim) at (-3,3);

\draw (5,5)  node (AA)     [black] {};
\draw (8,5)  node (BA)     [black] {};
\draw (5,8)  node (AB)     [black] {};
\draw (8,8)  node (BB)     [black] {};

\draw[connection]         (w1)      to  (AA);
\draw[connection]         (AA)      to  (AB);
\draw[connection]         (AB)      to  (BB);
\draw[connection]         (BB)      to  (BA);
\draw[connection]         (BA)      to  (AA);

\draw[connection]         (b1)       to  (w1);
\draw[connection]         (b1)       to  (w1bis);
\draw[connection]         (bb1)      to  (w1prim);

\draw[connection]  (w1)      to  (bb1);
\draw[connection]         (w1)      to  (bb1prim);
\draw[connection]         (w1)      to  (b1bis);

\end{scope}

\draw[very thick,decoration={
    markings,
    mark=at position 0.666  with {\arrow{>}}},
    postaction={decorate}]  
(0,0) -- (10,0);

\draw[very thick,decoration={
    markings,
    mark=at position 0.666  with {\arrow{>}}},
    postaction={decorate}]  
(0,10) -- (10,10);

\draw[very thick,decoration={
    markings,
    mark=at position 0.666  with {\arrow{>>}}},
    postaction={decorate}]  
(10,0) -- (10,10);

\draw[very thick,decoration={
    markings,
    mark=at position 0.666  with {\arrow{>>}}},
    postaction={decorate}]  
(0,0) -- (0,10);
\end{tikzpicture}
\caption{The oriented corner $c$ (red arrow) equivalent to the root (the black arrow) of a map. The first corner in the root face of the map following $c$ is labelled by $c'$ (red arrow). 
By deleting the root edge the map splits into two new maps. 
The oriented corners $c$ and $c'$ are contained in two oriented corners of the new maps. They give the roots of those maps (the blue arrows).}
\label{deletion}
\end{figure}
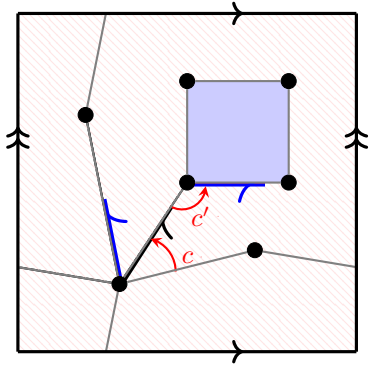

We classify the root edges of maps. Let $f$ be the number of faces of a map $M$ with the root vertex $e$;
\begin{enumerate}
\item $e$ is called  a \emph{bridge} if $M\setminus e$ is not connected,
\item otherwise $M\setminus e$ is connected and $e$ is called
\begin{itemize}
\item a \emph{border} if the number of faces in $M\setminus e$ is equal to $f-1$,
\item a \emph{twisted edge} if the number of faces in $M\setminus e$ is equal to $f$,
\item a \emph{handle} if the number of faces in $M\setminus e$ is equal to $f+1$.
\end{itemize}
\end{enumerate}

\begin{rem}
A leaf (\ie an edge connecting a vertex of degree 1) is considered as a bridge.
\end{rem}

\begin{de}\cite[Definition 4.1]{Lacroix} 
\label{measure} 
For a rooted map $M$, an invariant $\eta (M)$ is defined inductively as
follows.
\begin{enumerate}
\item If $M$ has no edges then $\eta (M) =0$.
\item Otherwise $M$ has the root edge $e$,
\begin{itemize}
\item $\eta (M) = \eta(M_1) +\eta (M_2)$ if $e$ is a bridge, while $M_1$ and $M_2$ are the connected components of $M\setminus e$,
\item $\eta (M) = \eta(M\setminus e )$ if $e$ is a border,
\item $\eta (M) = \eta(M\setminus e ) +1$ if $e$ is a twisted edge, 
\item if $e$ is a handle, there exists a unique map $M'$ with the root edge $e'$ constructed by \emph{twisting} the edge $e$ in $M$, in such a way that $e'$ is a handle and the maps $M\setminus e$, $M'\setminus e $ are equal. 
In this case we require that 
\begin{equation*}
\left\lbrace \eta (M), \eta (M')  \right\rbrace = 
\left\lbrace \eta (M\setminus e), \eta (M\setminus e) +1  \right\rbrace .
\end{equation*}
At most one of the maps $M$, $M'$ is orientable. For such a map $M$ we require $\eta (M)=\eta (M\setminus e)$.
\end{itemize}
\end{enumerate}

We call such an invariant a \emph{measure of non-orientability}.
\end{de}

Observe that the above definition introduces the whole family of measures of non-orientability $\eta$ 
and among of them there is no \emph{canonical measure of non-orientability}.

\begin{rem}
For a given rooted map $M$ 
\begin{equation*}
\eta \left(  M \right) = 0 \text{ if and only if } M \text{ is orientable}.
\end{equation*} 
\noindent
Indeed, removing twisted edges or handles during the root-deletion procedure are the only possibilities of increasing the recursively-defined statistic $\eta$. 
An orientable map does not have any twisted edges (a map with a twisted edge is embedded in a surface which contains the Möbius strip, hence is nonorientable). 
The recursive definition of $\eta$ guarantees that removing handles from an orientable map does not increase the statistic $\eta$. Hence for an orientable map $M$, we have $\eta (M) =0$. 
A reverse analysis or a simple induction on the number of edges provides the reverse implication. 
\end{rem}

\begin{de}
\label{measure1}
For a rooted $\mu$-list of maps $M =M_1 ,\ldots M_k$ we define a \emph{measure of non-orientability} $\eta$ of $M$ by 
\begin{equation*}
\eta\left( M \right)  := \eta_1 \left( M_1 \right) +\cdots + \eta_k \left( M_k \right)
\end{equation*}
\noindent
for any measures of non-orientability $\eta_i$ from \cref{measure}.
\end{de}

\subsection{Unhandled and unicellular maps.}

\begin{de}
The rooted map $M$ is called \emph{unhandled} if by iteratively performing the root-deletion process (see \cref{rootd}) it does not have any handles. 
The map $M$ is called \emph{unicellular} if it has only one face.
\end{de}

From now on we fix one of measures of non-orientability~$\eta$ of the class of maps. Do\l{}\oldk{e}ga \cite[Section 4]{Dol17} showed that for such a measure $\eta$ the polynomial $H_\eta$ given by the sum 
\begin{equation*}
\left(  H_\eta \right)_{\pi,\sigma}^\lambda := \sum_{M \in \Ml} \beta^{\eta (M)}
\end{equation*}
\noindent
has degree at most equal to $n +1 -\ell (\pi)-\ell (\sigma ) $ and the leading coefficient is enumerated by \emph{unhandled} \emph{unicellular} maps. In particular, $\left(  H_\eta \right)_{\pi,\sigma}^\lambda$ may achieve this bound of the degree only if $\lambda =(n)$. 
He also showed that the aforementioned leading coefficient is also enumerated by \emph{oriented} maps with arbitrary face-type, namely 
\begin{equation*}
\Big\vert M \in \M  : M \emph{ is orientable }\Big\vert = \Big\vert M \in M_{\pi,\sigma}^{(n)} : M \emph{ is unhandled }  \Big\vert .
\end{equation*}
\noindent
In fact, there is an explicit bijection between those two families of maps. 
Do\l{}\oldk{e}ga proved \cite[Theorem 1.4]{Dol17} that for the statistic $\eta$
\begin{equation*}
h_{\pi,\sigma}^{(n)} (\beta ) =
\sum_{M \in M_{\pi ,\sigma}^{(n)}} \beta^{\eta (M)}
\end{equation*}
\noindent
holds true for $\beta \in \{-1,0,1\}$, moreover for $M \in M_{\pi ,\sigma}^{(n)}$ the statistic $\eta (M) =0$ vanishes if and only if $M$ is \emph{orientable}; furthermore $\eta (M) = n +1 -\ell (\pi)-\ell (\sigma ) $ if and only if $M$ is \emph{unhandled} and \emph{unicellular}.

\smallskip
\smallskip

The result of Do\l{}\oldk{e}ga is easily transferable to the context of $\mu$-lists of maps. 
Let us choose the measures of non-orientability $\eta_i$ for $i \in [k], k=\ell (\mu )$, which form the measure $\eta$ as it is described in \cref{measure1}. 

\begin{lem}
\label{hak}
For the statistic $\eta$, the polynomial $ \left(  H_\eta \right)_{\pi,\sigma}^{\lambda; \bullet}$ given by the sum 
\begin{equation}
\label{polynomial}
\left(  H_\eta \right)_{\pi,\sigma}^{\lambda; \bullet} :=\sum_{\mu : \lambda \preceq \mu} 
\left(  H_\eta \right)_{\pi,\sigma}^{\lambda; \mu}
\end{equation}
\noindent 
where
\begin{equation}
\label{polynomial2}
\left(  H_\eta \right)_{\pi,\sigma}^{\lambda; \mu} (\beta )=\sum_{M \in \Mlp} \beta^{\eta (M)}
\end{equation}
\noindent
is of degree at most $d (\pi, \sigma ;\lambda ) $. Moreover, a $\mu$-lists of maps $M$ contributes to the ground term if and only if $M$ is a list of orientable maps. The $\mu$-lists of maps $M$ contributes to the leading coefficient if and only if $M$ is a list of unicellular and unhandled maps, in particular $\mu = \lambda$.
\end{lem}

\begin{proof}
Each $M  = (M_1 , \ldots , M_k ) \in \Mlp$ decompose into a list of maps $M_i \in M^{\lambda_{\vert \mu_i}}_{{\pi_{\vert \mu_i}},{\sigma_{\vert \mu_i}}}$ for some partitions $\pi_{\vert \mu_i} ,\sigma_{\vert \mu_i} ,\lambda_{\vert \mu_i} \vdash \mu_i$ satisfying
\begin{equation*}
\bigcup_{i=1}^k \pi_{\vert \mu_i} =\pi , \quad 
\bigcup_{i=1}^k \sigma_{\vert \mu_i} =\sigma , \quad 
\bigcup_{i=1}^k \lambda_{\vert \mu_i} =\lambda .
\end{equation*}
\noindent
We denote by $\mathcal{P}_\pi^\mu$ the set of lists of partitions $\left( \pi_{\vert \mu_1} ,\ldots , \pi_{\vert \mu_k}\right) $, where $\pi_{\vert \mu_i} \vdash \mu_i $ and 
\begin{equation*}
\bigcup_{i=1}^k \pi_{\vert \mu_i} =\pi .
\end{equation*}
\noindent
Observe, that \cref{polynomial} can be rewritten in such a way:
\begin{equation*}
\sum_{M \in \Mlp} \beta^{\eta (M)} = 
\sum_{\substack{(\pi^1 ,\ldots , \pi^k) \in \mathcal{P}_\pi^\mu \\ (\sigma^1 ,\ldots ,\sigma^k) \in \mathcal{P}_\sigma^\mu \\ (\lambda^1 ,\ldots ,\lambda^k) \in \mathcal{P}_\lambda^\mu}}
\prod_{i=1}^k 
\sum_{M \in M^{\lambda^i}_{{\pi^i},{\sigma^i}}} \beta^{\eta_i (M)} .
\end{equation*}
\noindent
We use the result of Do\l{}\oldk{e}ga for each most right side sum separately. 
Each such a sum  has degree at most equal to $\mu_i +1 -\ell (\pi^i)-\ell (\sigma^i ) $ and the top-degree coefficient is enumerated by unhandled unicellular maps. 
Since
\begin{equation*}
n +\ell ( \mu ) -\ell (\pi )-\ell (\sigma ) = \sum_{i=1}^k \Big( \mu_i +1 -\ell (\pi^i)-\ell (\sigma^i ) \Big) ,
\end{equation*}
\noindent
we conclude that \cref{polynomial} has degree at most equal to $d (\pi, \sigma ;\lambda ) $ and the top-degree coefficient is enumerated by $\mu$-lists of unhandled unicellular maps. 
\end{proof}

\begin{cor}
\label{wniosek}
For three given partitions $\pi,\sigma ,\lambda\vdash n$ we have 
\begin{equation*}
\Big\vert M \in \Mp  : M \emph{ is orientable }\Big\vert = \Big\vert M \in M_{\pi,\sigma}^{\mu ;\mu} : M \emph{ is unhandled }  \Big\vert .
\end{equation*}
\end{cor}

\begin{proof}
Fix a list $M \in M_{\pi,\sigma}^{\mu ;\mu} $ of unhandled and unicellular maps. 
For each connected component of $M$ we use the aforementioned bijection between such maps and oriented maps with arbitrary face-type given by Do\l{}\oldk{e}ga \cite[Corollary 3.10]{Dol17}. 
We get a $\mu$-list of orientable maps with arbitrary face type.
\end{proof}

\subsection{Measure of non-bipartiteness for matchings.}
\label{problemy}
The hypothetical statistic $\wt_\lambda$ from the Matchings-Jack Conjecture should be a marker of non-bipartiteness for matchings. 
Naturally, matchings correspond to lists of maps, in particular bipartite matching to lists of oriented maps. 

The naive thought how the statistic $\wt_\lambda$ should be defined is to adapt the measure of non-orientability introduced by La Croix by the correspondence between matchings and collections of maps given by \cref{coro}. 
Regretfully, the measure introduced by La Croix is defined for lists of rooted maps, however there is no canonical way to create such a list from an element of $\GMlp$. 

However, there is one special class of matchings, which may be identified with lists of rooted maps, namely $\mathcal{G}_{\pi,\sigma}^{\lambda ; \lambda }$. When the number of faces is equal to the number of connected components, numbering and rooting \emph{faces} overlap with numbering and rooting \emph{components}. For a fixed measure of non-orientability $\eta$ we define
\begin{align*} 
\stat : \quad \mathcal{G}_{\pi,\sigma}^{\lambda ; \lambda } 
&\quad\longrightarrow\quad\left[ d (\pi, \sigma ;\lambda )\right] \\
\delta 
&\quad\longmapsto\quad \stat\left(\delta \right) := \eta \left( M_\delta \right) 
\end{align*} 
\noindent
For given partitions $\lambda, \pi ,\sigma \vdash n$ we define the following polynomial
\begin{equation}
\label{stt}
\left( G_\eta \right)_{\pi,\sigma}^{\lambda ;\lambda} := 
\sum_{\delta \in \Gll} \beta^{\stat (\delta )}.
\end{equation}

\begin{de}
\label{unhandledM}
We say that a matching $\delta \in \Gll$ is \emph{unhandled} if the corresponding map $M_\delta \in M_{\pi, \sigma} ^{\lambda ; \lambda}$ is so.
\end{de}

\begin{lem}
\label{lemlem}
For any triple of partitions $\pi,\sigma ,\lambda \vdash n$ the corresponding polynomial $\left( G_\eta \right)_{\pi,\sigma}^{\lambda ;\lambda} $ is of degree at most $d (\pi, \sigma ;\lambda ) $. Moreover, the matching $\delta$ contributes to the ground term if and only if $\delta$ is bipartite. The matching $\delta$ contributes to the leading coefficient if and only if $\delta$ is an unhandled matching. 

Moreover, the top-degree coefficient may be enumerated in two different manners:
\begin{equation*}
\Big\vert \delta \in \Gll : \delta \emph{ is unhandled }  \Big\vert =
\sum_{\nu : \nu \preceq \lambda} \dfrac{z_\lambda}{z_\nu} \Big\vert \delta \in \mathcal{G}_{\pi, \sigma}^{\nu ;\lambda}  : \delta\emph{ is bipartite }\Big\vert .
\end{equation*}
\end{lem}

\begin{proof}
Observe that for fixed measure of non-orientability $\eta$ polynomials $\left( G_\eta \right)_{\pi,\sigma}^{\lambda ;\lambda} $ and $\left( H_\eta \right)_{\pi,\sigma}^{\lambda ;\lambda} $ are equal. 
The first statement follows immediately from \cref{hak}. 
The second statement is an easy conclusion of \cref{wniosek} and relation given in \cref{strange2}.
\end{proof}

\section{Jack characters and structure constants}
\label{5555}

\subsection{Jack characters}
We expand Jack polynomial in the basis of power-sum symmetric functions:
\begin{equation} 
\label{eq:definition-theta}
J^{(\alpha)}_\lambda = \sum_\mu \theta^{(\alpha)}_\mu(\lambda)\ p_\mu. 
\end{equation}
The above sum runs over partitions $\mu$ such that $|\mu|=|\lambda|$. 
The coefficient $\theta^{(\alpha)}_\mu(\lambda)$ is called \emph{unnormalized Jack character}.

Jack characters $\theta^{(\alpha)}_\mu$ provide a kind of \emph{dual} information about the Jack polynomials.
Better understanding of the combinatorics of Jack characters may lead to a better understanding of Jack polynomials themselves. 
This kind of approach may be traced back to the work of Kerov and Olshanski \cite{MR1288389}. 
For a fixed conjugacy class~$\mu$ they considered characters of the symmetric group evaluated on $\mu$. 
This is opposite to the usual way of viewing the characters of the symmetric groups, namely to fix the representation $\lambda$ and to consider the character as a function of the conjugacy class~$\mu$.
Lassalle \cite{MR2424904,MR2562783} adapted idea of Kerov and Olshanski to the framework of Jack characters.

As Jack symmetric functions $\left( J^{(\alpha)}_\lambda \right)_\lambda $ form a basis of the symmetric functions, the functions $\left( \theta^{(\alpha)}_\mu\right)_{\mu \vdash n} $ form a basis of the algebra of functions on Young diagrams with $n$ boxes \cite[Proposition 4.1]{Fer12}. 
Do\l{}\oldk{e}ga and F\'eray \cite[Appendix B.2]{Dolega2014} showed that the coefficients appearing in the expansion of a pointwise product of two unnormalized Jack characters in the unnormalized Jack character basis coincide with the connection coefficients from \cref{star}, namely
\begin{equation*}
\theta^{(\alpha)}_\pi \cdot\theta^{(\alpha)}_\sigma=\sum_{\mu\vdash n} c_{\pi,\sigma}^\mu \theta^{(\alpha)}_\mu .
\end{equation*}
\noindent
for all triples of partitions $\pi,\sigma,\mu \vdash n$. This observation encourages us to look more closely into the field of connection coefficients via the context of Jack characters.

\subsection{Normalized Jack characters}
\label{sec:jack-first-definition}
We define Jack characters $\Ch_\pi$ by a choice of the normalization of $\theta^{(\alpha)}_\pi$. 
We will use the normalization introduced by Do\l{}\oldk{e}ga and F\'eray \cite{Dolega2014} which offers some advantages over the original normalization of Lassalle. 
Therefore, with the right choice of the multiplicative constant, the unnormalized Jack character
$\theta_{\lambda}^{(\alpha)}(\pi)$ from \cref{eq:definition-theta} 
becomes the \emph{normalized Jack character $\Ch^{(\alpha)}_\pi(\lambda)$}, defined as follows.

\begin{de}
\label{def:jack-character-classical}
For a given number $\alpha>0$ and a partition $\pi$, the \emph{normalized Jack character
$\Ch_\pi^{(\alpha)}(\lambda)$}  is defined by:
\begin{equation*}
\Ch_{\pi}^{(\alpha)}(\lambda):=
\begin{cases}
{\dfrac{1}{\sqrt{\alpha}}}^{|\pi|+\ell(\pi)}
\dbinom{|\lambda|-|\pi|+m_1(\pi)}{m_1(\pi)}
\ z_\pi \ \theta^{(\alpha)}_{\pi \cup 1^{|\lambda|-|\pi|}}(\lambda)
&\text{if }|\lambda| \ge |\pi| ,\\
0 & \text{if }|\lambda| < |\pi|,
\end{cases}
\end{equation*}
where $z_\pi$ is the standard numerical factor, and $\cup$ denotes concatenation of two partitions, see \cref{p1}.
The choice of an empty partition $\pi=\emptyset$ is acceptable; in this case
$\Ch_\emptyset^{(\alpha)}(\lambda)=1$.
\end{de}

\subsection{Structure constants}
\emph{Structure constants} $\g{\mu}$ of Jack characters are defined by expansion of the \emph{pointwise product} of two Jack characters in the basis of Jack characters:
\begin{equation*}
\Ch_\pi \cdot \Ch_\sigma =\sum_\mu \g{\mu} (\delta) \Ch_\mu  .
\end{equation*}

Explicit motivation for studying such quantities comes from a special choice of the deformation parameter $\alpha =1$, when Jack polynomials coincide with Schur polynomials.
In this case, Frobenius duality ensures that the structure constants coincide with the \emph{connection coefficients for the symmetric groups} \cite{IK99}.

Do\l{}\oldk{e}ga and F\'eray proved \cite[Theorem 1.4]{Dolega2014} that each structure constant $g^\mu_{\pi,\sigma}$ is a polynomial in the variable $\delta :=  \sqrt{\alpha} - \dfrac{1}{\sqrt{\alpha}}$ of degree bounded as follows:
\begin{equation}
\label{estimation}
\deg_\delta g^\mu_{\pi,\sigma} \leq \min_{i=1,2,3} \Big( n_i ( \pi )+  n_i ( \sigma ) -  n_i ( \mu )\Big), 
\end{equation}
\noindent
where 
\begin{align*} 
n_1 ( \pi ) &= |\pi |+\ell (\pi), \\
n_2 ( \pi ) &= |\pi |-\ell (\pi), \\
n_3 ( \pi ) &= |\pi |-\ell (\pi) +m_1 (\pi ).
\end{align*}  
\noindent
For example, we have
\begin{align*} 
\Ch_3 \Ch_2 = &
6 \delta \Ch_3+\Ch_{3,2}+  6 \Ch_{2,1} + 6 \Ch_4, \\
\Ch_{3} \Ch_{3} =&
(6 \delta^2+3) \Ch_{3} + 9 \delta \Ch_{2,1} + 18 \delta \Ch_{4}  + 3 \Ch_{1,1,1} +\\
& +9 \Ch_{3,1} + 9 \Ch_{2,2} + 9 \Ch_5 + \Ch_{3,3}.
\end{align*}
The numerical computations, such as the ones above, suggest that the structure constants of Jack characters might have some algebraic and combinatorial structure, which was proposed in the following conjecture \cite[Conjecture 0.1]{Sniady2016}.

\begin{con}[Structure constants of Jack characters]
\label{conj:positivity-structure-coefficients}
For any partitions $\pi,\sigma,\mu$, the corresponding structure constant
\[ g^\mu_{\pi,\sigma}(\delta)\in\Q[\delta] \]
is a polynomial with non-negative integer coefficients.
\end{con}

\subsection{The top-degree part of structure constants.}
We present an explicit formula for the top-degree part of structure constants of Jack characters. 

Let us recall that we present an oriented map as a graph on the plane with a fixed cyclic order of outgoing edges together with a choice of the root, see \cref{Orientedmap}. By convention we fixed the counter-clockwise orientation around vertices or, equivalently, the clockwise orientation of the page, see \cref{Orientedmap}. Similarly, we will present a $\mu$-collections of maps.  

Let us recall that $\MOp$ denotes the set of all $\mu$-lists of bipartite rooted and oriented maps which satisfy
\begin{equation*}
 \Lambda_\mathcal{W} (M ) = \pi  \quad\quad\text{and}\quad\quad
\Lambda_\mathcal{B} (M ) = \sigma ,
\end{equation*}
\noindent
see \cref{main-definition}. 

\begin{figure}
\centering
\begin{subfigure}[b]{0.999\textwidth}
\center
\begin{tikzpicture}[scale=0.25]

\draw[gray, thick,decoration={markings,mark=at position 0.3  with {\arrow{>}}},postaction={decorate},rounded corners=7pt]  (10,10) --  (11,11) -- (13,11) -- (14,10) node [midway, above, sloped] (TextNode) {};
\draw[gray, thick,rounded corners=7pt]  (10,10) --  (11,9) -- (13,9) -- (14,10) node [midway, below, sloped] (TextNode) {};
\draw[gray, thick,rounded corners=7pt]  (10,10)  -- (14,10) node [midway, below, sloped] (TextNode) {};

\draw[gray, thick,decoration={markings,mark=at position 0.3  with {\arrow{>}}},postaction={decorate},rounded corners=7pt]  (10,6) --  (11,7) -- (13,7) -- (14,6) node [midway, above, sloped] (TextNode) {};
\draw[gray, thick,rounded corners=7pt] (10,6) --  (11,5) -- (13,5) -- (14,6) node [midway, below, sloped] (TextNode) {};
\draw[gray, thick,rounded corners=7pt] (12,6)  -- (14,6) node [midway, below, sloped] (TextNode) {};

\filldraw[black] (10,10) circle (8pt) node[anchor=east] {$1$};
\filldraw[black] (10,6) circle (8pt) node[anchor=east] {$2$};
\filldraw[black] (12,6) circle (8pt) node[anchor=east] {};

\fill[white] (14,10) circle (8pt) node[anchor=east] {};
\draw[black] (14,10) circle (8pt) node[anchor=west] {};
\fill[white] (14,6) circle (8pt) node[anchor=east] {};
\draw[black] (14,6) circle (8pt) node[anchor=west] {};

\begin{scope}[shift={(8,0)}]
\draw[gray, thick,decoration={markings,mark=at position 0.3  with {\arrow{>}}},postaction={decorate},rounded corners=7pt]  (10,10) --  (11,11) -- (13,11) -- (14,10) node [midway, above, sloped] (TextNode) {};
\draw[gray, thick,rounded corners=7pt]  (10,10) --  (11,9) -- (13,9) -- (14,10) node [midway, below, sloped] (TextNode) {};
\draw[gray, thick,rounded corners=7pt]  (10,10)  -- (14,10) node [midway, below, sloped] (TextNode) {};

\draw[gray, thick,rounded corners=7pt] (10,6) --  (11,7) -- (13,7) -- (14,6) node [midway, above, sloped] (TextNode) {};
\draw[gray, thick,decoration={markings,mark=at position 0.3  with {\arrow{>}}},postaction={decorate},rounded corners=7pt]  (10,6) --  (11,5) -- (13,5) -- (14,6) node [midway, below, sloped] (TextNode) {};
\draw[gray, thick,rounded corners=7pt] (12,6)  -- (14,6) node [midway, below, sloped] (TextNode) {};

\filldraw[black] (10,10) circle (8pt) node[anchor=east] {$1$};
\filldraw[black] (10,6) circle (8pt) node[anchor=east] {$2$};
\filldraw[black] (12,6) circle (8pt) node[anchor=east] {};

\fill[white] (14,10) circle (8pt) node[anchor=east] {};
\draw[black] (14,10) circle (8pt) node[anchor=west] {};
\fill[white] (14,6) circle (8pt) node[anchor=east] {};
\draw[black] (14,6) circle (8pt) node[anchor=west] {};
\end{scope}

\begin{scope}[shift={(16,0)}]
\draw[gray, thick,decoration={markings,mark=at position 0.3  with {\arrow{>}}},postaction={decorate},rounded corners=7pt]  (10,10) --  (11,11) -- (13,11) -- (14,10) node [midway, above, sloped] (TextNode) {};
\draw[gray, thick,rounded corners=7pt]  (10,10) --  (11,9) -- (13,9) -- (14,10) node [midway, below, sloped] (TextNode) {};
\draw[gray, thick,rounded corners=7pt]  (10,10)  -- (14,10) node [midway, below, sloped] (TextNode) {};

\draw[gray, thick,rounded corners=7pt]  (10,6) --  (11,7) -- (13,7) -- (14,6) node [midway, above, sloped] (TextNode) {};
\draw[gray, thick,rounded corners=7pt] (10,6) --  (11,5) -- (13,5) -- (14,6) node [midway, below, sloped] (TextNode) {};
\draw[gray, thick,decoration={markings,mark=at position 0.5  with {\arrow{>}}},postaction={decorate},rounded corners=7pt]  (12,6)  -- (14,6) node [midway, below, sloped] (TextNode) {};

\filldraw[black] (10,10) circle (8pt) node[anchor=east] {$1$};
\filldraw[black] (10,6) circle (8pt) node[anchor=east] {$2$};
\filldraw[black] (12,6) circle (8pt) node[anchor=east] {};

\fill[white] (14,10) circle (8pt) node[anchor=east] {};
\draw[black] (14,10) circle (8pt) node[anchor=west] {};
\fill[white] (14,6) circle (8pt) node[anchor=east] {};
\draw[black] (14,6) circle (8pt) node[anchor=west] {};
\end{scope}

\begin{scope}[shift={(27,0)}]
\draw[gray, thick,decoration={markings,mark=at position 0.3  with {\arrow{>}}},postaction={decorate},rounded corners=7pt]  (10,10) --  (11,11) -- (13,11) -- (14,10) node [midway, above, sloped] (TextNode) {};
\draw[gray, thick,rounded corners=7pt]  (10,10) --  (11,10.3) -- (13,9) -- (14,10) node [midway, below, sloped] (TextNode) {};
\fill[white] (12,9.6) circle (8pt) node[anchor=east] {};
\draw[gray, thick,rounded corners=7pt]  (10,10) --  (11,9) -- (13,10.3) -- (14,10) node [midway, below, sloped] (TextNode) {};

\draw[gray, thick,decoration={markings,mark=at position 0.3  with {\arrow{>}}},postaction={decorate},rounded corners=7pt]  (10,6) --  (11,7) -- (13,7) -- (14,6) node [midway, above, sloped] (TextNode) {};
\draw[gray, thick,rounded corners=7pt] (10,6) --  (11,5) -- (13,5) -- (14,6) node [midway, below, sloped] (TextNode) {};
\draw[gray, thick,rounded corners=7pt] (12,6)  -- (14,6) node [midway, below, sloped] (TextNode) {};

\filldraw[black] (10,10) circle (8pt) node[anchor=east] {$1$};
\filldraw[black] (10,6) circle (8pt) node[anchor=east] {$2$};
\filldraw[black] (12,6) circle (8pt) node[anchor=east] {};

\fill[white] (14,10) circle (8pt) node[anchor=east] {};
\draw[black] (14,10) circle (8pt) node[anchor=west] {};
\fill[white] (14,6) circle (8pt) node[anchor=east] {};
\draw[black] (14,6) circle (8pt) node[anchor=west] {};
\end{scope}

\begin{scope}[shift={(35,0)}]
\draw[gray, thick,decoration={markings,mark=at position 0.3  with {\arrow{>}}},postaction={decorate},rounded corners=7pt]  (10,10) --  (11,11) -- (13,11) -- (14,10) node [midway, above, sloped] (TextNode) {};
\draw[gray, thick,rounded corners=7pt]  (10,10) --  (11,10.3) -- (13,9) -- (14,10) node [midway, below, sloped] (TextNode) {};
\fill[white] (12,9.6) circle (8pt) node[anchor=east] {};
\draw[gray, thick,rounded corners=7pt]  (10,10) --  (11,9) -- (13,10.3) -- (14,10) node [midway, below, sloped] (TextNode) {};

\draw[gray, thick,rounded corners=7pt] (10,6) --  (11,7) -- (13,7) -- (14,6) node [midway, above, sloped] (TextNode) {};
\draw[gray, thick,decoration={markings,mark=at position 0.3  with {\arrow{>}}},postaction={decorate},rounded corners=7pt]  (10,6) --  (11,5) -- (13,5) -- (14,6) node [midway, below, sloped] (TextNode) {};
\draw[gray, thick,rounded corners=7pt] (12,6)  -- (14,6) node [midway, below, sloped] (TextNode) {};

\filldraw[black] (10,10) circle (8pt) node[anchor=east] {$1$};
\filldraw[black] (10,6) circle (8pt) node[anchor=east] {$2$};
\filldraw[black] (12,6) circle (8pt) node[anchor=east] {};

\fill[white] (14,10) circle (8pt) node[anchor=east] {};
\draw[black] (14,10) circle (8pt) node[anchor=west] {};
\fill[white] (14,6) circle (8pt) node[anchor=east] {};
\draw[black] (14,6) circle (8pt) node[anchor=west] {};
\end{scope}

\begin{scope}[shift={(43,0)}]
\draw[gray, thick,decoration={markings,mark=at position 0.3  with {\arrow{>}}},postaction={decorate},rounded corners=7pt]  (10,10) --  (11,11) -- (13,11) -- (14,10) node [midway, above, sloped] (TextNode) {};
\draw[gray, thick,rounded corners=7pt]  (10,10) --  (11,10.3) -- (13,9) -- (14,10) node [midway, below, sloped] (TextNode) {};
\fill[white] (12,9.6) circle (8pt) node[anchor=east] {};
\draw[gray, thick,rounded corners=7pt]  (10,10) --  (11,9) -- (13,10.3) -- (14,10) node [midway, below, sloped] (TextNode) {};

\draw[gray, thick,rounded corners=7pt]  (10,6) --  (11,7) -- (13,7) -- (14,6) node [midway, above, sloped] (TextNode) {};
\draw[gray, thick,rounded corners=7pt] (10,6) --  (11,5) -- (13,5) -- (14,6) node [midway, below, sloped] (TextNode) {};
\draw[gray, thick,decoration={markings,mark=at position 0.5  with {\arrow{>}}},postaction={decorate},rounded corners=7pt]  (12,6)  -- (14,6) node [midway, below, sloped] (TextNode) {};

\filldraw[black] (10,10) circle (8pt) node[anchor=east] {$1$};
\filldraw[black] (10,6) circle (8pt) node[anchor=east] {$2$};
\filldraw[black] (12,6) circle (8pt) node[anchor=east] {};

\fill[white] (14,10) circle (8pt) node[anchor=east] {};
\draw[black] (14,10) circle (8pt) node[anchor=west] {};
\fill[white] (14,6) circle (8pt) node[anchor=east] {};
\draw[black] (14,6) circle (8pt) node[anchor=west] {};
\end{scope}

\begin{scope}[shift={(0,-15)}]
\draw[gray, thick,decoration={markings,mark=at position 0.3  with {\arrow{>}}},postaction={decorate},rounded corners=7pt]  (10,10) --  (11,11) -- (13,11) -- (14,10) node [midway, above, sloped] (TextNode) {};
\draw[gray, thick,rounded corners=7pt]  (10,10) --  (11,9) -- (13,9) -- (14,10) node [midway, below, sloped] (TextNode) {};
\draw[gray, thick,rounded corners=7pt]  (10,10)  -- (14,10) node [midway, below, sloped] (TextNode) {};
\filldraw[black] (10,10) circle (8pt) node[anchor=east] {$2$};
\fill[white] (14,10) circle (8pt) node[anchor=east] {};
\draw[black] (14,10) circle (8pt) node[anchor=west] {};
\end{scope}

\begin{scope}[shift={(0,-7)}]
\draw[gray, thick,decoration={markings,mark=at position 0.3  with {\arrow{>}}},postaction={decorate},rounded corners=7pt]  (10,6) --  (11,7) -- (13,7) -- (14,6) node [midway, above, sloped] (TextNode) {};
\draw[gray, thick,rounded corners=7pt] (10,6) --  (11,5) -- (13,5) -- (14,6) node [midway, below, sloped] (TextNode) {};
\draw[gray, thick,rounded corners=7pt] (12,6)  -- (14,6) node [midway, below, sloped] (TextNode) {};
\filldraw[black] (10,6) circle (8pt) node[anchor=east] {$1$};
\filldraw[black] (12,6) circle (8pt) node[anchor=east] {};
\fill[white] (14,6) circle (8pt) node[anchor=east] {};
\draw[black] (14,6) circle (8pt) node[anchor=west] {};
\end{scope}

\begin{scope}[shift={(8,-15)}]
\draw[gray, thick,decoration={markings,mark=at position 0.3  with {\arrow{>}}},postaction={decorate},rounded corners=7pt]  (10,10) --  (11,11) -- (13,11) -- (14,10) node [midway, above, sloped] (TextNode) {};
\draw[gray, thick,rounded corners=7pt]  (10,10) --  (11,9) -- (13,9) -- (14,10) node [midway, below, sloped] (TextNode) {};
\draw[gray, thick,rounded corners=7pt]  (10,10)  -- (14,10) node [midway, below, sloped] (TextNode) {};
\filldraw[black] (10,10) circle (8pt) node[anchor=east] {$2$};
\fill[white] (14,10) circle (8pt) node[anchor=east] {};
\draw[black] (14,10) circle (8pt) node[anchor=west] {};
\end{scope}

\begin{scope}[shift={(8,-7)}]
\draw[gray, thick,rounded corners=7pt] (10,6) --  (11,7) -- (13,7) -- (14,6) node [midway, above, sloped] (TextNode) {};
\draw[gray, thick,decoration={markings,mark=at position 0.3  with {\arrow{>}}},postaction={decorate},rounded corners=7pt]  (10,6) --  (11,5) -- (13,5) -- (14,6) node [midway, below, sloped] (TextNode) {};
\draw[gray, thick,rounded corners=7pt] (12,6)  -- (14,6) node [midway, below, sloped] (TextNode) {};
\filldraw[black] (10,6) circle (8pt) node[anchor=east] {$1$};
\filldraw[black] (12,6) circle (8pt) node[anchor=east] {};
\fill[white] (14,6) circle (8pt) node[anchor=east] {};
\draw[black] (14,6) circle (8pt) node[anchor=west] {};
\end{scope}

\begin{scope}[shift={(16,-15)}]
\draw[gray, thick,decoration={markings,mark=at position 0.3  with {\arrow{>}}},postaction={decorate},rounded corners=7pt]  (10,10) --  (11,11) -- (13,11) -- (14,10) node [midway, above, sloped] (TextNode) {};
\draw[gray, thick,rounded corners=7pt]  (10,10) --  (11,9) -- (13,9) -- (14,10) node [midway, below, sloped] (TextNode) {};
\draw[gray, thick,rounded corners=7pt]  (10,10)  -- (14,10) node [midway, below, sloped] (TextNode) {};
\filldraw[black] (10,10) circle (8pt) node[anchor=east] {$2$};
\fill[white] (14,10) circle (8pt) node[anchor=east] {};
\draw[black] (14,10) circle (8pt) node[anchor=west] {};
\end{scope}

\begin{scope}[shift={(16,-7)}]
\draw[gray, thick,rounded corners=7pt]  (10,6) --  (11,7) -- (13,7) -- (14,6) node [midway, above, sloped] (TextNode) {};
\draw[gray, thick,rounded corners=7pt] (10,6) --  (11,5) -- (13,5) -- (14,6) node [midway, below, sloped] (TextNode) {};
\draw[gray, thick,decoration={markings,mark=at position 0.5  with {\arrow{>}}},postaction={decorate},rounded corners=7pt]  (12,6)  -- (14,6) node [midway, below, sloped] (TextNode) {};
\filldraw[black] (10,6) circle (8pt) node[anchor=east] {$1$};
\filldraw[black] (12,6) circle (8pt) node[anchor=east] {};
\fill[white] (14,6) circle (8pt) node[anchor=east] {};
\draw[black] (14,6) circle (8pt) node[anchor=west] {};
\end{scope}

\begin{scope}[shift={(27,-15)}]
\draw[gray, thick,decoration={markings,mark=at position 0.3  with {\arrow{>}}},postaction={decorate},rounded corners=7pt]  (10,10) --  (11,11) -- (13,11) -- (14,10) node [midway, above, sloped] (TextNode) {};
\draw[gray, thick,rounded corners=7pt]  (10,10) --  (11,10.3) -- (13,9) -- (14,10) node [midway, below, sloped] (TextNode) {};
\fill[white] (12,9.6) circle (8pt) node[anchor=east] {};
\draw[gray, thick,rounded corners=7pt]  (10,10) --  (11,9) -- (13,10.3) -- (14,10) node [midway, below, sloped] (TextNode) {};
\filldraw[black] (10,10) circle (8pt) node[anchor=east] {$2$};
\fill[white] (14,10) circle (8pt) node[anchor=east] {};
\draw[black] (14,10) circle (8pt) node[anchor=west] {};
\end{scope}

\begin{scope}[shift={(27,-7)}]
\draw[gray, thick,decoration={markings,mark=at position 0.3  with {\arrow{>}}},postaction={decorate},rounded corners=7pt]  (10,6) --  (11,7) -- (13,7) -- (14,6) node [midway, above, sloped] (TextNode) {};
\draw[gray, thick,rounded corners=7pt] (10,6) --  (11,5) -- (13,5) -- (14,6) node [midway, below, sloped] (TextNode) {};
\draw[gray, thick,rounded corners=7pt] (12,6)  -- (14,6) node [midway, below, sloped] (TextNode) {};
\filldraw[black] (10,6) circle (8pt) node[anchor=east] {$1$};
\filldraw[black] (12,6) circle (8pt) node[anchor=east] {};
\fill[white] (14,6) circle (8pt) node[anchor=east] {};
\draw[black] (14,6) circle (8pt) node[anchor=west] {};
\end{scope}

\begin{scope}[shift={(35,-15)}]
\draw[gray, thick,decoration={markings,mark=at position 0.3  with {\arrow{>}}},postaction={decorate},rounded corners=7pt]  (10,10) --  (11,11) -- (13,11) -- (14,10) node [midway, above, sloped] (TextNode) {};
\draw[gray, thick,rounded corners=7pt]  (10,10) --  (11,10.3) -- (13,9) -- (14,10) node [midway, below, sloped] (TextNode) {};
\fill[white] (12,9.6) circle (8pt) node[anchor=east] {};
\draw[gray, thick,rounded corners=7pt]  (10,10) --  (11,9) -- (13,10.3) -- (14,10) node [midway, below, sloped] (TextNode) {};
\filldraw[black] (10,10) circle (8pt) node[anchor=east] {$2$};
\fill[white] (14,10) circle (8pt) node[anchor=east] {};
\draw[black] (14,10) circle (8pt) node[anchor=west] {};
\end{scope}

\begin{scope}[shift={(35,-7)}]
\draw[gray, thick,rounded corners=7pt] (10,6) --  (11,7) -- (13,7) -- (14,6) node [midway, above, sloped] (TextNode) {};
\draw[gray, thick,decoration={markings,mark=at position 0.3  with {\arrow{>}}},postaction={decorate},rounded corners=7pt]  (10,6) --  (11,5) -- (13,5) -- (14,6) node [midway, below, sloped] (TextNode) {};
\draw[gray, thick,rounded corners=7pt] (12,6)  -- (14,6) node [midway, below, sloped] (TextNode) {};
\filldraw[black] (10,6) circle (8pt) node[anchor=east] {$1$};
\filldraw[black] (12,6) circle (8pt) node[anchor=east] {};
\fill[white] (14,6) circle (8pt) node[anchor=east] {};
\draw[black] (14,6) circle (8pt) node[anchor=west] {};
\end{scope}

\begin{scope}[shift={(43,-15)}]
\draw[gray, thick,decoration={markings,mark=at position 0.3  with {\arrow{>}}},postaction={decorate},rounded corners=7pt]  (10,10) --  (11,11) -- (13,11) -- (14,10) node [midway, above, sloped] (TextNode) {};
\draw[gray, thick,rounded corners=7pt]  (10,10) --  (11,10.3) -- (13,9) -- (14,10) node [midway, below, sloped] (TextNode) {};
\fill[white] (12,9.6) circle (8pt) node[anchor=east] {};
\draw[gray, thick,rounded corners=7pt]  (10,10) --  (11,9) -- (13,10.3) -- (14,10) node [midway, below, sloped] (TextNode) {};
\filldraw[black] (10,10) circle (8pt) node[anchor=east] {$2$};
\fill[white] (14,10) circle (8pt) node[anchor=east] {};
\draw[black] (14,10) circle (8pt) node[anchor=west] {};
\end{scope}

\begin{scope}[shift={(43,-7)}]
\draw[gray, thick,rounded corners=7pt]  (10,6) --  (11,7) -- (13,7) -- (14,6) node [midway, above, sloped] (TextNode) {};
\draw[gray, thick,rounded corners=7pt] (10,6) --  (11,5) -- (13,5) -- (14,6) node [midway, below, sloped] (TextNode) {};
\draw[gray, thick,decoration={markings,mark=at position 0.5  with {\arrow{>}}},postaction={decorate},rounded corners=7pt]  (12,6)  -- (14,6) node [midway, below, sloped] (TextNode) {};
\filldraw[black] (10,6) circle (8pt) node[anchor=east] {$1$};
\filldraw[black] (12,6) circle (8pt) node[anchor=east] {};
\fill[white] (14,6) circle (8pt) node[anchor=east] {};
\draw[black] (14,6) circle (8pt) node[anchor=west] {};
\end{scope}

\end{tikzpicture}
\caption{All lists of maps in the set $ \MOpjj$ for partitions $\pi =(3,3), \sigma =(3,2)$, and $\mu = (3,3)$. Those lists of maps consist of two connected components which are numbered by $1$ and $2$, each has $3$ edges. The vertex structure is given by $\pi$ and $\sigma$.}
\label{fig90}
\end{subfigure}
\centering

\begin{subfigure}[b]{0.6\textwidth}
\center
\begin{tikzpicture}[scale=0.26]

\draw[gray, thick,decoration={markings,mark=at position 0.3  with {\arrow{>}}},postaction={decorate},rounded corners=7pt]  (10,10) --  (11,11) -- (13,11) -- (14,10) node [midway, above, sloped] (TextNode) {};
\draw[gray, thick,rounded corners=7pt]  (10,10) --  (11,9) -- (13,9) -- (14,10) node [midway, below, sloped] (TextNode) {};
\draw[gray, thick,rounded corners=7pt]  (10,10)  -- (14,10) node [midway, below, sloped] (TextNode) {};

\filldraw[black] (10,10) circle (8pt) node[anchor=east] {};
\fill[white] (14,10) circle (8pt) node[anchor=east] {};
\draw[black] (14,10) circle (8pt) node[anchor=west] {};

\draw[gray, thick,decoration={markings,mark=at position 0.3  with {\arrow{>}}},postaction={decorate},rounded corners=7pt]  (20,10) --  (21,11) -- (23,11) -- (24,10) node [midway, above, sloped] (TextNode) {};
\draw[gray, thick,rounded corners=7pt]  (20,10) --  (21,10.3) -- (23,9) -- (24,10) node [midway, below, sloped] (TextNode) {};
\fill[white] (22,9.6) circle (8pt) node[anchor=east] {};
\draw[gray, thick,rounded corners=7pt]  (20,10) --  (21,9) -- (23,10.3) -- (24,10) node [midway, below, sloped] (TextNode) {};

\filldraw[black] (20,10) circle (8pt) node[anchor=east] {};
\fill[white] (24,10) circle (8pt) node[anchor=east] {};
\draw[black] (24,10) circle (8pt) node[anchor=west] {};
\end{tikzpicture}
\caption{Maps from the set $\widetilde{ M}_{(3),(3)}^{\bullet ;(3)}$. Each of them can be rooted in the unique way.}
\label{fig91}
\end{subfigure}
\centering

\begin{subfigure}[b]{0.6\textwidth}
\center
\begin{tikzpicture}[scale=0.26]
\begin{scope}[shift={(-25,-11)}]
\draw[gray, thick,decoration={markings,mark=at position 0.3  with {\arrow{>}}},postaction={decorate},rounded corners=7pt]  (32,10) --  (33,11) -- (35,11) -- (36,10) node [midway, above, sloped] (TextNode) {};
\draw[gray, thick,rounded corners=7pt] (32,10) --  (33,9) -- (35,9) -- (36,10) node [midway, below, sloped] (TextNode) {};
\draw[gray, thick,rounded corners=7pt] (34,10)  -- (36,10) node [midway, below, sloped] (TextNode) {};

\filldraw[black] (32,10) circle (8pt) node[anchor=east] {};
\filldraw[black] (34,10) circle (8pt) node[anchor=east] {};
\fill[white] (36,10) circle (8pt) node[anchor=east] {};
\draw[black] (36,10) circle (8pt) node[anchor=west] {};

\draw[gray, thick,rounded corners=7pt] (40,10) --  (41,11) -- (43,11) -- (44,10) node [midway, above, sloped] (TextNode) {};
\draw[gray, thick,decoration={markings,mark=at position 0.3  with {\arrow{>}}},postaction={decorate},rounded corners=7pt]  (40,10) --  (41,9) -- (43,9) -- (44,10) node [midway, below, sloped] (TextNode) {};
\draw[gray, thick,rounded corners=7pt] (42,10)  -- (44,10) node [midway, below, sloped] (TextNode) {};

\filldraw[black] (40,10) circle (8pt) node[anchor=east] {};
\filldraw[black] (42,10) circle (8pt) node[anchor=east] {};
\fill[white] (44,10) circle (8pt) node[anchor=east] {};
\draw[black] (44,10) circle (8pt) node[anchor=west] {};

\draw[gray, thick,rounded corners=7pt] (48,10) --  (49,11) -- (51,11) -- (52,10) node [midway, above, sloped] (TextNode) {};
\draw[gray, thick,rounded corners=7pt] (48,10) --  (49,9) -- (51,9) -- (52,10) node [midway, below, sloped] (TextNode) {};
\draw[gray, thick,decoration={markings,mark=at position 0.5  with {\arrow{>}}},postaction={decorate},rounded corners=7pt]  (50,10)  -- (52,10) node [midway, below, sloped] (TextNode) {};

\filldraw[black] (48,10) circle (8pt) node[anchor=east] {};
\filldraw[black] (50,10) circle (8pt) node[anchor=east] {};
\fill[white] (52,10) circle (8pt) node[anchor=east] {};
\draw[black] (52,10) circle (8pt) node[anchor=west] {};
\end{scope}

\end{tikzpicture}
\caption{The only map from the set $\widetilde{M}_{(3),(2,1)}^{\bullet;(3)}$. It could be rooted in three different ways.} 
\label{fig92}
\end{subfigure}

\caption{There are twelve lists of maps in a set $ \MOpjj$ for partitions $\pi =(3,3), \sigma =(3,2)$, and $\mu = (3,3)$, see \cref{fig90}. 
Each of them consists of a map from $\widetilde{ M}_{(3),(3)}^{\bullet ;(3)}$ and $\widetilde{M}_{(3),(2,1)}^{\bullet;(3)}$ presented on \cref{fig91} and \cref{fig92} respectively.} 
\label{main-definition}
\end{figure}
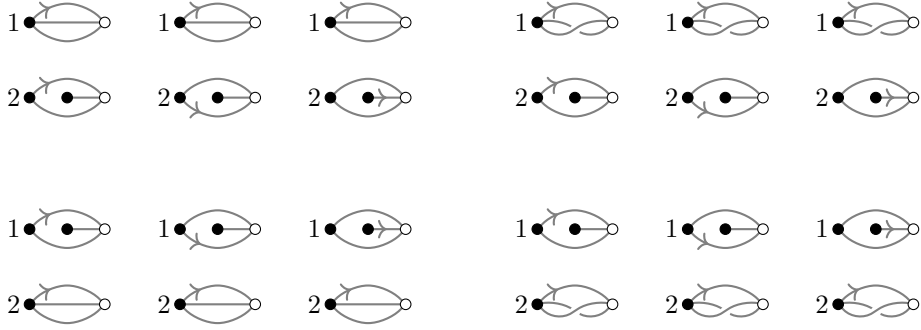
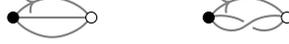
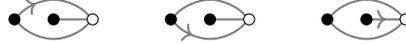

\begin{tw}
\label{main result}
For any triple of partitions $\pi,\sigma ,\mu$, the corresponding polynomial $g^\mu_{\pi,\sigma}(\delta) $ achieves one of the upper bounds on the degree given in \cref{Max}, namely
\[ d \left( \pi,\sigma ;\lambda \right) := \Big( |\pi | - \ell (\pi ) \Big) +\Big( |\sigma | - \ell (\sigma ) \Big) -
\Big( |\mu| - \ell (\mu ) \Big) 
\]
if and only if $|\mu | \geq |\pi| ,| \sigma |$, and both partitions $\pi\cup 1^{|\mu | -|\pi|}$ and $\sigma \cup 1^{|\mu | -|\sigma|}$ are sub-partitions of $\mu$, see \cref{order}. 
For such partitions, the leading coefficient of $g^\mu_{\pi,\sigma}(\delta)$ is a positive integer expressed in the following way:
\[
\left[ \delta^{d (\pi, \sigma ;\mu )}\right] g^\mu_{\pi ,\sigma}= 
C (\pi, \sigma ;\mu) \cdot \dfrac{z_\pi z_\sigma}{z_\mu} \Big\vert  \MOpj \Big\vert ,
\]
where 
\begin{multline*}
C (\pi, \sigma ;\mu ) =\sum_{k=0}^{m_1 (\mu )} 
\binom{m_1 (\mu )}{k}
\binom{m_1 (\pi)+|\mu|-|\pi| -m_1(\mu)}{m_1(\pi) -k} \cdot \\
\binom{m_1(\sigma)+|\mu|-|\sigma| - m_1(\mu)+k}{m_1 (\sigma ) -m_1(\pi) +k} ,
\end{multline*} 
which is equal to 
\begin{equation*}
\binom{m_1(\pi) + |\mu | -|\pi |}{m_1(\pi)} \binom{m_1(\sigma) + |\mu | -|\sigma |}{m_1(\sigma)}
\end{equation*}
\noindent
if $m_1 (\mu )=0$ and is equal to $1$ if $\pi,\sigma ,\mu$ are partitions of the same integer. 
\end{tw}

\cref{6666} is devoted to the proof of above theorem.

\begin{ex}
Let us consider three partitions $\pi =(3,2), \sigma =(3,3)$, and $\mu = (3,3)$. In \cref{main-definition} we have shown that $ \MOpjj =12$. Using the theorem above, the $d \left( \pi,\sigma ;\mu \right)$-coefficient is equal to
\begin{equation*}
\left[ \delta^{d (\pi,\sigma ;\mu )} \right] g^\mu_{\pi,\sigma} =   \dfrac{6 \cdot 18 }{18 }
\binom{1}{0}
\binom{0}{0}
  12 =72 .
\end{equation*}
\end{ex}

\subsection{Relations between the structure constants $g^\mu_{\pi,\sigma}$ and the connection coefficients $c^\mu_{\pi,\sigma}$}
\label{rell}
It is worth mentioning that the coefficients $c^\mu_{\pi,\sigma}$ are indexed by three partitions of the same size, while the quantities $g^\mu_{\pi,\sigma}$ 
are indexed by triples of \emph{arbitrary} partitions. 
Do\l{}\oldk{e}ga and F\'eray investigated the relationship between these two families of coefficients and showed \cite[Equation (19)]{Dolega2014} that for $\mu,\pi,\sigma \vdash n$,
\begin{equation}
\label{connection}
c^\mu_{\pi,\sigma} = 
{\sqrt{\alpha}}^{d (\pi,\sigma ;\mu ) }
\dfrac{z_{\tilde{\mu}}}{z_{\tilde{\pi}} z_{\tilde{\sigma}}}
\sum_{i =0}^{m_1 (\pi )} g^{\tilde{\mu}\cup 1^i}_{\tilde{\pi},\tilde{\sigma}}
\cdot i! \binom{n - | \tilde{\mu} |}{i} ,
\end{equation}
\noindent
where $\widetilde{\pi}$ is constructed from the partition $\pi$ by deleting all units. 

Do\l{}\oldk{e}ga and F\'eray \cite{Dolega2014} proven the polynomiality and the bound on the degree of $g^\mu_{\pi,\sigma}$. Using \cref{connection} they deduced the polynomiality and the bound of the degree of connection coefficients $c^\mu_{\pi,\sigma}$.  
We establish other relations between those two families of coefficients.

\begin{cor}
\label{5.67}
For three given partitions $\mu, \nu,\lambda\vdash n$, each of the polynomials $c_{\mu, \nu}^{\lambda} (\beta )$ and $g^{\tilde{\mu}}_{\tilde{\pi},\tilde{\sigma}} (\delta )$ is of degree at most $d ( \pi,\sigma ;\mu )$, and their leading coefficients coincide up to a normalizing constant, namely
\begin{equation*}
\left[ \beta^{d ( \pi,\sigma ;\mu )} \right] c_{\mu, \nu}^{\lambda} =
\dfrac{z_{\tilde{\mu}}}{z_{\tilde{\pi}} z_{\tilde{\sigma}}} \cdot
\left[ \delta^{d ( \pi,\sigma ;\mu )} \right] g^{\tilde{\mu}}_{\tilde{\pi},\tilde{\sigma}} .
\end{equation*} 
\end{cor}

\begin{proof}
Fix three partitions $\mu, \nu,\lambda\vdash n$. Observe that for each $i\geq 0$, the third estimation shown in \cref{estimation} gives us
\[ \deg_\delta g^{\tilde{\mu}\cup 1^i}_{\tilde{\pi},\tilde{\sigma}} \leq d \left(  \pi,\sigma ;\mu \right) -i . \]
Let us recall that $\delta = \sqrt{\alpha} - \dfrac{1}{\sqrt{\alpha}}$, hence the right-hand side of \cref{connection} is of $\alpha$-degree at most equal to $2  d ( \pi,\sigma ;\mu )$, and in the sum over $i$, the only contribution to the $2  d ( \pi,\sigma ;\mu )$-degree coefficient comes from  $g^{\tilde{\mu}}_{\tilde{\pi},\tilde{\sigma}}$. We have
\begin{multline*}
\left[ \sqrt{\alpha}^{2 d ( \pi,\sigma ;\mu )} \right] \Bigg(
{\sqrt{\alpha}}^{d (\pi,\sigma ;\mu ) }
\dfrac{z_{\tilde{\mu}}}{z_{\tilde{\pi}} z_{\tilde{\sigma}}}
\sum_{i =0}^{m_1 (\pi )} 
g^{\tilde{\mu}\cup 1^i}_{\tilde{\pi},\tilde{\sigma}}
\cdot i! \binom{n - | \tilde{\mu} |}{i}  \Bigg)
= \\
\dfrac{z_{\tilde{\mu}}}{z_{\tilde{\pi}} z_{\tilde{\sigma}}}
\left[ \delta^{d ( \pi,\sigma ;\mu )} \right] 
g^{\tilde{\mu}}_{\tilde{\pi},\tilde{\sigma}} .
\end{multline*}
Since $\beta =\alpha -1$, the $2 d ( \pi,\sigma ;\mu )$-degree coefficient of $ c_{\mu, \nu}^{\lambda}$ in variable $\sqrt{\alpha}$ coincides with $d ( \pi,\sigma ;\mu )$-degree coefficient in variable $\beta$. Hence \cref{connection} finishes the proof.
\end{proof}

Assuming \cref{main result} we are ready to prove the main result of this paper. 
The proof may seem intricate, it combines different facts which have been proven so far.

\begin{proof}[Proof of \cref{C-main result}]
Fix partitions $\pi,\sigma ,\lambda \vdash n$. 
We investigate the polynomial $c_{\pi, \sigma}^{\lambda} (\beta )$. 
By \cref{5.67} we have
\begin{equation*}
\left[ \beta^{d ( \pi,\sigma ;\lambda )} \right] c_{\pi, \sigma}^{\lambda} =
\dfrac{z_{\tilde{\lambda}}}{z_{\tilde{\pi}} z_{\tilde{\sigma}}} \cdot
\left[ \delta^{d ( \pi,\sigma ;\lambda )} \right] g^{\tilde{\lambda}}_{\tilde{\pi},\tilde{\sigma}} ,
\end{equation*}
\noindent
and by \cref{main result} we know that the polynomial $g^{\tilde{\lambda}}_{\tilde{\pi},\tilde{\sigma}}$ achieves the $d ( \pi,\sigma ;\lambda )$-degree part if and only if $|\tilde{\lambda} | \geq |\tilde{\pi}| ,| \tilde{\sigma} |$, $\tilde{\pi} \preceq \tilde{\lambda}$, and $\tilde{\sigma} \preceq\tilde{\lambda}$. 
Observe that this condition is equivalent to $\pi\preceq\lambda$ and $\sigma\preceq\lambda$. Hence the condition on partitions $\pi, \sigma,\lambda$ for achieving by $c_{\pi, \sigma}^{\lambda}$ the $d ( \pi,\sigma ;\lambda )$-degree.

Thus, we have
\begin{equation*}
\left[ \beta^{d ( \pi,\sigma ;\lambda )} \right] c_{\pi, \sigma}^{\lambda} 
\overset{\text{\cref{5.67}}}{=}
\dfrac{z_{\tilde{\lambda}}}{z_{\tilde{\pi}} z_{\tilde{\sigma}}} \cdot
\left[ \delta^{d ( \pi,\sigma ;\lambda )} \right] g^{\tilde{\lambda}}_{\tilde{\pi},\tilde{\sigma}} 
\overset{\text{\cref{main result}}}{=} 
\Big\vert  {\widetilde{M}}^{\bullet; \tilde{\lambda}}_{\tilde{\pi},\tilde{\sigma}} \Big\vert  .
\end{equation*}
\noindent
Since there is only one map $M_1 \in \widetilde{M}_{(1),(1)}^{(1)}$, we have
\begin{equation*}
\Big\vert  {\widetilde{M}}^{\bullet; \tilde{\lambda}}_{\tilde{\pi},\tilde{\sigma}} \Big\vert  =
\Big\vert  \MOpl \Big\vert .
\end{equation*}
\noindent 
Indeed, from any $\lambda$-list of maps $M \in \MOpl$ we can canonically create a $\tilde{\lambda}$-list of map $\widetilde{M} \in {\widetilde{M}}^{\bullet; \tilde{\lambda}}_{\tilde{\pi},\tilde{\sigma}}$ by erasing the last $| \lambda | - | \tilde{\lambda} |$ components. 
This procedure is reversible, since we can add new $M_1$ components to $\widetilde{M}$. 
Then we have
\begin{equation*}
\Big\vert  \MOpl \Big\vert =
\sum_{\nu :\nu \preceq \lambda} \Big\vert  \widetilde{M}^{\nu ; \lambda}_{\pi,\sigma}  \Big\vert 
\overset{\text{\cref{strange2}}}{=} 
\dfrac{z_\lambda}{z_\nu}
\sum_{\nu :\nu \preceq \lambda} \Big\vert  \mathcal{\widetilde{G}}^{\nu ; \lambda}_{\pi,\sigma} \Big\vert .
\end{equation*}
\noindent
Hence
\begin{equation*}
\left[ \beta^{d ( \pi,\sigma ;\lambda )} \right] c_{\pi, \sigma}^{\lambda} 
=
\dfrac{z_\lambda}{z_\nu}
\sum_{\nu :\nu \preceq \lambda} \Big\vert  \mathcal{\widetilde{G}}^{\nu ; \lambda}_{\pi,\sigma} \Big\vert .
\end{equation*}

From \cref{lemlem} we conclude that the leading coefficient of $ c_{\pi, \sigma}^{\lambda}$ overlaps with the leading coefficient of the polynomial
\begin{equation*}
\left( G_\eta \right)_{\pi,\sigma}^{\lambda ;\lambda} := 
\sum_{\delta \in \Gll} \beta^{\stat (\delta )},
\end{equation*}
\noindent
see \cref{stt}, and that both are of the same degree. 
From \cref{lemlem} we also get the second expression for the leading coefficient of the polynomial $ c_{\pi, \sigma}^{\lambda}$.
\end{proof}

\section{The top-degree part of structure constants}
\label{6666}
This section is devoted to the proof of \cref{main result}. 
Firstly, we present some basic computations leading to the exact formulas for the top-degree part of Jack characters. 
We present those formulas in terms of injective embeddings into Young diagrams. 
Secondly, we consider a particular class of collections of bipartite maps $ P^\mu_{\pi,\sigma}$ which constitute a good candidate for the 
top-degree parts of the structure constants $ g^\mu_{\pi,\sigma} $. 
Finally, we prove that those candidates 
for the top-degree part of structure constants $ g^\mu_{\pi,\sigma}$ (see \cref{candidate}) are indeed them.

\subsection{Embeddings of bicolored graphs}

A bicolored graph $G$ is a bipartite graph together with a choice of the colouring of its vertex set 
$\V$; we denote by $\V_{\bullet}$ and $\V_{\circ}$ respectively the
sets of black and white vertices of $G$.

\begin{de}
An \emph{injective embedding} $F$ of a bicolored graph $G$ to a Young diagram $\lambda$ is a function
which maps $\V_{\circ}$ to the set of columns of $\lambda$, maps $\V_{\bullet}$ to the set of rows of
$\lambda$, and maps injectively the set of edges $\mathcal{E}$ to the set of boxes of $\lambda$, see \cref{fig32}.
We also require that $F$ preserves the relation of incidence, \ie each
vertex $v \in  \V$ should be mapped to a row or a column $F(v)$ which contains the box $F(e)$, for every edge $e \in \mathcal{E}$ incident to $v$.
We denote by $N_G(\lambda )$ the number of such embeddings of $G$ into $\lambda$.
\end{de}

It is also useful to consider injective embeddings of a graph $G$ into a Young diagram $\lambda$, with the roles of black and white vertices reversed (\ie black vertices are mapped into \emph{columns}, white vertices into the \emph{rows}). We refer to such embeddings as \emph{negative injective embeddings} and denote the number of such embeddings as $\overline{N_G(\lambda )}$.

\begin{de}
\label{graph}
For any partition $\pi =\left( \pi_1 ,\ldots ,\pi_r \right) $ we define the graph $G_\pi$ as the unique bicoloured graph consisting of $r$ black vertices of degrees $\pi_1 ,\ldots ,\pi_r $ respectively and $|\pi |$ white vertices, each of degree one (see \cref{fig32}). Similarly, we define $\overline{G_\pi}$ as the unique bicoloured graph consisting of $r$ white vertices of degrees $\pi_1 ,\ldots ,\pi_r $ respectively and $|\pi |$ black vertices, each of degree one. 
\end{de}

\begin{rem}
\label{rem1}
The number $ N_{G_\pi}\left( \lambda \right)$ of injective embeddings of the graph $G_\pi$ into the Young diagram $\lambda$ is equal to the number $ \overline{N_{\overline{G_\pi}}\left( \lambda \right)}$ of negative injective embeddings of the graph $\overline{G_\pi}$ into the Young diagram $\lambda$.
\end{rem}

\begin{figure}
\centering
\begin{tikzpicture}[scale=0.8]

\draw[gray, thick] (1,4) -- (3,5.5) node [midway, above, sloped] (TextNode) {$e_1$};
\draw[gray, thick] (1,4) -- (3,4) node [midway, above, sloped] (TextNode) {$e_2$};
\draw[gray, thick] (1,4) -- (3,2.5) node [midway, above, sloped] (TextNode) {$e_3$};

\draw[gray, thick] (1,1) -- (3,1) node [midway, above, sloped] (TextNode) {$e_4$};

\filldraw[black] (1,1) circle (4pt) node[anchor=east] {$v^\bullet_2$};
\filldraw[black] (1,4) circle (4pt) node[anchor=east] {$v^\bullet_1$};

\fill[white] (3,5.5) circle (4pt) node[anchor=east] {};
\fill[white] (3,4) circle (4pt) node[anchor=west] {};
\fill[white] (3,2.5) circle (4pt) node[anchor=east] {};
\fill[white] (3,1) circle (4pt) node[anchor=west] {};

\draw[black] (3,1) circle (4pt) node[anchor=west] {$v^\circ_4$};
\draw[black] (3,2.5) circle (4pt) node[anchor=west] {$v^\circ_3$};
\draw[black] (3,4) circle (4pt) node[anchor=west] {$v^\circ_2$};
\draw[black] (3,5.5) circle (4pt) node[anchor=west] {$v^\circ_1$};

\draw (8,3) rectangle (9,4) node[gray,pos=.5] {$e_3$}; 
\draw (9,3) rectangle (10,4) node[gray,pos=.5] {$e_2$};
\draw (10,3) rectangle (11,4) node[gray,pos=.5] {};
\draw (11,3) rectangle (12,4) node[gray,pos=.5] {$e_1$};
\draw (8,4) rectangle (9,5) node[gray,pos=.5] {$e_4$};
\draw (9,4) rectangle (10,5) node[gray,pos=.5] {};
\draw (10,4) rectangle (11,5) node[gray,pos=.5] {};

\draw[<-] (8.5,2.9) -- (8.5,2) node[anchor=north] {$v^\circ_3$,$v^\circ_4$};
\draw[<-] (9.5,2.9) -- (9.5,2) node[anchor=north] {$v^\circ_2$};
\draw[<-] (11.5,2.9) -- (11.5,2) node[anchor=north] {$v^\circ_1$};

\draw[<-] (7.9,3.5) -- (7,3.5) node[anchor=east] {$v^\bullet_1$};
\draw[<-] (7.9,4.5) -- (7,4.5) node[anchor=east] {$v^\bullet_2$};

\end{tikzpicture}
\caption{The graph $G_\pi$ associated with the partition $\pi = \left( 3,1\right) $. On the right, an example of its injective embedding into the Young diagram $\lambda =\left( 4,3 \right)$.} 
\label{fig32}
\end{figure}
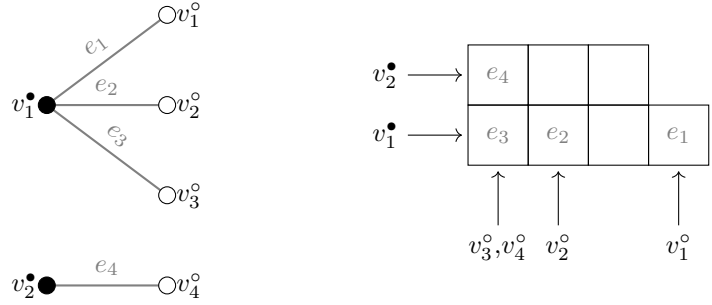

\subsection{Exact formulas for top-degree part of Jack characters}
\label{88}
Śniady proved \cite[Proposition 3.5]{Sniady2015} that each Jack character is a function on the set $\mathbb{Y}$ of Young diagrams
\begin{equation*}
   \mathbb{Y}    \ni \lambda \quad\longmapsto     \quad  \Ch_\pi (\lambda) \in \Laurentdd
\end{equation*}
\noindent
with values in the set $\Laurentdd$ of Laurent polynomials in the variable $A$ of degree at most $|\pi |-\ell (\pi )$. 
We denote by
\begin{equation*}
\A{\topp} \Ch_\pi \left( \lambda \right) := \A{|\pi |-\ell (\pi)} \Ch_\pi \left( \lambda \right) 
\end{equation*}
\noindent
the leading part of this Laurent polynomial. We shall express this quantity in terms of injective embeddings of $G_\pi$ into $\lambda$.

\begin{prop}
\label{AtopCh}
For any Young diagram $\lambda \in \mathbb{Y}$ and partition $\pi$, we have that
\begin{equation*}
\A{\topp} \Ch_\pi \left( \lambda \right)  =  N_{G_\pi}\left( \lambda \right) .
\end{equation*} 
\noindent
That is, the leading part of $\Ch_\pi (\lambda)$ is equal to the number of injective embeddings of the graph $G_\pi$ into the Young diagram $\lambda$. 
\end{prop}

\begin{ex}
Let us consider the partition $\pi =\left( 3, 1\right)$ and the Young diagram $\lambda = \left( \lambda_1 , \lambda_2 \right) $. We have
\begin{equation*}
\begin{split}
\A{\topp} \Ch_{(3,1)} \left( \lambda_1 , \lambda_2\right) =  N_{G_\pi}\left( \lambda \right) =
{\lambda_1}^{\underline{4}}_{} &+ 
{\lambda_1}^{\underline{3}}_{} \cdot {\lambda_2}^{\underline{1}}_{}  \\
&+{\lambda_1}^{\underline{1}}_{} \cdot {\lambda_2}^{\underline{3}}_{} +
{\lambda_2}^{\underline{4}}_{} .
\end{split}
\end{equation*}
\noindent
One of embeddings which contributes to $N_{G_\pi}( \lambda )$ is presented on \cref{fig32}. 
\end{ex}

Before proving \cref{AtopCh} we introduce the notion of $\alpha$-shifted symmetric functions (see more in \cite[Section 2.2]{MR2424904} or \cite[Definition 2.2]{MR3649821}) and present Jack characters in this context.

\begin{de}
\label{frt}
An $\alpha$-shifted symmetric function $F = (F_N)_{N \geq 1}$ is a sequence of polynomials $F_N$ such that
\begin{itemize}

\item for each $N \geq 1$, $F_N$ is a polynomial in $N$ variables $x_1 ,\ldots ,x_N$ with coefficients in the field of rational functions $\mathbb{Q}(\alpha)$ in some indeterminate $\alpha$ that is symmetric in the variables 
\[\xi_1 := x_1 -\dfrac{1}{\alpha} ,\quad \xi_2 := x_2 -\dfrac{2}{\alpha} ,\ldots ,\quad\xi_N := x_N -\dfrac{N}{\alpha} ,\]
\item for each $N \geq1$, $F_{N+1} (x_1 ,\ldots ,x_N,0) =F_N (x_1 ,\ldots ,x_N) $ (the stability property),
\item $\sup_{N\geq1} \deg (F_N ) <\infty$.
\end{itemize}
\noindent
The degree of a shifted-symmetric function $F$ is defined as maximum of the degrees of the corresponding
polynomials $F_N (x_1, \ldots , x_N)$.
\end{de}

Śniady and F\'eray gave some abstract characterizations of Jack characters \cite[Theorem 1.7, Theorem A.2]{Sniady2015}. 
We present the one given by F\'eray, which can be traced back to the earlier work of Knop and Sahi \cite{MR1399412}.

\begin{tw}\cite[Theorem A.2]{Sniady2015}
Let $\pi$ be a partition and $A$ be a complex number such that $-\dfrac{1}{\alpha} = -\dfrac{1}{A^2}$ is not a positive integer. There exists a unique shifted-symmetric function $F$ such that: 
\begin{itemize}

\item[(J1)] $F$ is a shifted-symmetric function of degree $| \pi |$, and its top-degree homogeneous part is equal to
\[ A^{|\pi|-\ell(\pi)}\ p_\pi(\lambda_1,\dots,\lambda_m),\]
\noindent
where $p_\pi$ is the \emph{power-sum symmetric polynomial} given by the formula
\begin{equation*}
p^{}_\pi \left( \lambda \right) = \prod_r \sum_i \lambda_i^{\pi_r} .
\end{equation*}

\item[(J2)] 
$F(\lambda) = 0$ holds for each Young diagram $\lambda$ such that $|\lambda | < | \pi |$ (the vanishing property).

\end{itemize}
 
Moreover, if $\alpha$ is a positive real number, the function $F= (F_N)_{N \geq 1}$ satisfies
$\Ch_\pi ( \lambda ) =   F_r(\lambda_1 ,\ldots ,\lambda_r )$ for each Young diagram $\lambda =(\lambda_1 ,\ldots ,\lambda_r )$.
\end{tw}

To keep notation short, we introduce the following symmetric function
\begin{equation*}
\widehat{p^{}_\pi}\left(\lambda\right) :=
\mathop{\mathlarger{\mathlarger{\mathlarger{\mathlarger{\mathlarger{\mathlarger{\ast}}}}}}}\limits_{r}^{}
\sum_i                     \lambda_i^{\underline{\pi_r}} ,
\end{equation*}
\noindent
where 
\[
\lambda_i^{\underline{l_i}} \ast \lambda_j^{\underline{l_j}} =
\begin{cases}
\lambda_i^{\underline{l_i +l_j}}                             &   \text{if }  i=j ,\\
\lambda_i^{\underline{l_i}} \cdot \lambda_j^{\underline{l_j}}&  \text{otherwise, } 
\end{cases}
\]
\noindent
and 
\[
\lambda^{\underline{l}} =\underbrace{\lambda \cdot(\lambda-1 )\cdots (\lambda -l+1)}_{l \text{ factors}} .
\]

\begin{proof}[Proof of \cref{AtopCh}]
Observe that
\begin{equation*}
\widehat{p^{}_\pi}\left(\lambda\right)   =  N_{G_\pi}\left( \lambda \right) . 
\end{equation*} 
\noindent
We will show that
\[
\A{\topp} \Ch_\pi \left( \lambda \right) =\widehat{p^{}_\pi}\left(\lambda\right) .
\]

Let $F$ be an $\alpha$-shifted symmetric function associated to $\pi$ by \cref{frt}. 
Let us choose a sufficiently large integer $N$, \emph{e.g.}~$N> | \pi |$. 
Let us treat the coefficients of the polynomial $F_N$ as variables. 
The equality 
\[ F_N (\lambda) =\Ch_\pi ( \lambda ) \in \Laurentdd ,
\]
which holds for each $\lambda \in \mathbb{Y}$, becomes a system of equations with coefficients in $\mathbb{N}_{\geq 0}$. 
This system is large enough to conclude that each coefficient of a polynomial $F_N$ is a linear combination of the quantities $\Ch_\pi ( \lambda ) $ over $\mathbb{Q}$, hence $F_N$ is a polynomial in $N$ variables with coefficients in $\Laurentdd$.

Notice that formally we have equality for all $\alpha >0 $. However, the rational function from $\mathbb{Q}(\alpha )$ is uniquely determined by its \emph{values} for $\alpha \geq 0$. 

Since $F_N$ is a shifted-symmetric function with coefficients in the set $\Laurentdd$, its $A$-top degree
\begin{equation*}
\A{\topp} F_N \left( \lambda_1 ,\ldots , \lambda_N \right) :=\A{|\pi |-\ell (\pi)} F_N \left( \lambda_1 ,\ldots , \lambda_N \right) 
\end{equation*}
\noindent
is a \emph{symmetric} function in the variables $\lambda_1 ,\ldots , \lambda_N $. 
Indeed, for each permutation $\sigma$ of $[N]$ we have
\begin{align*} 
\A{\topp} F_N \left(x_1,\ldots,x_N \right) &=\A{\topp} F_N \left( x_1-\dfrac{1}{A^2},\ldots,x_N -\dfrac{N}{A^2}\right) = \\
&\A{\topp} F_N \left( x_{\sigma (1)}-\dfrac{\sigma (1)}{A^2},\ldots,x_{\sigma (N)} -\dfrac{\sigma (N)}{A^2}\right) =\\
&\A{\topp} F_N \left(x_{\sigma (1)},\ldots,x_{\sigma  (N) } \right) .
\end{align*} 

Since $F_N$ is of a degree~$| \pi | $, the polynomial $\A{\topp} F_N $ has the same bound of the degree. 
Observe that the homogeneous top-degree part of $\widehat{p^{}_\pi}\left(\lambda\right) $ is equal to 
$ p^{}_\pi  \left( \lambda \right)$ and so does the homogeneous top-degree part of $\A{\topp} F_N$. Polynomials $\widehat{p^{}_\pi}\left(\lambda\right) $ and $\A{\topp} F_N$ are both symmetric, hence
\begin{equation*}
\A{\topp} F_N -\widehat{p^{}_\pi}
\end{equation*}
\noindent
is a symmetric polynomial in variables $( \lambda_1 ,\ldots ,\lambda_N )$ of a degree at most $|\pi |-1$. 

We use the following notation:
\begin{equation*}
\mathcal{Y}_0 := \left\lbrace ( \lambda_1 ,\ldots ,\lambda_N ) \in \mathbb{Z}^N : \lambda_1 \geq\ldots \geq\lambda_N\geq 0 \text{ and }\lambda_1 +\ldots +\lambda_N < |\pi | \right\rbrace .
\end{equation*}
By the vanishing property we have
\begin{equation*}
\A{\topp} F_N  \left( \lambda \right) = \A{\topp}  \Ch_\pi \left( \lambda \right) =0
\end{equation*}
\noindent
for all elements $\lambda\in\mathcal{Y}_0$. 
Since there are no injective embeddings of $G_\pi$ into a Young diagram with the number of boxes smaller then the number of edges in $G_\pi$, we have
\begin{equation*}
\widehat{p^{}_\pi} \left( \lambda \right) = N_{G_\pi} \left( \lambda \right) =0
\end{equation*}
\noindent
for all elements $\lambda\in\mathcal{Y}_0$. 
From that we deduce that $\mathcal{Y}_0$ is a set of zeros of the polynomial $\A{\topp} F_N -\widehat{p^{}_\pi}$. The appropriate set of zeros of a polynomial of sufficiently small degree determines the vanishing of the polynomial. In fact, we can use the characterisation given by Śniady \cite[Lemma 7.1]{Sniady2015} to conclude that the symmetric polynomial
\begin{equation*}
W (x_1 , \ldots ,x_N ):= \Big( \A{\topp} F_N -\widehat{p^{}_\pi}\Big) (x_1 -1 , \ldots ,x_N -1 ) ,
\end{equation*}
\noindent
which is of a degree at most $|\pi |-1$, vanishes. Hence we conclude that
\begin{equation*}
 \A{\topp} F_N = \widehat{p^{}_\pi}
\end{equation*}
\noindent
which finishes the proof.
\end{proof}

\subsection{Hands-shaking procedure}
\label{procedure11}
Let 
\begin{equation*}
\pi =\left( \pi_1,\ldots,\pi_n \right) ,\quad
\sigma =\left( \sigma_1,\ldots,\sigma_l \right)
\end{equation*}
\noindent
be two partitions. 
We define a class of collections of maps by the following procedure:
\begin{enumerate}
\item[Step 1.] For each $i \in [n]$ we assign a white vertex with $\pi_i$ outgoing half-edges. We label this vertex by the number $i$ and we root it, \ie we choose one of the outgoing half-edges and decorate it. Similarly, for each $j\in [l]$ we assign a black vertex with $\sigma_j$ outgoing half-edges and we root it.
\item[Step 2.] We match some of the half-edges going out from the white vertices with some of those going out from black vertices. 
\item[Step 3.] We close each of the non-closed half-edges by a white or a black vertex so that the graph remains bipartite.
\end{enumerate}

We call the procedure described above the ``hands-shaking procedure''. 
The name provenance could be explained as follows: there are white and black vertices with hands; the number of hands is given by the partitions $\pi$ and $\sigma$. 
They shake theirs hands in any way they like, but only black-white connections are allowed. 
On \cref{proc1} we present an example of applying this procedure.

\begin{de}
\label{procedure1}
For a given triple of partitions $\pi, \sigma, \mu$ we denote by $P^\mu_{\pi,\sigma} $ the set of all $\mu$-collections of maps which may be obtain as an outcome of performing the above presented ``hands-shaking procedure''. 
\end{de}

Each $\mu$-collection of maps $M \in P^\mu_{\pi,\sigma} $ can be obtained in the \emph{unique} way as an outcome of the presented procedure. The uniqueness follows from the fact that the position of each edge from $M$ is uniquely determined by the labellings on the rooted vertices and the order of the outgoing half-edges.

\begin{figure}
\centering
\begin{tikzpicture}[scale=0.8]

\draw[gray, thick,decoration={markings,mark=at position 0.866  with {\arrow{>}}},postaction={decorate}] (1,10) -- (1.5,10.5);

\draw[gray, thick] (1,10) --  (1.5,10);

\draw[gray, thick] (1,10) -- (1.5,9.5);

\draw[gray, thick,decoration={markings,mark=at position 0.866  with {\arrow{>}}},postaction={decorate}] (1,8.5) -- (1.5,9);

\draw[gray, thick] (1,8.5) --  (1.5,8);

\draw[gray, thick,decoration={markings,mark=at position 0.866  with {\arrow{>}}},postaction={decorate}]  (1,7) --  (1.5,7);

\draw[gray, thick,decoration={markings,mark=at position 0.866  with {\arrow{>}}},postaction={decorate}]  (5,9.5) -- (4.5,10);

\draw[gray, thick] (5,9.5) --  (4.5,9);

\draw[gray, thick,decoration={markings,mark=at position 0.866  with {\arrow{>}}},postaction={decorate}]  (5,7.5) -- (4.5,8);

\draw[gray, thick] (5,7.5) --  (4.5,7);

\filldraw[black] (1,10) circle (3pt) node[anchor=east] {1};
\filldraw[black] (1,8.5) circle (3pt) node[anchor=east] {2};
\filldraw[black] (1,7) circle (3pt) node[anchor=east] {3};

\fill[white] (5,7.5) circle (3pt) node[anchor=east] {};
\draw[black] (5,7.5) circle (3pt) node[anchor=west] {2};
\fill[white] (5,9.5) circle (3pt) node[anchor=east] {};
\draw[black] (5,9.5) circle (3pt) node[anchor=west] {1};

\draw[gray, thick]  (9,10) --  (9.5,10);

\draw[gray, thick] (9,10) -- (9.5,9.5);

\draw[gray, thick,decoration={markings,mark=at position 0.166  with {\arrow{>}}},decoration={markings,mark=at position 0.836  with {\arrow{<}}},postaction={decorate},rounded corners=15pt] (9,10) -- (9.5,10.5) -- (12.5,10) -- (13,9.5);

\draw[gray, thick,decoration={markings,mark=at position 0.166  with {\arrow{>}}},postaction={decorate},rounded corners=15pt]   (9,8.5) -- (9.5,9)  --  (12.5,7) -- (13,7.5);

\draw[white, line width=4pt,rounded corners=15pt] (9,8.5) --  (9.5,8) -- (12.5,8) -- (13,7.5);

\draw[gray, thick,decoration={markings,mark=at position 0.836  with {\arrow{<}}},postaction={decorate},rounded corners=15pt] (9,8.5) --  (9.5,8) -- (12.5,8) -- (13,7.5);

\draw[white, line width=4pt,rounded corners=15pt] (9,7) --  (9.5,7) -- (12,8)-- (12.5,9)  -- (13,9.5) ;

\draw[gray, thick,decoration={markings,mark=at position 0.166  with {\arrow{>}}}, postaction={decorate},rounded corners=15pt] (9,7) --  (9.5,7) -- (12,8)-- (12.5,9)  -- (13,9.5) ;

\filldraw[black] (9,10) circle (3pt) node[anchor=east] {1};
\filldraw[black] (9,8.5) circle (3pt) node[anchor=east] {2};
\filldraw[black] (9,7) circle (3pt) node[anchor=east] {3};

\fill[white] (13,7.5) circle (3pt) node[anchor=east] {};
\draw[black] (13,7.5) circle (3pt) node[anchor=west] {2};
\fill[white] (13,9.5) circle (3pt) node[anchor=east] {};
\draw[black] (13,9.5) circle (3pt) node[anchor=west] {1};

\draw[gray, thick] (1,2) --  (1.5,2);

\draw[gray, thick] (1,2) -- (1.5,1.5);

\draw[gray, thick,decoration={markings,mark=at position 0.166  with {\arrow{>}}},decoration={markings,mark=at position 0.836  with {\arrow{<}}},postaction={decorate},rounded corners=15pt]  (1,2) -- (1.5,2.5) -- (4.5,2) -- (5,1.5);

\draw[gray, thick,decoration={markings,mark=at position 0.166  with {\arrow{>}}},postaction={decorate},rounded corners=15pt]  (1,0.5) -- (1.5,1)  --  (4.5,-1) -- (5,-0.5);

\draw[white, line width=2pt,rounded corners=15pt] (1,0.5) --  (1.5,0) -- (4.5,0) -- (5,-0.5);

\draw[gray, thick,decoration={markings,mark=at position 0.836  with {\arrow{<}}},postaction={decorate},rounded corners=15pt]  (1,0.5) --  (1.5,0) -- (4.5,0) -- (5,-0.5);

\draw[white, line width=2pt,rounded corners=15pt] (1,-1) --  (1.5,-1) -- (4,0)-- (4.5,1)  -- (5,1.5) ;

\draw[gray, thick,decoration={markings,mark=at position 0.166  with {\arrow{>}}}, postaction={decorate},rounded corners=15pt] (1,-1) --  (1.5,-1) -- (4,0)-- (4.5,1)  -- (5,1.5) ;

\filldraw[black] (1,2) circle (3pt) node[anchor=east] {1};
\filldraw[black] (1,0.5) circle (3pt) node[anchor=east] {2};
\filldraw[black] (1,-1) circle (3pt) node[anchor=east] {3};

\fill[white] (5,-0.5) circle (3pt) node[anchor=east] {};
\draw[black] (5,-0.5) circle (3pt) node[anchor=west] {2};
\fill[white] (5,1.5) circle (3pt) node[anchor=east] {};
\draw[black] (5,1.5) circle (3pt) node[anchor=west] {1};

\fill[white] (1.5,2) circle (3pt) node[anchor=east] {};
\draw[black] (1.5,2) circle (3pt) node[anchor=west] {};
\fill[white] (1.5,1.5) circle (3pt) node[anchor=east] {};
\draw[black] (1.5,1.5) circle (3pt) node[anchor=west] {};

\begin{scope}[shift={(-1,-17.7)}]

\draw[gray, dashed, thick,decoration={markings,mark=at position 0.833  with {\arrow{<}}}, postaction={decorate},rounded corners=9pt]  (13,18) --  (13.5,18.5) -- (14.5,18.5)-- (15,18);
\draw[gray, thick,decoration={markings,mark=at position 0.166  with {\arrow{>}}}, postaction={decorate},rounded corners=9pt]  (13,18) --  (13.5,17.5) -- (14.5,17.5)-- (15,18);
\draw[black] (14,18) circle (28pt) node[anchor=west] {};

\filldraw[black] (13,18) circle (3pt) node[anchor=east] {$2$};

\fill[white] (15,18) circle (3pt) node[anchor=east] {};
\draw[black] (15,18) circle (3pt) node[anchor=west] {$2$};

\end{scope}

\begin{scope}[shift={(-4,-17.7)}]

\draw[gray, dashed, thick,decoration={markings,mark=at position 0.266  with {\arrow{>}}}, postaction={decorate},rounded corners=9pt]  (13.1,18.4) --  (13.5,18.5) -- (14.5,18.5)-- (15,18);
\draw[gray, thick,decoration={markings,mark=at position 0.3  with {\arrow{>}}},decoration={markings,mark=at position 0.7  with {\arrow{<}}},postaction={decorate},rounded corners=9pt]   (14,17.6) -- (14.5,17.5)-- (15,18);
\draw[gray, thick,rounded corners=6pt]  (13.6,17.1) -- (13.8,17.2) --   (14,17.6);
\draw[gray, thick,rounded corners=9pt]  (13.3,18) --(13.6,17.6) --   (14,17.6);
\draw[black] (14,18) circle (28pt) node[anchor=west] {};

\filldraw[black] (13.1,18.4) circle (3pt) node[anchor=east] {$3$};
\filldraw[black]  (14,17.6) circle (3pt) node[anchor=south] {$1$};

\fill[white] (15,18) circle (3pt) node[anchor=east] {};
\draw[black] (15,18) circle (3pt) node[anchor=west] {$1$};

\fill[white] (13.3,18)  circle (3pt) node[anchor=east] {};
\draw[black] (13.3,18) circle (3pt) node[anchor=west] {};
\fill[white] (13.6,17.1) circle (3pt) node[anchor=east] {};
\draw[black] (13.6,17.1) circle (3pt) node[anchor=west] {};

\end{scope}

\node[text width=4.5cm] at (3,4.8) 
    {Step 1. Black and white vertices with outgoing half-edges of degrees $(\sigma_i )$ and $(\pi_j)$ respectively.};
\node[text width=4.5cm] at (11.5,4.8) 
    {Step 2. Some of the outgoing half-edges were matched. The crossing of edges is not important.};    
    
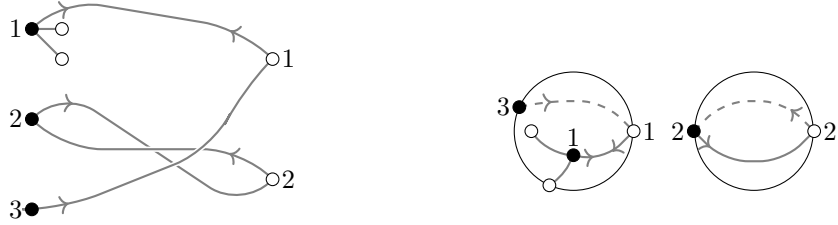
\node[text width=4.5cm] at (3,-2.8) 
    {Step 3. The rest of outgoing half-edges is closed. The collection of maps is bipartite.};
\node[text width=5.2cm] at (11.5,-2.8) 
    {As an outcome we obtain the following collection of two maps drawn on a pair of spheres.};

\end{tikzpicture}
\caption{The three steps of ``hands-shaking procedure''. As an output we obtain the $(4,2)$-collection of bipartite maps. Vertices are labelled and rooted as as the ``hands-shaking procedure'' describes.}
\label{proc1}
\end{figure}

\begin{ob}
\label{strange6}
For given partitions $\pi,\sigma, \mu $ the set $P^\mu_{\pi,\sigma}$ is non-empty if and only if the following conditions holds
\begin{enumerate}
\item \label{rt} $|\pi |,|\sigma | \leq |\mu |$,
\item \label{rv} both partitions $\pi\cup 1^{|\mu | -|\pi|}$ and $\sigma \cup 1^{|\mu | -|\sigma|}$ are sub-partitions of $\mu$.
\end{enumerate}
\end{ob}

\begin{proof}
Firstly, we will show the necessity of conditions. 
Observe that by performing ``hands-shaking procedure'', in which we obtain a $\mu$-collection of map, the vertex set is given by
\begin{equation*}
 \Lambda_\mathcal{W} (M ) = \pi \cup 1^{|\mu | -| \pi |}  \quad\quad\text{and}\quad\quad
\Lambda_\mathcal{B} (M ) = \sigma \cup 1^{|\mu | -| \sigma |}.
\end{equation*}
\noindent
The first condition follows immediately. Partitions describing white or black vertices distributions are sub-partitions of a partition describing face distribution. Hence the second condition has to be satisfied. 

For partitions satisfying those two conditions one can exhibit a collection of maps from $P^\mu_{\pi,\sigma}$, which proves the sufficiency of those conditions. 
\end{proof}

\begin{ob}
\label{strange5}
For given partitions $\pi,\sigma, \mu $: $|\pi |,|\sigma | \leq |\mu |$ we have
\[
\Big\vert P^\mu_{\pi,\sigma} \Big\vert = 
C (\pi, \sigma ;\mu) \cdot \dfrac{z_\pi z_\sigma}{z_\mu} \Big\vert  \MOpj \Big\vert 
\]
where 
\begin{multline}
\label{99}
C (\pi, \sigma ;\mu ) =\sum_{k=0}^{m_1 (\mu )} 
\binom{m_1 (\mu )}{k}
\binom{m_1 (\pi)+|\mu|-|\pi| -m_1(\mu)}{m_1(\pi) -k} \cdot \\
\binom{m_1(\sigma)+|\mu|-|\sigma| - m_1(\mu)+k}{m_1 (\sigma ) -m_1(\pi) +k} ,
\end{multline} 
which is equal to 
\begin{equation*}
\binom{m_1(\pi) + |\mu | -|\pi |}{m_1(\pi)} \binom{m_1(\sigma) + |\mu | -|\sigma |}{m_1(\sigma)} 
\end{equation*}
\noindent
if $m_1 (\mu )=0$ and is equal to $1$ if $\pi,\sigma ,\mu$ are partitions of the same integer. 
\end{ob}

\begin{proof}
Observe that the elements of $P^\mu_{\pi,\sigma} $ are $\mu$-collection of bipartite orientable maps whose vertex set is given by
\begin{equation*}
 \Lambda_\mathcal{W} (M ) = \pi \cup 1^{|\mu | -| \pi |}  \quad\quad\text{and}\quad\quad
\Lambda_\mathcal{B} (M ) = \sigma \cup 1^{|\mu | -| \sigma |}.
\end{equation*}
\noindent
Each such an element has the following labels and roots on the vertices and half-edges:
\begin{enumerate}
\item there are $n$ white vertices of degrees $\pi_1, \ldots ,\pi_n$, each being labelled by a relevant natural number from $[n]$ and rooted, \ie we choose one of the outgoing half-edges and decorate it by an arrow,
\item there are $l$ black vertices of degrees $\sigma_1, \ldots ,\sigma_l$, each being labelled by a relevant natural number from $[l]$ and rooted.
\end{enumerate} 
Moreover, each connected component of an element from $ P^\mu_{\pi,\sigma} $ has \emph{at least one} decorated vertex. 
\medskip

We use the double counting method as in \cref{strange}. 
For each $M\in P^\mu_{\pi,\sigma}$ we can root and number the connected components in $z_\mu$ ways. 

Let us choose $M \in  \MOpj$. 
The procedure of labelling and rooting the vertices is much more subtle. 
Firstly, we have to choose $m_1(\pi) $ white (respectively $m_1 (\sigma)$ black) vertices and label them by adequate numbers. 
At the first sight, we could do this in
\begin{equation*}
\binom{m_1(\pi) + |\mu | -|\pi |}{m_1(\pi)} \binom{m_1(\sigma) + |\mu | -|\sigma |}{m_1(\sigma)} z_\pi z_\sigma
\end{equation*}
\noindent
ways (which is equal to $ z_\pi z_\sigma$ if $\pi,\sigma,\mu$ are partitions of the same integer). However, in the definition of $P^\mu_{\pi,\sigma}$ we required to contain at least one labelled vertex from each connected component. This is trivially satisfied if $m_1 (\mu) =0$. This consideration yields the expression \cref{99}. We describe briefly the details. 

There is $m_1 (\mu) $ one-element connected components in $M$. Denote the set of those components by $M_1$. For each integer $k$: $0\leq k \leq m_1 (\mu )$ we can choose $k$ white vertices from $M_1$ and we required that exactly those white vertices among all white vertices in $M_1$ are numbered. The number of possible ways of numbering vertices of $M$ in such a way is the contribution to the sum in \cref{99} relevant to $k$. We sum up over all $k = 0,\ldots ,m_1 (\mu)$.
\end{proof}

\subsection{Proof of \cref{main result}}
We prove that candidates $ p^\mu_{\pi,\sigma} :=| P^\mu_{\pi,\sigma} |$ for top-degree part of structure constants $ g^\mu_{\pi,\sigma}$ suit well for that role.

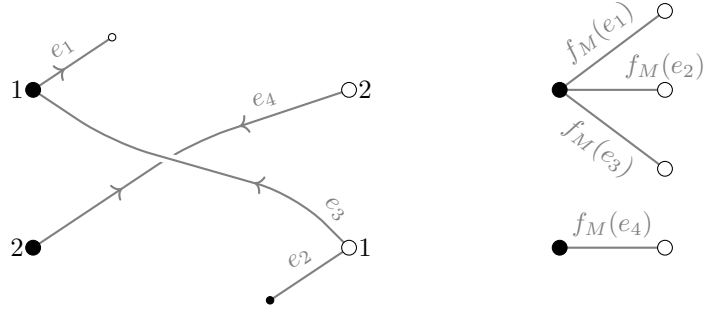
\begin{figure}
\centering
\begin{tikzpicture}[scale=0.7]

\draw[gray, thick] (1,4) -- (3,5.5) node [midway, above, sloped] (TextNode) {$f_M (e_1)$};
\draw[gray, thick] (1,4) -- (3,4) node [ above, sloped] (TextNode) {$f_M (e_2)$};
\draw[gray, thick] (1,4) -- (3,2.5) node [midway, below, sloped] (TextNode) {$f_M (e_3)$};

\draw[gray, thick] (1,1) -- (3,1) node [midway, above, sloped] (TextNode) {$f_M (e_4)$};

\filldraw[black] (1,1) circle (4pt) node[anchor=south] {};
\filldraw[black] (1,4) circle (4pt) node[anchor=south] {};

\fill[white] (3,5.5) circle (4pt) node[anchor=east] {};
\fill[white] (3,4) circle (4pt) node[anchor=west] {};
\fill[white] (3,2.5) circle (4pt) node[anchor=east] {};
\fill[white] (3,1) circle (4pt) node[anchor=west] {};

\draw[black] (3,1) circle (4pt) node[anchor=south] {};
\draw[black] (3,2.5) circle (4pt) node[anchor=south] {};
\draw[black] (3,4) circle (4pt) node[anchor=south] {};
\draw[black] (3,5.5) circle (4pt) node[anchor=south] {};

\draw[gray,decoration={markings,mark=at position 0.4 with {\arrow{>}}},postaction={decorate}, thick] (-9,4) -- (-7.5,5)  node [midway, above, sloped] (TextNode) {$e_1$};
\draw[gray, thick] (-3,1) -- (-4.5,0)  node [midway, above, sloped] (TextNode) {$e_2$};
\draw[gray, thick,decoration={markings,mark=at position 0.3  with {\arrow{>}}},decoration={markings,mark=at position 0.7  with {\arrow{<}}},postaction={decorate},rounded corners=15pt](-9,1) --  (-6,3) -- (-3,4) node [midway, above, sloped] (TextNode) {$e_4$};
\draw[white, line width=4pt,rounded corners=15pt] (-9,4) --  (-7.5,3) -- (-4,2) -- (-3,1);
\draw[gray, thick,decoration={markings,mark=at position 0.7  with {\arrow{<}}},postaction={decorate},rounded corners=15pt] (-9,4) --  (-7.5,3) -- (-4,2) -- (-3,1) node [midway, above, sloped] (TextNode) {$e_3$};

\filldraw[black] (-9,4) circle (4pt) node[anchor=east] {1};
\filldraw[black] (-9,1) circle (4pt) node[anchor=east] {2};
\filldraw[black] (-4.5,0) circle (2pt) node[anchor=west] {};

\fill[white] (-7.5,5) circle (2pt) node[anchor=east] {};
\draw[black] (-7.5,5) circle (2pt) node[anchor=west] {};
\fill[white] (-3,4) circle (4pt) node[anchor=east] {};
\draw[black] (-3,4) circle (4pt) node[anchor=west] {2};
\fill[white] (-3,1) circle (4pt) node[anchor=east] {};
\draw[black] (-3,1) circle (4pt) node[anchor=west] {1};

\end{tikzpicture}
\caption{Example of a collection of maps $M \in P_{(2,1),(2,1)}^{(3,1)}$ and an example of a bijection $f_M^\mu$ between the edges in $M$ and the edges of the graph $G_{(3,1)}$.} 
\label{fig4}
\end{figure}

\begin{prop}
\label{candidate}
For any Young diagram $\lambda \in \mathbb{Y}$, the following equality holds:
\begin{equation}
\A{\topp} \Ch_\pi \left( \lambda \right) \cdot \A{\topp} \Ch_\sigma \left( \lambda \right) =
\sum_\mu p^\mu_{\pi,\sigma} \A{\topp} \Ch_\mu \left( \lambda \right) .
\end{equation}
\end{prop}

\begin{proof}
According to \cref{AtopCh}, the two quantities
\begin{equation*}
\A{\topp} \Ch_\sigma \left( \lambda \right) = {N_{G_\sigma}\left( \lambda \right)}
\quad\quad\text{and}\quad\quad 
\A{\topp} \Ch_\mu \left( \lambda \right) ={N_{G_\mu}\left( \lambda \right)}
\end{equation*} 
\noindent
can be represented equivalently by the number of injective embeddings of $G_\sigma $ and $G_\mu$ into $\lambda$. 
Similarly, 
\begin{equation*}
\A{\topp} \Ch_\pi \left( \lambda \right) = \overline{N_{\overline{G_\pi}}\left( \lambda \right)}
\end{equation*}
\noindent 
is equal to the number of negative injective embeddings of $\overline{G_\pi}$ into $\lambda$ (see \cref{rem1}).

For each $M \in P_{\pi, \sigma}^\mu$ we choose some bijection $f_M^\mu$ between the edges of $M$ and the edges of the graph $G_\mu$ (see \cref{graph}), which preserves the connected components, see \cref{fig4}.

We shall construct a bijection between:
\begin{itemize}
\item a pair $ \Big(\overline{N_{\overline{G_\pi}}\left( \lambda \right)} , {N_{G_\sigma}\left( \lambda \right)} \Big) $ consisting of negative injective embeddings and injective embeddings of $G_\sigma$ and $\overline{G_\pi}$ into $\lambda$ respectively;
\item a pair $ \Big( P^\mu_{\pi,\sigma}, {N_{G_\mu}\left( \lambda \right)}  \Big) $ consisting of collections of maps from the class $  P^\mu_{\pi,\sigma}$ and injective embeddings of $G_\mu$ into $\lambda$.
\end{itemize} 
\noindent
Construction of such bijection follows the statement of \cref{candidate}. We proceed analogously as in the ``hands-shaking procedure'' described in \cref{procedure11}.

For each $i \in [n]$ we assign a white vertex with $\pi_i$ outgoing half-edges. We label this vertex by a number $i$ and root it, \ie we choose one of outgoing half-edges and label it. 
We can choose a bijection between such half-edges and the edges in $\overline{G_\pi}$ which preserves the connected components. 
Similarly, for each $j\in [l]$ we assign a black vertex with $\sigma_j$ outgoing half-edges and we root it. 
Then we choose a bijection between such half-edges and the edges in ${G_\sigma}$ which preserves the connected components. 
\medskip

A reverse injective embedding of $\overline{G_\pi}$ and an injective embedding of $G_\sigma$ into $\lambda$ transfer into an injective embedding of above described half-edges going out from labelled and rooted black and white vertices.

We use the procedure described in \cref{procedure11} to connect \emph{in the unique way} those outgoing half-edges which are embedded in the same box of Young diagram $\lambda$. 
We close each of non-closed half-edges by a white or a black vertex so that the graph remains bipartite. 

In that way we obtain a list of maps $M \in  P^\mu_{\pi,\sigma}$ injectively embedded into the Young diagram $\lambda$. 
Observe that all edges from any given connected component of $M$ are embedded into the boxes of $\lambda$ which are in the same row. 
Using the bijection $f_M^\mu$ between the edges of $M$ and the edges of $G_\mu$, we obtain the injective embedding of $G_\mu$ into $\lambda$.
\medskip

The above procedure is reversible. 
Indeed, for a given collection of maps $M \in  P^\mu_{\pi,\sigma}$ and an injective embedding of $G_\mu$ into diagram $\lambda$, we can easily construct the injective embedding of the edges of $M$ into the diagram $\lambda$, for which all edges from any given connected component of $M$ are embedded to the boxes from the same row. From such an object we can recover the elements from ${N_{G_\sigma}\left( \lambda \right)} $ and $\overline{N_{\overline{G_\pi}}\left( \lambda \right)}$.
\end{proof}

With \cref{candidate} in hand, we are ready to present the proof of \cref{main result}.

\begin{proof}[Proof of \cref{main result}]
The upper bound of a degree for polynomials $\g{\mu} (\delta )$ is given in \cref{estimation}. 
Since $\delta =\dfrac{1}{A} -A$, we have the following estimation
\begin{equation*}
\deg_A \g{\mu} = \deg_\delta \g{\mu} \leq d\left( \pi, \sigma ;\mu\right) .
\end{equation*}

Let us fix a Young diagram $\lambda$.
Recall that the evaluation of $\Ch_\pi $ on any Young diagram $\lambda$ is a Laurent polynomial in $\Laurent$ of a degree at most $n_2 (\pi) :=|\pi |-\ell (\pi)$. We investigate the ${n_2 (\pi) +n_2 (\sigma)}$ degree part of the pointwise product of two Jack characters, namely
\begin{equation*}
\A{n_2 (\pi) +n_2 (\sigma)} \Ch_\pi \left( \lambda \right) \cdot \Ch_\sigma  \left( \lambda \right) =
\A{n_2 (\pi) +n_2 (\sigma)} \sum_\mu \g{\mu}  \Ch_\mu   \left( \lambda \right).
\end{equation*} 
\noindent
By the estimations on the upper bounds of the $A$-degrees of Laurent polynomials $ \Ch_\pi (\lambda )$ and $\g{\mu}$ we have
\begin{equation*}
\A{\topp} \Ch_\pi \left( \lambda \right) \cdot \A{\topp} \Ch_\sigma  \left( \lambda \right) =\sum_\mu \A{d(\pi, \sigma ;\mu )} \g{\mu}  \A{\topp} \Ch_\mu   \left( \lambda \right).
\end{equation*} 

We compare the above equation with \cref{candidate} and we get
\begin{equation*}
\sum_\mu \A{d(\pi, \sigma ;\mu )}g^\mu_{\pi,\sigma} \Ch_\mu\left(\lambda\right) =
\sum_\mu p^\mu_{\pi,\sigma}  \Ch_\mu\left(\lambda\right) .
\end{equation*}
\noindent
Recall that $\Ch_\mu (\lambda ) = \widehat{p^{}_\mu} (\lambda )$. We have
\begin{equation}
\label{basis}
\sum_\mu \A{d(\pi, \sigma ;\mu )}g^\mu_{\pi,\sigma} \widehat{p^{}_\mu}\left(\lambda\right) =
\sum_\mu p^\mu_{\pi,\sigma}  \widehat{p^{}_\mu}\left(\lambda\right) ,
\end{equation}

The function $\widehat{p^{}_\mu}\left(\lambda\right) $ is symmetric and its homogeneous top-degree part coincides with  the power-sum symmetric polynomial $p^{}_\mu$. 
This coincidence together with the fact that power-sum symmetric functions form a basis of symmetric functions allows us to deduce that functions $\widehat{p^{}_\mu}\left(\lambda\right) $ form also such a basis. 
We may look at \cref{basis} as on the equality of symmetric functions. Since the basis determines its coefficients in the unique way, we conclude that
\begin{equation*}
\A{d(\pi, \sigma ;\mu )} \g{\mu} = p^\mu_{\pi,\sigma} .
\end{equation*}
\noindent
The $d(\pi, \sigma ;\mu )$-degree coefficients in variable $A$ and $\delta$ of $\g{\mu}$ are equal. We conclude
\begin{equation*}
\delta^{(\pi, \sigma ;\mu )} \g{\mu} = p^\mu_{\pi,\sigma} .
\end{equation*}
\noindent
\cref{strange5} and \cref{strange6} finish the proof.
\end{proof}

\begin{appendices}
\section{Top-degree parts in the Matchings-Jack Conjecture and the $b$-Conjecture}
\label{appendix}
We shall prove that our result about the top-degree part in the Matchings-Jack Conjecture presented in \cref{C-main result} and the result of Do\l{}\oldk{e}ga \cite[Theorem 1.5]{Dol17} about the top-degree part in $b$-Conjecture are equivalent.

First note that the polynomials $c^\lambda_{\pi,\sigma}$ and $h^\lambda_{\pi,\sigma}$ are related as follows
\begin{multline}
\label{star3} 
  \sum_{n\geq 1} t^n \sum_{\lambda,\pi,\sigma\vdash n}
h_{\pi, \sigma}^{\lambda}
p_\pi(\mathbf{x})
p_\sigma(\mathbf{y})
p_\lambda(\mathbf{z})=\\
\alpha t\dfrac{\partial}{\partial t} \log 
   \left(\sum_{n\geq 1} t^n \sum_{\lambda,\pi,\sigma\vdash n}
\frac{c_{\pi, \sigma}^{\lambda}}{\alpha^{\ell(\lambda)} z_\lambda}
p_\pi(\mathbf{x})
p_\sigma(\mathbf{y})
p_\lambda(\mathbf{z})\right) 
\end{multline}
and
\begin{multline}
\label{star4} 
 \sum_{n\geq 1} t^n \sum_{\lambda,\pi,\sigma\vdash n}
\frac{c_{\pi, \sigma}^{\lambda}}{\alpha^{\ell(\lambda)} z_\lambda}
p_\pi(\mathbf{x})
p_\sigma(\mathbf{y})
p_\lambda(\mathbf{z}) = \\
\exp   
   \left(\sum_{n\geq 1} \dfrac{1}{\alpha n}  t^n \sum_{\lambda,\pi,\sigma\vdash n}
h_{\pi, \sigma}^{\lambda}
p_\pi(\mathbf{x})
p_\sigma(\mathbf{y})
 p_\lambda(\mathbf{z})   \right) ,
\end{multline}
see \cref{star} and \cref{star2}.

In \cref{C-main result} we showed that the leading coefficient of $c^\lambda_{\pi,\sigma}$ can be expressed in the following way:
\begin{equation*}
\left[ \beta^{d (\pi, \sigma ;\lambda )}\right] c^\lambda_{\pi ,\sigma}= 
\Big\vert M \in \Mll : M \text{ is unhandled }  \Big\vert
\end{equation*} 
\noindent
where $\Mll$ is the set of $\lambda$-\emph{lists of unicellular maps} with the white and black vertices distribution given by $\pi$ and $\sigma$ respectively.

On the other hand, Do\l{}\oldk{e}ga \cite[Theorem 1.5]{Dol17} showed that the leading coefficient of $h^{(n)}_{\pi,\sigma}$ can be expressed in the following way:
\begin{equation*}
\left[ \beta^{d (\pi, \sigma ;(n) )}\right] h^{(n)}_{\pi ,\sigma}= 
\Big\vert M \in M^{(n)}_{\pi ,\sigma} : M \text{ is unhandled }  \Big\vert
\end{equation*} 
\noindent
where $M^{(n)}_{\pi ,\sigma}$ is the set of \emph{unicellular maps} with the white and black vertices distribution given by $\pi$ and $\sigma$ respectively.

\begin{rem}
\label{rozbicie}
Observe that multiplication of power-sum symmetric functions expresses as follows
\[p_{\lambda_1}(\mathbf{z}) \cdot p_{\lambda_1}(\mathbf{z}) =  p_{\lambda_1 \cup \lambda_1}(\mathbf{z})\]
\noindent
in the terms of concatenations of relevant partitions.
\end{rem}

We investigate the $\left[ p_{(n)}(\mathbf{z})\right]  $ coefficient in both sides of \cref{star3}. We have
\begin{multline*}
 t^n \sum_{\pi,\sigma\vdash n}
h_{\pi, \sigma}^{(n)}
p_\pi(\mathbf{x})
p_\sigma(\mathbf{y})=
\alpha t\dfrac{\partial}{\partial t}  \left( 
   t^n \sum_{\pi,\sigma\vdash n}
\frac{c_{\pi, \sigma}^{(n)}}{\alpha z_{(n)}}
p_\pi(\mathbf{x})
p_\sigma(\mathbf{y})\right)  =\\
\alpha n  
   t^n \sum_{\pi,\sigma\vdash n}
\frac{c_{\pi, \sigma}^{(n)}}{\alpha n}
p_\pi(\mathbf{x})
p_\sigma(\mathbf{y})
\end{multline*}
\noindent
hence $c_{\pi, \sigma}^{(n)}$ and $ h_{\pi, \sigma}^{(n)} $ are equal.

Since $c_{\pi, \sigma}^{(n)} = h_{\pi, \sigma}^{(n)} $, it might seem that our result extends the result of Do\l{}\oldk{e}ga. However, a more subtle analysis of relationships between the coefficients of $c_{\pi,\sigma}^\lambda$ and $h_{\pi,\sigma}^\lambda$ shows that both results are equivalent. 

The power series expansion of the exponent function in \cref{star4} gives us
\begin{multline}
\label{star5}
\sum_{n \geq 1}
  t^n \sum_{\lambda,\pi,\sigma\vdash n}
\frac{c_{\pi, \sigma}^{\lambda}}{\alpha^{\ell(\lambda)} z_\lambda}
p_\pi(\mathbf{x})
p_\sigma(\mathbf{y}) 
p_\lambda(\mathbf{z}) = \\
\sum_{k \geq 0}\dfrac{1}{k !}
   \left(\sum_{s\geq 1} \dfrac{1}{s}t^s \sum_{\lambda,\pi,\sigma\vdash s}
\dfrac{h_{\pi, \sigma}^{\lambda}}{\alpha}
p_\pi(\mathbf{x})
p_\sigma(\mathbf{y})
 p_\lambda(\mathbf{z})   \right)^k .
\end{multline}
\noindent
We denote by $\mathcal{P}_k^{\lambda, \pi,\sigma}$ the set of triplets of lists of partitions 
\[ \Big( \left( \lambda^1 ,\ldots , \lambda^k \right) , \left( \pi^1 ,\ldots , \pi^k \right) ,\left( \sigma^1 ,\ldots , \sigma^k \right) \Big)\]
such that
\begin{equation*}
\bigcup_{i=1}^k \lambda^i = \lambda ,\quad\quad 
\bigcup_{i=1}^k \mu^i = \mu ,\quad\quad
\bigcup_{i=1}^k \sigma^i = \sigma
\end{equation*}
\noindent
and for each $i$ we have $|\lambda^i | =|\pi^i |= |\sigma^i |$. 

Let us investigate the $\left[ p_\pi(\mathbf{x}) p_\sigma(\mathbf{y})p_{\lambda}(\mathbf{z})\right]  $ coefficient in both sides of \cref{star5}. We have
\begin{equation}
\label{star6}
  t^n 
\frac{c_{\pi, \sigma}^{\lambda}}{\alpha^{\ell(\lambda)} z_\lambda} = \\
t^n \sum_{k : 1 \leq k\leq \ell (\lambda ) }\dfrac{1}{k !}
 \sum_{\subalign{\big(&( \lambda^1 ,\ldots , \lambda^k ) , \\&( \pi^1 ,\ldots , \pi^k ) ,\\ &( \sigma^1 ,\ldots , \sigma^k ) \big) \in \mathcal{P}_k^{\lambda, \pi,\sigma}}}
   \prod_{i=1}^k \dfrac{1}{|\lambda^i |} 
\dfrac{h_{\pi^i, \sigma^i}^{\lambda^i}}{\alpha}.
\end{equation}
\noindent
Do\l{}\oldk{e}ga and F\'eray \cite[Theorem 1.2]{DolegaFeray2016} gave the following bound on the degree 
\[\deg h_{\pi^i, \sigma^i}^{\lambda^i} \leq 
 |\lambda^i | +2 -\ell (\lambda^1 ) -\ell (\pi^i ) -\ell (\sigma^i ).
\]
Hence, each summand of the first sum on the right-hand side of \cref{star6} has degree equal to at most
\[ n+k -\ell (\lambda ) -\ell (\pi ) -\ell (\sigma) ,
\]
and the maximal bound may be achieved only for summands corresponding to $k=\ell (\lambda)$. For such a summand, its bound on the degree is the same as the bound on the degree for the left-hand side of \cref{star6} given by \cref{Max}. 
We have
\begin{equation*}
\frac{1}{ z_\lambda} \left[ \alpha^{d(\pi ,\sigma;\lambda)} \right] c_{\pi, \sigma}^{\lambda} = \\
\dfrac{1}{\ell (\lambda ) !} 
\sum_{\subalign{\big(&( \lambda^1 ,\ldots , \lambda^{\ell (\lambda )}  ) , \\&( \pi^1 ,\ldots , \pi^{\ell (\lambda )}  ) ,\\ &( \sigma^1 ,\ldots , \sigma^{\ell (\lambda )}  ) \big) \in \mathcal{P}_{\ell (\lambda )}^{\lambda, \pi,\sigma}}}
   \prod_{i=1}^{\ell (\lambda )} \dfrac{1}{|\lambda^i |} 
\left[ \alpha^{ |\lambda_i| +1 -\ell (\pi ) -\ell (\sigma)} \right]  h_{\pi^i, \sigma^i}^{\lambda_i } .
\end{equation*}

For a Young diagram $\lambda = (\lambda_1 ,\ldots ,\lambda_k )$, denote by $\mathcal{C}_\lambda$ the set of all \emph{compositions} of a type $\lambda$, \ie the set of all lists $(\lambda_{\sigma (1)} ,\ldots ,\lambda_{\sigma (k)} )$, for some $\sigma \in \Sym{n}$. 
Observe that 
\begin{equation*}
| \mathcal{C}_\lambda | =\dfrac{\ell (\lambda) !}{\sum_i m_i (\lambda ) !} .
\end{equation*} 
\noindent
Observe that for $k =\ell (\lambda )$ the first list in any triplet from $\mathcal{P}_k^{\lambda, \pi,\sigma}$ is a composition of a type of the Young diagram $\lambda$. 
We have
\begin{equation*}
\frac{1}{ z_\lambda} \left[ \alpha^{d(\pi ,\sigma;\lambda)} \right] c_{\pi, \sigma}^{\lambda} = \\
\dfrac{1}{\ell (\lambda ) !} 
| \mathcal{C}_\lambda |
\sum_{\subalign{\Big(&\big( (\lambda_1) ,\ldots , (\lambda_{\ell (\lambda )} )  \big) , \\&\big( \pi^1 ,\ldots , \pi^{\ell (\lambda )}  \big) ,\\ &\big( \sigma^1 ,\ldots , \sigma^{\ell (\lambda )}  \big) \Big) \in \mathcal{P}_{\ell (\lambda )}^{\lambda, \pi,\sigma}}}
   \prod_{i=1}^{\ell (\lambda )} \dfrac{1}{\lambda_i} 
\left[ \alpha^{\lambda_i +1 -\ell (\pi ) -\ell (\sigma)} \right]  h_{\pi^i, \sigma^i}^{(\lambda_i )} 
\end{equation*}
\noindent
and hence
\begin{equation}
\label{star7}
\left[ \alpha^{d(\pi ,\sigma;\lambda)} \right] c_{\pi, \sigma}^{\lambda} = \\
\sum_{\subalign{\Big(&\big( (\lambda_1) ,\ldots , (\lambda_{\ell (\lambda )} )  \big) , \\&\big( \pi^1 ,\ldots , \pi^{\ell (\lambda )}  \big) ,\\ &\big( \sigma^1 ,\ldots , \sigma^{\ell (\lambda )}  \big) \Big) \in \mathcal{P}_{\ell (\lambda )}^{\lambda, \pi,\sigma}}}
   \prod_{i=1}^{\ell (\lambda )} 
\left[ \alpha^{\lambda_i +1 -\ell (\pi ) -\ell (\sigma)} \right]  h_{\pi^i, \sigma^i}^{ (\lambda_i )} 
\end{equation}

Do\l{}\oldk{e}ga's result \cite[Theorem 1.5]{Dol17} shows us that 
\[ \left[ \alpha^{\lambda_i +1 -\ell (\pi ) -\ell (\sigma)} \right] h_{\pi^i, \sigma^i}^{ (\lambda_i )} =
 \Big\vert M\in M_{\pi^i, \sigma^i}^{ (\lambda_i )} : M \text{ is unhandled}\Big\vert .\] 
Directly  from the definition of $ \mathcal{P}_{\ell (\lambda )}^{\lambda, \pi,\sigma}$ we obtain that 
\[\left[ \alpha^{d(\pi ,\sigma;\lambda)} \right] c_{\pi, \sigma}^{\lambda} = 
\Big\vert M\in \Mll : M \text{ is unhandled}\Big\vert ,\]
\noindent
which allows us to conclude the equivalence of both results.
\end{appendices}

\let\k\oldk
\let\c\oldc
\bibliographystyle{alpha}
\bibliography{Jack_bib}

\end{document}